\def\bu{\bullet}
\def\marker{\>\hbox{${\vcenter{\vbox{
    \hrule height 0.4pt\hbox{\vrule width 0.4pt height 6pt
    \kern6pt\vrule width 0.4pt}\hrule height 0.4pt}}}$}\>}
\def\gpic#1{#1
     \smallskip\par\noindent{\centerline{\box\graph}} \medskip}
\documentclass[12pt]{article}
\usepackage{amsmath,amssymb,amsfonts,euscript,graphicx,amsthm}
\usepackage{fullpage}

\newtheorem{thm}{Theorem}[section]
\newtheorem{lem}[thm]{Lemma}

\newtheorem{conj}[thm]{Conjecture}
\theoremstyle{definition}
\newtheorem{definition}[thm]{Definition}
\newtheorem{remark}[thm]{Remark}
\newtheorem{example}[thm]{Example}
\newcommand{\comment}[1]{}    

\def\Mad{{\rm Mad}\,}
\def\esub{\subseteq}

\def\st{\colon\,}
\def\FR#1#2{\frac{#1}{#2}}
\def\SM#1#2{\sum_{#1\in #2}}
\def\FL#1{\left\lfloor{#1}\right\rfloor}
\def\CL#1{\left\lceil{#1}\right\rceil}

\def\pwp{\phi_{w'}}
\def\ppwp{\phi'_{w'}}
\def\rw#1#2{\rho_w(#1,#2)}
\def\rp#1#2{\rho_{w'}(#1,#2)}
\def\bA{\overline{{\bf A}}}
\def\bB{\overline{{\bf B}}}
\def\bC{\overline{{\bf C}}}
\def\bD{\overline{{\bf D}}}
\def\bE{\overline{{\bf E}}}
\def\bF{\overline{{\bf F}}}
\def\bG{\overline{{\bf G}}}
\def\bH{\overline{{\bf H}}}
\def\bI{\overline{{\bf I}}}
\def\bJ{\overline{{\bf J}}}
\def\bK{\overline{{\bf K}}}

\def\bZ{\overline{{\bf Z}}}
\def\nul{\varnothing}

\title{The $1,2,3$-Conjecture and $1,2$-Conjecture for Sparse Graphs}

\author{
Daniel W. Cranston\thanks{dcranston@vcu.edu; Virginia Commonwealth
University, Richmond, VA.}\,,
Sogol Jahanbekam\thanks{jahanbe1@illinois.edu; 
University of Illinois, Urbana, IL.  Research supported in part by National
Science Foundation grant DMS 09-01276.}\,, and
Douglas B. West\thanks{west@math.uiuc.edu; Zhejiang Normal University, Jinhua,
China, and University of Illinois, Urbana, IL.  Research supported in part by
National Security Agency grant H98230-10-1-0363.}
}

\date{\today}

\begin{document}
\maketitle

\vspace{-2pc}

\begin{abstract}
We apply the Discharging Method to prove the $1,2,3$-Conjecture and the
$1,2$-Conjecture for graphs with maximum average degree less than $8/3$.
Stronger results on these conjectures have been proved, but this is the first
application of discharging to them, and the structure theorems and reducibility
results are of independent interest.

\smallskip
\noindent
{\it Keywords}:   $1,2,3$-Conjecture,  $1,2$-Conjecture, discharging method.\\
{\it Mathematics Subject Classification}: 05C15, 05C22,  05C78.
\end{abstract}

\baselineskip16pt

\section{Introduction}

Variations on coloring problems in graph theory have involved many ways of
generating vertex colorings.  Without restrictions, the minimum number of
distinct colors needed to label the vertices of $G$ so that adjacent vertices
have different colors is the \textit{chromatic number} $\chi(G)$.  We consider
restricted colorings produced from weights on the edges and vertices.

Let $V(G)$ and $E(G)$ denote the vertex set and edge set of a graph $G$, and
let $\Gamma_G(v)$ denote the set of edges incident to a vertex $v$.  An
\textit{$S$-weighting} of a graph $G$ is a map $w\st E(G)\to S$.  A
\textit{total $S$-weighting} is a map $w\st E(G)\cup V(G)\to S$.  More
specifically, a $k$-weighting is an $S$-weighting with $S=\{1,\ldots,k\}$.
For a weighting $w$, let $\phi_w(v)=\sum_{e\in \Gamma_G(v)}w(e)$.
For a total weighting $w$, let $\phi_w(v)=w(v)+\sum_{e\in \Gamma_G(v)}w(e)$;
that is, each vertex is assigned the total of the weights it ``sees''.  A
weighting or total weighting $w$ is \textit{proper} if $\phi_w$ is a proper
coloring of $G$.  We seek proper $k$-weightings or proper total $k$-weightings
for small $k$.

\begin{conj}[\rm $1,2,3$-Conjecture; Kar\'onski--{\L}uczak--Thomason \cite{KLT}]
\label{123conj}
Every graph without isolated edges has a proper $3$-weighting.
\end{conj}

\begin{conj}[\rm  $1,2$-Conjecture; Przyby{\l}o--Wo\'zniak \cite{pw}]
\label{12conj}
Every graph has a proper total $2$-weighting.
\end{conj}

Toward the $1,2,3$-Conjecture, the original paper proved it for graphs with
chromatic number at most $3$.  Addario-Berry et al.~\cite{admrt} showed that
every graph without isolated edges has a proper $k$-weighting when $k=30$.
After improvements to $k=15$ in \cite{adr} and $k=13$ in \cite{wy}, Kalkowski,
Kar\'onski, and Pfender \cite{kkp} showed that every graph without isolated
edges has a proper $5$-weighting.  Toward the $1,2$-Conjecture, Przyby\l o and
Wo\'zniak \cite{pw} proved it for complete graphs and for graphs with chromatic
number at most $3$.  Kalkowski~\cite{K} showed that every graph has a proper
total $3$-weighting; furthermore, there is such a weighting with the edge
weights in one spanning tree fixed arbitrarily and the vertex weights chosen
from $\{1,2\}$.  Seamone~\cite{S} surveyed progress on these and related
problems.

List versions of the conjectures have been proposed.  A graph is
$k$-\textit{weight-choosable} if whenever each edge is given a list of $k$
available integers, a proper weighting can be chosen from the lists.
A graph is $(k,k')$-\textit{weight-choosable} if whenever each vertex has a
list of size $k$ and each edge has a list of size $k'$, a proper total
weighting can be chosen from the lists.

\begin{conj}[Bartnicki--Grytczuk--Niwczyk \cite{bgn}]\label{bgn}
Every graph without isolated edges is $3$-weight-choosable.
\end{conj}

\begin{conj}[Wong--Zhu \cite{wz}]\label{wz}
Every graph is $(2,2)$-weight-choosable. Every graph without isolated edges is
$(1,3)$-weight-choosable.
\end{conj} 

These conjectures are stronger than the original conjectures, which concern
the special case where the lists consist of the smallest positive integers.

Wong, Yang, and Zhu~\cite{wyz} proved that the complete multipartite graph
$K_{n,m,1,1,...,1}$ is $(2, 2)$-weight-choosable and that complete bipartite
graphs other than $K_2$ are $(1, 2)$-weight-choosable.
Bartnicki, Grytczuk, and Niwczyk \cite{bgn} applied the Combinatorial
Nullstellensatz \cite{a} to prove Conjecture \ref{bgn}  for complete graphs,
complete bipartite graphs, and trees.  Wong and Zhu \cite{wz} applied the
Combinatorial Nullstellensatz to prove Conjecture \ref{wz} for complements of
linear forests; this includes complete graphs.  They also proved that every
tree with an even number of edges is $(1,2)$-weight-choosable.  Wong, Yang, and
Zhu~\cite{wyz} continued this approach, proving Conjecture~\ref{wz} for graphs
with maximum degree $3$.  Finally, Wong and Zhu~\cite{wz2} proved that every
graph is $(2,3)$-weight-choosable.

Our results use the ``Discharging Method'' and apply to sparse graphs.
Sparseness is imposed by bounding the \textit{maximum average degree} of $G$,
denoted $\Mad(G)$, which is the largest average degree among the subgraphs of
$G$: $\Mad(G)=\max_{H\esub G} \FR{2|E(H)|}{|V(H)|}$.  A consequence of our
results is that Conjectures~\ref{123conj} and \ref{12conj} hold for every graph
$G$ such that $\Mad(G)<8/3$.  However, that consequence is already known,
since every subgraph of a graph $G$ with $\Mad(G)<8/3$ has a vertex with degree
at most $2$, and therefore by induction is $3$-colorable.  The original papers
proved the conjectures for $3$-colorable graphs.

The novelty of our results is thus in the structure theorems proved by
discharging (which may be of use in solving other problems) and in the 
reducibility theorems showing that various configurations cannot occur in
minimal counterexamples to the conjecture.

We note that proofs via the Combinatorial Nullstellensatz are nonconstructive,
in that the parameter space to be searched for the guaranteed weighting is
exponentially large.  Inductive proofs via discharging, such as ours, yield
polynomial-time algorithms to produce the solution, though implementation may
be complicated.  The original proofs of these conjectures for $3$-colorable
graphs are also constructive.

\medskip
Because our proofs for $\Mad(G)<8/3$ are fairly long, we first present in
Section~\ref{sec52} short proofs of the weaker results that the claims hold
when $\Mad(G)<5/2$.  The reducibility arguments in these proofs are used in
the stronger results.

To discuss both problems together, let {\it $j$-weighting} mean a $3$-weighting
when $j=3$ and a total $2$-weighting when $j=2$.  A graph is {\it $j$-bad} if
it has no proper $j$-weighting (and no isolated edge if $j=3$).  Configurations
forbidden from minimal $j$-bad graphs are {\it $j$-reducible}.

Our proofs of $j$-reducibility use the restriction of weights to values at most
$j$, so they do not extend to the list versions.  Also, unlike in most coloring
problems, vertices of degree $1$ do not immediately yield reducible
configurations, since the weight on a pendant edge affects whether its incident
edges are properly colored.  

In Section~\ref{sec52} we first obtain some $3$-reducible configurations.
We next use discharging to show that every graph with average degree less than
$5/2$ contains a $3$-reducible configuration.  We assign each vertex initial
charge equal to its degree and then shift charge so that if no specified
configuration occurs, then every vertex has final charge at least $c$.
Since $\Mad(G)<c$ and $G'\esub G$ imply $\Mad(G')<c$ (by definition), this
proves that $G$ has a proper $3$-weighting when $\Mad(G)<5/2$.
In Section~\ref{sec52} we also use this method to prove that $G$
has a proper (total) $2$-weighting when $\Mad(G)<5/2$.  Both results use the
same discharging argument, although the sets of reducible configurations are
different.

When $\Mad(G)\ge 5/2$, no longer must $G$ have a configuration among those in
Section~\ref{sec52}.  We need additional reducible configurations to complete
an unavoidable set when $\Mad(G)<8/3$.  In Section~\ref{sec12} and
Section~\ref{sec123}, respectively, we complete the proofs of the
$1,2$-Conjecture and the $1,2,3$-Conjecture for graphs $G$ such that
$\Mad(G)<8/3$.

\section{Reducible Configurations and $\Mad(G)<5/2$}\label{sec52}

In discharging arguments for sparse graphs, it is convenient to have concise
terminology for vertices satisfying degree constraints.

\begin{definition}\label{basic}
A vertex with degree $k$, at least $k$, or at most $k$ is a
\textit{$k$-vertex}, a \textit{$k^+$-vertex}, or a \textit{$k^-$-vertex},
respectively.  A \textit{$j$-neighbor} of $v$ is a $j$-vertex that is a
neighbor of $v$.

Write $N_G(v)$ for the neighborhood of $v$ in $G$ and $d_G(v)$ for its degree.
For $v\in V(G)$ and $U\esub N_G(v)$, let $[v,U]$ denote the set of edges
joining $v$ to $U$.

A weighting or total weighting $w$ \textit{satisfies} an edge $uv$ if
$\phi_w(u)\ne\phi_w(v)$, or equivalently if $\rw{u}{v}\ne\rw{v}{u}$, where we
define $\rw{x}{y}=\phi_w(x)-w(xy)$ when $x$ and $y$ are adjacent.

A {\em configuration} in a graph $G$ is a subgraph $C$ together with specified
degrees in $G$ for $V(C)$.  The {\em core} of the configuration is $E(C)$, and
the resulting {\em derived graph} is $G-E(C)$.
\end{definition}

Our simplest configurations consist of a vertex with specified degree and the
edges from it to certain neighbors of specified degrees.
We begin with a lemma that reduces the length of reducibility proofs:
$1$-neighbors are ``easier'' to handle than $2$-neighbors, so when we claim
that a configuration is reducible when a particular vertex has degree $1$ or
$2$, in the proof we may assume that it has degree $2$.

\begin{lem}\label{assume2}
If a vertex $z$ in a $2$- or $3$-reducible configuration $C$ has degree $1$ in
$C$ and is specified as a $2$-vertex of the full graph, then the configuration
$C'$ obtained from $C$ by instead specifying $z$ as a $1$-vertex (with 
neighbor $v$) is also reducible.
\end{lem}
\begin{proof}
Let $H$ be a graph containing $C'$, and let $H'$ be the derived graph; $z$ is
isolated in $H'$.  Form $G$ by adding vertices $a$ and $b$ and edges $ab$ and
$bz$ to $H$.  Now $C$ arises in $G$, and its derived graph $G'$ arises from
$H'$ in the same way that $G$ arises from $H$.

If $H$ is a minimal $j$-bad graph, then $H'$ has a desired weighting.  Since
also the path $P_3$ has such a weighting, $G'$ has such a weighting.  Since $C$
is $j$-reducible, $G$ therefore also has such a weighting.  To
obtain the desired weighting of $H$, note that all edges are satisfied when $a$
and $b$ are deleted from the weighting of $G$, except possibly $zv$.

For $j=2$, the weight on $z$ is needed only to satisfy $zv$ in $H$ and can be
respecified so that $zv$ is satisfied.  For $j=3$, the edge $zv$ is satisfied
automatically since $d_H(v)>1$.
\end{proof} 

Reducibility proofs may use some types of inferences many times.  The next
lemma enables us to express statements concisely and reduce repetitive language.
It can be stated in more generality, but for clarity we list just typical
situations in which we will use it.

\begin{lem}\label{choose}
Let $w$ be a partial $j$-weighting of a graph $G$ ($w$ is not specified
everywhere).  In the situations below, the weights on a set $S$ can be chosen
to satisfy the edges in a set $F$ if the weights on all the edges (or
vertices) incident to $F$ and not in $S$ are already known:

1) The edges of $F$ have a common endpoint $v$, incident to all edges of $S$
(possibly also $v\in S$ when $j=2$), and $|F|\le(j-1)|S|$.

2) $F$ consists of two edges, with $S$ a single edge joining them and $j=3$.
\end{lem}

\begin{proof}
Let $k=|S|$.  Since weights are chosen from $\{1,\ldots,j\}$, the sum
of $k$ weights has $1+(j-1)k$ possible values.  Each edge in $F$ uses that sum
in determining whether the values of $\phi$ differ at its endpoints.  Each edge
in $F$ thus forbids at most one value of the sum in a proper $j$-weighting.
There are at least $k(j-1)$ possible augmentations above the least value of the
sum, so when $k(j-1)\ge|F|$ the labels can be chosen to satisfy all of $F$.

Note that in (2) the weights on $F$ may be unspecified; the weight on an edge
does not affect whether it is satisfied.  Similarly, if $F=\{uv\}$, and $S$ is
a single edge incident to $v$ or is $v$ itself, and the weights of all other
items incident to $uv$ are known, then the weight on $S$ can be chosen in 
$j-1$ ways to satisfy $F$.
\end{proof}

\begin{remark}\label{plan}
We use Lemma~\ref{choose} frequently in reducibility arguments, invoked without
mention in $2$-reducibility when we write ``choose $w(vz)$ to statisfy $vx$''
or ``choose $w(v)$ and $w(vz)$ to satisfy $vx$ and $vv'$''.  In
$3$-reducibility, a choice can satisfy more.  In $2$-reducibility we can choose
one weight to avoid one value, but in $3$-reducibility it can avoid two values.

Another method of satisfying an edge $uv$ is to create sufficient imbalance
between the contributions at $u$ and $v$ to guarantee that $\phi(u)\ne\phi(v)$
when the weighting is completed.  When we write ``Set $w(uv)=3$ to ensure
satisfying $vz$'', we mean that no way of choosing weights on the remaining
edges can produce $\phi(w)=\phi(v)$.  Saying that an edge is ``automatically
satisfied'' has a similar meaning.  For example, any edge joining a $1$-vertex
to a $3$-vertex is automatically satisfied for (total) $2$-weightings, while
putting weight $1$ at the $1$-vertex ensures satisfying the edge even when the 
neighbor has degree $2$.

The figures for configurations show the core in bold; the derived graph $G'$ is
obtained by deleting the core.  Also, with $w'$ assumed to be a proper
$j$-weighting of $G'$, the label on an edge $e$ is $w'(e)$, and the label in a
circle at a vertex $x$ with one neighbor $u$ is $\rp{x}{u}$.  To satisfy $xu$,
the sum of the contributions at $u$ other than $w'(xu)$ must differ from
$\rp{x}{u}$.

The figures do not show cases where some of the specified vertices may be equal.
For instances where such equalities do not affect the validity of the written
argument, we make no comment about possible changes in the illustration.
\end{remark}

The task of proving reducibility for a configuration $C$ is the task of
modifying or extending a proper $j$-weighting $w'$ of the derived graph $G'$
to obtain a $j$-weighting $w$ of $G$ such that the edges in or incident to the
core of $C$ become satisfied, while the other edges of $G'$ remain satisfied.
If we do not change the weights on edges of $G'$ incident to the core, then
we do not change the fact that all edges of $G'$ not incident to the core
are satisfied.

With these preparations, we begin the reducibility arguments.  The first lemma
eliminates many degenerate cases of later configurations in which specified
vertices may be identical.

\begin{lem}\label{triangle}
The following configurations are both $2$-reducible and $3$-reducible:

(1) A $3$-cycle through two $2$-vertices and one $4^-$-vertex.

(2) A $3$-cycle through one $2$-vertex $z$ and two vertices that each may be
a $3$-vertex, a $4$-vertex with a $1$-neighbor, or a $5$-vertex with a
$1$-neighbor.  In addition, one neighbor of $z$ is allowed to be a $4$-vertex
with a $2$-neighbor other than $z$ (and no $1$-neighbor).
\end{lem}
\begin{proof}
When $G$ is a $3$-cycle, the weights can be chosen to produce colors
$\{3,4,5\}$ at the vertices, for either value of $j$.  Suppose $G\ne C_3$.
In each case, we extend a proper $j$-weighting $w'$ of a subgraph $G'$ obtain
by deleting the edges of the cycle.  The cases appear in Figure~\ref{figtri}.

For (1), let $v$ be the $3^+$-vertex on the cycle, and let $z$ and $z'$ be
the $2$-vertices.  To extend $w'$ to $w$, first set $w(zz')=1$, and also set
$w(v)=2$ if $j=2$.  This ensures satisfying $vz$ and $vz'$.  If $j=2$, then fix
$w(z)=1$, choose $w(vz)$ and $w(vz')$ to satisfy $\Gamma_{G'}(v)$, and choose
$w(z')$ to satisfy $zz'$.  If $j=3$, then require $w(vz)\ne w(vz')$ with
$w(vz)\in\{1,2\}$ and $w(vz')\in\{2,3\}$ to satisfy $zz'$.  There are three
choices for $w(vz)+w(vz')$, so we can choose them also to satisfy
$\Gamma_{G'}(v)$, since $d_{G'}(v)\le2$.

For (2), let $v$ and $v'$ be the other vertices of the triangle.
If $d_G(v)\ge4$, then let $u$ be a vertex of smallest degree in $N(v)-\{z\}$;
similarly define $u'\in N(v')$.  Form $G'$ from $G$ by deleting
$\{vz,v'z,vv'\}$ and the edges $vu$ and $v'u'$ (if they exist).
Figure~\ref{figtri} shows one of the possibilities at each of $v$ and $v'$.

We first ensure that $vz$ and $v'z$ will be satisfied by setting $w(vv')=j$
(and $w(z)=1$ if $j=2$); this will yield $\rw zv\le 3<4\le \rw vz$, since
$d_G(v)\ge3$.

For $d_G(v)=3$, choose $w(vz)$ (and $w(v)$ if $j=2$) to satisfy the one edge in
$\Gamma_{G'}(v)$.  For $d_G(v)\in\{4,5\}$ and $d_G(u)=1$, choose $w(vz)$ and
$w(vu)$ (and $w(v)$ if $j=2$) to satisfy $\Gamma_{G'}(v)$.  These cases have
extra flexibility, so that if all contributions to $\phi_w(v')$ are already
known, then $vv'$ can also be satisfied.

For $d_G(v)=4$ and $d_G(u)=2$, choose $w(vu)$ (and $w(u)$ if $j=2$) to satisfy
$\Gamma_{G'}(u)$, and then choose $w(vz)$ (and $w(v)$ if $j=2$) to satisfy
$vu$ and $\Gamma_{G'}(v)$.  In this case we do not satisfy $vv'$ using edges at
$v$.  Instead, we satisfy $vv'$ using one of the earlier cases at $v'$ after
$\phi_w(v)$ is known; this case is only allowed to occur at one of $\{v,v'\}$.
\end{proof}

\begin{figure}[h]
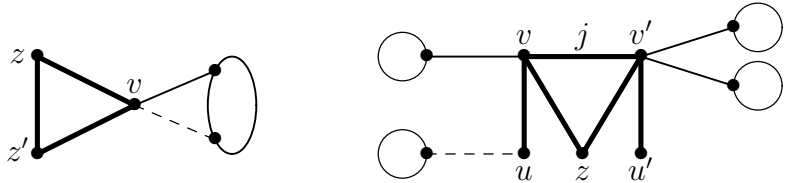

\gpic{
\expandafter\ifx\csname graph\endcsname\relax \csname newbox\endcsname\graph\fi
\expandafter\ifx\csname graphtemp\endcsname\relax \csname newdimen\endcsname\graphtemp\fi
\setbox\graph=\vtop{\vskip 0pt\hbox{%
    \special{pn 11}%
    \special{ar 1121 535 127 255 0 6.28319}%
    \graphtemp=.5ex\advance\graphtemp by 0.790in
    \rlap{\kern 0.102in\lower\graphtemp\hbox to 0pt{\hss $\bu$\hss}}%
    \graphtemp=.5ex\advance\graphtemp by 0.280in
    \rlap{\kern 0.102in\lower\graphtemp\hbox to 0pt{\hss $\bu$\hss}}%
    \graphtemp=.5ex\advance\graphtemp by 0.535in
    \rlap{\kern 0.611in\lower\graphtemp\hbox to 0pt{\hss $\bu$\hss}}%
    \graphtemp=.5ex\advance\graphtemp by 0.355in
    \rlap{\kern 1.031in\lower\graphtemp\hbox to 0pt{\hss $\bu$\hss}}%
    \graphtemp=.5ex\advance\graphtemp by 0.715in
    \rlap{\kern 1.031in\lower\graphtemp\hbox to 0pt{\hss $\bu$\hss}}%
    \special{pn 28}%
    \special{pa 611 535}%
    \special{pa 102 790}%
    \special{fp}%
    \special{pa 102 790}%
    \special{pa 102 280}%
    \special{fp}%
    \special{pa 102 280}%
    \special{pa 611 535}%
    \special{fp}%
    \special{pn 11}%
    \special{pa 611 535}%
    \special{pa 1031 355}%
    \special{fp}%
    \special{pn 8}%
    \special{pa 611 535}%
    \special{pa 1031 715}%
    \special{da 0.051}%
    \graphtemp=.5ex\advance\graphtemp by 0.790in
    \rlap{\kern 0.000in\lower\graphtemp\hbox to 0pt{\hss $z'$\hss}}%
    \graphtemp=.5ex\advance\graphtemp by 0.280in
    \rlap{\kern 0.000in\lower\graphtemp\hbox to 0pt{\hss $z$\hss}}%
    \graphtemp=.5ex\advance\graphtemp by 0.433in
    \rlap{\kern 0.611in\lower\graphtemp\hbox to 0pt{\hss $v$\hss}}%
    \graphtemp=.5ex\advance\graphtemp by 0.790in
    \rlap{\kern 2.140in\lower\graphtemp\hbox to 0pt{\hss $\bu$\hss}}%
    \graphtemp=.5ex\advance\graphtemp by 0.280in
    \rlap{\kern 2.140in\lower\graphtemp\hbox to 0pt{\hss $\bu$\hss}}%
    \graphtemp=.5ex\advance\graphtemp by 0.790in
    \rlap{\kern 2.650in\lower\graphtemp\hbox to 0pt{\hss $\bu$\hss}}%
    \graphtemp=.5ex\advance\graphtemp by 0.280in
    \rlap{\kern 2.650in\lower\graphtemp\hbox to 0pt{\hss $\bu$\hss}}%
    \graphtemp=.5ex\advance\graphtemp by 0.790in
    \rlap{\kern 2.955in\lower\graphtemp\hbox to 0pt{\hss $\bu$\hss}}%
    \graphtemp=.5ex\advance\graphtemp by 0.790in
    \rlap{\kern 3.261in\lower\graphtemp\hbox to 0pt{\hss $\bu$\hss}}%
    \graphtemp=.5ex\advance\graphtemp by 0.280in
    \rlap{\kern 3.261in\lower\graphtemp\hbox to 0pt{\hss $\bu$\hss}}%
    \graphtemp=.5ex\advance\graphtemp by 0.433in
    \rlap{\kern 3.745in\lower\graphtemp\hbox to 0pt{\hss $\bu$\hss}}%
    \graphtemp=.5ex\advance\graphtemp by 0.127in
    \rlap{\kern 3.745in\lower\graphtemp\hbox to 0pt{\hss $\bu$\hss}}%
    \special{pn 28}%
    \special{pa 3261 280}%
    \special{pa 2650 280}%
    \special{fp}%
    \special{pa 2650 280}%
    \special{pa 2955 790}%
    \special{fp}%
    \special{pa 2955 790}%
    \special{pa 3261 280}%
    \special{fp}%
    \special{pa 2650 790}%
    \special{pa 2650 280}%
    \special{fp}%
    \special{pa 3261 790}%
    \special{pa 3261 280}%
    \special{fp}%
    \special{pn 11}%
    \special{pa 3745 127}%
    \special{pa 3261 280}%
    \special{fp}%
    \special{pa 3261 280}%
    \special{pa 3745 433}%
    \special{fp}%
    \special{pa 2140 280}%
    \special{pa 2650 280}%
    \special{fp}%
    \special{pn 8}%
    \special{pa 2140 790}%
    \special{pa 2650 790}%
    \special{da 0.051}%
    \special{ar 2013 790 127 127 0 6.28319}%
    \special{ar 2013 280 127 127 0 6.28319}%
    \special{ar 3873 433 127 127 0 6.28319}%
    \special{ar 3873 127 127 127 0 6.28319}%
    \graphtemp=.5ex\advance\graphtemp by 0.892in
    \rlap{\kern 2.650in\lower\graphtemp\hbox to 0pt{\hss $u$\hss}}%
    \graphtemp=.5ex\advance\graphtemp by 0.892in
    \rlap{\kern 2.955in\lower\graphtemp\hbox to 0pt{\hss $z$\hss}}%
    \graphtemp=.5ex\advance\graphtemp by 0.892in
    \rlap{\kern 3.261in\lower\graphtemp\hbox to 0pt{\hss $u'$\hss}}%
    \graphtemp=.5ex\advance\graphtemp by 0.178in
    \rlap{\kern 2.650in\lower\graphtemp\hbox to 0pt{\hss $v$\hss}}%
    \graphtemp=.5ex\advance\graphtemp by 0.178in
    \rlap{\kern 2.955in\lower\graphtemp\hbox to 0pt{\hss $j$\hss}}%
    \graphtemp=.5ex\advance\graphtemp by 0.178in
    \rlap{\kern 3.261in\lower\graphtemp\hbox to 0pt{\hss $v'$\hss}}%
    \hbox{\vrule depth0.917in width0pt height 0pt}%
    \kern 4.000in
  }%
}%
}
\caption{Cases (1) and (2) for Lemma~\ref{triangle}\label{figtri}}
\end{figure}

\begin{lem}\label{3config}
The following configurations are $3$-reducible.

{\bf A}. A $2$-vertex or $3$-vertex having a $1$-neighbor.

{\bf B}. A $4^-$-vertex whose neighbors all have degree $2$.

{\bf C}. A $3$-vertex having two $2$-neighbors, one of which has a $2$-neighbor.

{\bf D}. A $4$-vertex having a $1$-neighbor and a $2^-$-neighbor.

{\bf E}. A $5^+$-vertex $v$ with 
$3p_1+2p_2\ge d_G(v)$, where $p_i$ is the number of $i$-neighbors of $v$.  
\end{lem}
\begin{proof}
Let $v$ be such a vertex in a graph $G$.  Let $U_i$ be the set of $i$-neighbors
of $v$.  Form the derived graph $G'$ as specified in Definition~\ref{basic}
(deleting the bold core), except that in addition any resulting isolated edges
are also deleted.  We show that a proper $3$-weighting $w'$ of $G'$ can
be used to obtain a proper $3$-weighting $w$ of $G$.

{\bf Case A:} {\it $d_G(v)\le3$ and $v$ has a $1$-neighbor $u$.}
As in Lemma~\ref{choose}, we can choose $w(uv)$ to satisfy the other edges at
$v$.  With $d_G(v)\ge2$, the edge $uv$ is automatically satisfied.

By Case {\bf A}, deleting the core in Cases {\bf B,C,D} leaves no isolated
edges.

\begin{figure}[h]
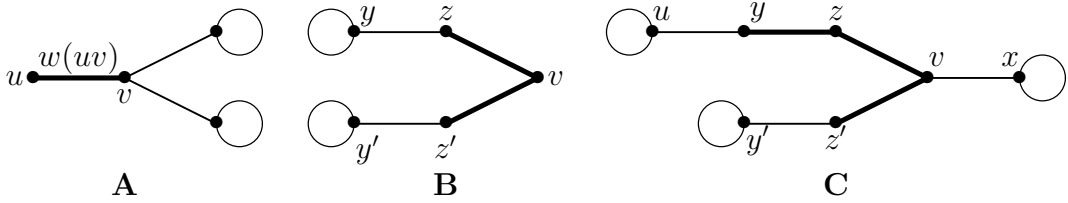

\gpic{
\expandafter\ifx\csname graph\endcsname\relax \csname newbox\endcsname\graph\fi
\expandafter\ifx\csname graphtemp\endcsname\relax \csname newdimen\endcsname\graphtemp\fi
\setbox\graph=\vtop{\vskip 0pt\hbox{%
    \graphtemp=.5ex\advance\graphtemp by 0.360in
    \rlap{\kern 0.096in\lower\graphtemp\hbox to 0pt{\hss $\bu$\hss}}%
    \graphtemp=.5ex\advance\graphtemp by 0.360in
    \rlap{\kern 0.576in\lower\graphtemp\hbox to 0pt{\hss $\bu$\hss}}%
    \graphtemp=.5ex\advance\graphtemp by 0.120in
    \rlap{\kern 1.057in\lower\graphtemp\hbox to 0pt{\hss $\bu$\hss}}%
    \graphtemp=.5ex\advance\graphtemp by 0.600in
    \rlap{\kern 1.057in\lower\graphtemp\hbox to 0pt{\hss $\bu$\hss}}%
    \special{pn 28}%
    \special{pa 96 360}%
    \special{pa 576 360}%
    \special{fp}%
    \special{pn 11}%
    \special{pa 1057 120}%
    \special{pa 576 360}%
    \special{fp}%
    \special{pa 576 360}%
    \special{pa 1057 600}%
    \special{fp}%
    \graphtemp=.5ex\advance\graphtemp by 0.360in
    \rlap{\kern 0.000in\lower\graphtemp\hbox to 0pt{\hss $u$\hss}}%
    \graphtemp=.5ex\advance\graphtemp by 0.456in
    \rlap{\kern 0.576in\lower\graphtemp\hbox to 0pt{\hss $v$\hss}}%
    \graphtemp=.5ex\advance\graphtemp by 0.264in
    \rlap{\kern 0.336in\lower\graphtemp\hbox to 0pt{\hss $w(uv)$\hss}}%
    \special{pn 8}%
    \special{ar 1177 120 120 120 0 6.28319}%
    \special{ar 1177 600 120 120 0 6.28319}%
    \graphtemp=.5ex\advance\graphtemp by 0.600in
    \rlap{\kern 1.777in\lower\graphtemp\hbox to 0pt{\hss $\bu$\hss}}%
    \graphtemp=.5ex\advance\graphtemp by 0.600in
    \rlap{\kern 2.258in\lower\graphtemp\hbox to 0pt{\hss $\bu$\hss}}%
    \graphtemp=.5ex\advance\graphtemp by 0.360in
    \rlap{\kern 2.738in\lower\graphtemp\hbox to 0pt{\hss $\bu$\hss}}%
    \graphtemp=.5ex\advance\graphtemp by 0.120in
    \rlap{\kern 2.258in\lower\graphtemp\hbox to 0pt{\hss $\bu$\hss}}%
    \graphtemp=.5ex\advance\graphtemp by 0.120in
    \rlap{\kern 1.777in\lower\graphtemp\hbox to 0pt{\hss $\bu$\hss}}%
    \special{ar 1657 600 120 120 0 6.28319}%
    \special{ar 1657 120 120 120 0 6.28319}%
    \special{pn 28}%
    \special{pa 2258 600}%
    \special{pa 2738 360}%
    \special{fp}%
    \special{pa 2738 360}%
    \special{pa 2258 120}%
    \special{fp}%
    \special{pn 11}%
    \special{pa 1777 600}%
    \special{pa 2258 600}%
    \special{fp}%
    \special{pa 2258 120}%
    \special{pa 1777 120}%
    \special{fp}%
    \graphtemp=.5ex\advance\graphtemp by 0.024in
    \rlap{\kern 1.849in\lower\graphtemp\hbox to 0pt{\hss $y$\hss}}%
    \graphtemp=.5ex\advance\graphtemp by 0.745in
    \rlap{\kern 1.849in\lower\graphtemp\hbox to 0pt{\hss $y'$\hss}}%
    \graphtemp=.5ex\advance\graphtemp by 0.024in
    \rlap{\kern 2.258in\lower\graphtemp\hbox to 0pt{\hss $z$\hss}}%
    \graphtemp=.5ex\advance\graphtemp by 0.360in
    \rlap{\kern 2.834in\lower\graphtemp\hbox to 0pt{\hss $v$\hss}}%
    \graphtemp=.5ex\advance\graphtemp by 0.745in
    \rlap{\kern 2.258in\lower\graphtemp\hbox to 0pt{\hss $z'$\hss}}%
    \graphtemp=.5ex\advance\graphtemp by 0.600in
    \rlap{\kern 3.819in\lower\graphtemp\hbox to 0pt{\hss $\bu$\hss}}%
    \graphtemp=.5ex\advance\graphtemp by 0.120in
    \rlap{\kern 3.819in\lower\graphtemp\hbox to 0pt{\hss $\bu$\hss}}%
    \graphtemp=.5ex\advance\graphtemp by 0.600in
    \rlap{\kern 4.299in\lower\graphtemp\hbox to 0pt{\hss $\bu$\hss}}%
    \graphtemp=.5ex\advance\graphtemp by 0.120in
    \rlap{\kern 4.299in\lower\graphtemp\hbox to 0pt{\hss $\bu$\hss}}%
    \graphtemp=.5ex\advance\graphtemp by 0.360in
    \rlap{\kern 4.779in\lower\graphtemp\hbox to 0pt{\hss $\bu$\hss}}%
    \graphtemp=.5ex\advance\graphtemp by 0.360in
    \rlap{\kern 5.260in\lower\graphtemp\hbox to 0pt{\hss $\bu$\hss}}%
    \graphtemp=.5ex\advance\graphtemp by 0.120in
    \rlap{\kern 3.338in\lower\graphtemp\hbox to 0pt{\hss $\bu$\hss}}%
    \special{pa 3338 120}%
    \special{pa 3819 120}%
    \special{fp}%
    \special{pa 3819 600}%
    \special{pa 4299 600}%
    \special{fp}%
    \special{pa 4779 360}%
    \special{pa 5260 360}%
    \special{fp}%
    \special{pn 28}%
    \special{pa 3819 120}%
    \special{pa 4299 120}%
    \special{fp}%
    \special{pa 4299 120}%
    \special{pa 4779 360}%
    \special{fp}%
    \special{pa 4779 360}%
    \special{pa 4299 600}%
    \special{fp}%
    \special{pn 8}%
    \special{ar 3699 600 120 120 0 6.28319}%
    \special{ar 5380 360 120 120 0 6.28319}%
    \special{ar 3218 120 120 120 0 6.28319}%
    \graphtemp=.5ex\advance\graphtemp by 0.697in
    \rlap{\kern 3.891in\lower\graphtemp\hbox to 0pt{\hss $y'$\hss}}%
    \graphtemp=.5ex\advance\graphtemp by 0.697in
    \rlap{\kern 4.299in\lower\graphtemp\hbox to 0pt{\hss $z'$\hss}}%
    \graphtemp=.5ex\advance\graphtemp by 0.024in
    \rlap{\kern 4.299in\lower\graphtemp\hbox to 0pt{\hss $z$\hss}}%
    \graphtemp=.5ex\advance\graphtemp by 0.264in
    \rlap{\kern 4.828in\lower\graphtemp\hbox to 0pt{\hss $v$\hss}}%
    \graphtemp=.5ex\advance\graphtemp by 0.024in
    \rlap{\kern 3.386in\lower\graphtemp\hbox to 0pt{\hss $u$\hss}}%
    \graphtemp=.5ex\advance\graphtemp by 0.264in
    \rlap{\kern 5.212in\lower\graphtemp\hbox to 0pt{\hss $x$\hss}}%
    \graphtemp=.5ex\advance\graphtemp by 0.000in
    \rlap{\kern 3.891in\lower\graphtemp\hbox to 0pt{\hss $y$\hss}}%
    \graphtemp=.5ex\advance\graphtemp by 0.937in
    \rlap{\kern 0.576in\lower\graphtemp\hbox to 0pt{\hss {\bf A}\hss}}%
    \graphtemp=.5ex\advance\graphtemp by 0.937in
    \rlap{\kern 2.258in\lower\graphtemp\hbox to 0pt{\hss {\bf B}\hss}}%
    \graphtemp=.5ex\advance\graphtemp by 0.937in
    \rlap{\kern 4.299in\lower\graphtemp\hbox to 0pt{\hss {\bf C}\hss}}%
    \hbox{\vrule depth0.937in width0pt height 0pt}%
    \kern 5.500in
  }%
}%
}
\vspace{-.5pc}
\caption{Cases {\bf A, B, C} for Lemma~\ref{3config}\label{3fig52}}

\end{figure}
\vspace{-.5pc}

{\bf Case B:} {\it $d_G(v)\le 4$ and $U_2=N_G(v)$}. 
Let $z$ and $z'$ be $2$-neighbors of $v$, with $N_G(z)=\{v,y'\}$ and
$N_G(z')=\{v,y\}$.  By Lemma~\ref{triangle}, $\{y,y'\}\cap\{z,z'\}=\nul$ (see
Figure~\ref{3fig52}B).  Let $G'=G-\{vz,vz'\}$.
If $d_G(v)=2$, then choose $w(vz)$ to satisfy $yz$ and $vz'$, and choose
$w(vz')$ to satisfy $y'z'$ and $vz$.  
If $d_G(v)\in\{3,4\}$, then for $z\in N_G(v)$ with $zy\in E(G')$, choose
$w(vz)\in\{2,3\}-\{\rp{y}{z}\}$ to satisfy $yz$.  Since $d(v)\ge3\ge w'(zy)$,
such choices on all of $\Gamma_G(v)$ also satisfy $zv$.

{\bf Case C:} {\it $d_G(v)=3$ and $U_2=\{z,z'\}$, with $z$ having a
$2$-neighbor $y$.}
By Lemma~\ref{triangle}, we may assume $y\ne z'$.  Let $G'=G-\{vz,vz',zy\}$,
leaving $vx,z'y',yu\in E(G')$ (see Figure~\ref{3fig52}C).  Choose $w(vz)$ to
satisfy $zy$ and $vz'$, then $w(vz')$ to satisfy $z'y'$ and $vx$, and finally
$w(zy)$ to satisfy $yu$ and $vz$.

{\bf Case D:} {\it $d_G(v)=4$ and $N_G(v)=\{u,z,x,x'\}$ with $d_G(u)=1$
and $d_G(z)\le2$.}  By Lemma~\ref{assume2}, we may assume $d_G(z)=2$.  Let
$G'=G-\{vu,vz\}$, leaving $zy\in E(G')$ (see Figure~\ref{3fig52DE}D, where $y$
may be in $N_G(v)$).  When choosing $w(vz)$ to satisfy $zy$ and choosing
$w(vu)$ to satisfy $vz$, each has at least two possible values.  Hence they can
be chosen with three possible values for $w(vz)+w(vu)$, yielding a choice that
also satisfies $vx$ and $vx'$.

\begin{figure}[h]
\gpic{
\expandafter\ifx\csname graph\endcsname\relax \csname newbox\endcsname\graph\fi
\expandafter\ifx\csname graphtemp\endcsname\relax \csname newdimen\endcsname\graphtemp\fi
\setbox\graph=\vtop{\vskip 0pt\hbox{%
    \graphtemp=.5ex\advance\graphtemp by 0.661in
    \rlap{\kern 0.264in\lower\graphtemp\hbox to 0pt{\hss $\bu$\hss}}%
    \graphtemp=.5ex\advance\graphtemp by 0.661in
    \rlap{\kern 0.793in\lower\graphtemp\hbox to 0pt{\hss $\bu$\hss}}%
    \graphtemp=.5ex\advance\graphtemp by 0.132in
    \rlap{\kern 0.793in\lower\graphtemp\hbox to 0pt{\hss $\bu$\hss}}%
    \graphtemp=.5ex\advance\graphtemp by 0.397in
    \rlap{\kern 1.322in\lower\graphtemp\hbox to 0pt{\hss $\bu$\hss}}%
    \graphtemp=.5ex\advance\graphtemp by 0.661in
    \rlap{\kern 1.851in\lower\graphtemp\hbox to 0pt{\hss $\bu$\hss}}%
    \graphtemp=.5ex\advance\graphtemp by 0.132in
    \rlap{\kern 1.851in\lower\graphtemp\hbox to 0pt{\hss $\bu$\hss}}%
    \special{pn 11}%
    \special{pa 1851 661}%
    \special{pa 1322 397}%
    \special{fp}%
    \special{pa 1322 397}%
    \special{pa 1851 132}%
    \special{fp}%
    \special{pa 264 661}%
    \special{pa 793 661}%
    \special{fp}%
    \special{pn 28}%
    \special{pa 793 661}%
    \special{pa 1322 397}%
    \special{fp}%
    \special{pa 1322 397}%
    \special{pa 793 132}%
    \special{fp}%
    \special{pn 8}%
    \special{ar 1983 661 132 132 0 6.28319}%
    \special{ar 1983 132 132 132 0 6.28319}%
    \special{ar 132 661 132 132 0 6.28319}%
    \graphtemp=.5ex\advance\graphtemp by 0.736in
    \rlap{\kern 0.339in\lower\graphtemp\hbox to 0pt{\hss $y$\hss}}%
    \graphtemp=.5ex\advance\graphtemp by 0.767in
    \rlap{\kern 0.793in\lower\graphtemp\hbox to 0pt{\hss $z$\hss}}%
    \graphtemp=.5ex\advance\graphtemp by 0.026in
    \rlap{\kern 0.793in\lower\graphtemp\hbox to 0pt{\hss $u$\hss}}%
    \graphtemp=.5ex\advance\graphtemp by 0.291in
    \rlap{\kern 1.375in\lower\graphtemp\hbox to 0pt{\hss $v$\hss}}%
    \graphtemp=.5ex\advance\graphtemp by 0.736in
    \rlap{\kern 1.776in\lower\graphtemp\hbox to 0pt{\hss $x$\hss}}%
    \graphtemp=.5ex\advance\graphtemp by 0.057in
    \rlap{\kern 1.776in\lower\graphtemp\hbox to 0pt{\hss $x'$\hss}}%
    \graphtemp=.5ex\advance\graphtemp by 1.084in
    \rlap{\kern 1.322in\lower\graphtemp\hbox to 0pt{\hss {\bf D}\hss}}%
    \graphtemp=.5ex\advance\graphtemp by 1.084in
    \rlap{\kern 3.172in\lower\graphtemp\hbox to 0pt{\hss {\bf E}\hss}}%
    \graphtemp=.5ex\advance\graphtemp by 0.397in
    \rlap{\kern 3.701in\lower\graphtemp\hbox to 0pt{\hss $\bu$\hss}}%
    \graphtemp=.5ex\advance\graphtemp by 0.397in
    \rlap{\kern 3.172in\lower\graphtemp\hbox to 0pt{\hss $\bu$\hss}}%
    \graphtemp=.5ex\advance\graphtemp by 0.397in
    \rlap{\kern 2.644in\lower\graphtemp\hbox to 0pt{\hss $\bu$\hss}}%
    \graphtemp=.5ex\advance\graphtemp by 0.925in
    \rlap{\kern 3.701in\lower\graphtemp\hbox to 0pt{\hss $\bu$\hss}}%
    \graphtemp=.5ex\advance\graphtemp by 0.397in
    \rlap{\kern 4.230in\lower\graphtemp\hbox to 0pt{\hss $\bu$\hss}}%
    \graphtemp=.5ex\advance\graphtemp by 0.291in
    \rlap{\kern 3.701in\lower\graphtemp\hbox to 0pt{\hss $v$\hss}}%
    \graphtemp=.5ex\advance\graphtemp by 0.291in
    \rlap{\kern 3.172in\lower\graphtemp\hbox to 0pt{\hss $z$\hss}}%
    \graphtemp=.5ex\advance\graphtemp by 0.291in
    \rlap{\kern 2.644in\lower\graphtemp\hbox to 0pt{\hss $y$\hss}}%
    \special{pn 11}%
    \special{ar 2644 397 106 330 0 6.28319}%
    \special{ar 3172 397 106 330 0 6.28319}%
    \special{ar 3701 925 330 106 0 6.28319}%
    \special{ar 4230 397 106 330 0 6.28319}%
    \graphtemp=.5ex\advance\graphtemp by 0.859in
    \rlap{\kern 3.172in\lower\graphtemp\hbox to 0pt{\hss $U_2$\hss}}%
    \graphtemp=.5ex\advance\graphtemp by 0.925in
    \rlap{\kern 4.137in\lower\graphtemp\hbox to 0pt{\hss $U_1$\hss}}%
    \graphtemp=.5ex\advance\graphtemp by 0.397in
    \rlap{\kern 4.600in\lower\graphtemp\hbox to 0pt{\hss $N_{G'}(v)$\hss}}%
    \graphtemp=.5ex\advance\graphtemp by 0.925in
    \rlap{\kern 3.807in\lower\graphtemp\hbox to 0pt{\hss $u$\hss}}%
    \graphtemp=.5ex\advance\graphtemp by 0.291in
    \rlap{\kern 4.230in\lower\graphtemp\hbox to 0pt{\hss $x$\hss}}%
    \special{pa 3172 66}%
    \special{pa 2644 66}%
    \special{fp}%
    \special{pa 3172 397}%
    \special{pa 2644 397}%
    \special{fp}%
    \special{pa 3172 727}%
    \special{pa 2644 727}%
    \special{fp}%
    \special{pn 28}%
    \special{pa 3701 397}%
    \special{pa 4032 925}%
    \special{fp}%
    \special{pa 3701 397}%
    \special{pa 3701 925}%
    \special{fp}%
    \special{pa 3701 397}%
    \special{pa 3371 925}%
    \special{fp}%
    \special{pa 3701 397}%
    \special{pa 3172 727}%
    \special{fp}%
    \special{pa 3701 397}%
    \special{pa 3172 66}%
    \special{fp}%
    \special{pa 3701 397}%
    \special{pa 3172 397}%
    \special{fp}%
    \special{pn 11}%
    \special{pa 3701 397}%
    \special{pa 4230 66}%
    \special{fp}%
    \special{pa 3701 397}%
    \special{pa 4230 397}%
    \special{fp}%
    \special{pa 3701 397}%
    \special{pa 4230 727}%
    \special{fp}%
    \hbox{\vrule depth1.084in width0pt height 0pt}%
    \kern 4.600in
  }%
}%
}
\vspace{-.5pc}
\caption{Cases {\bf D} and {\bf E} for Lemma~\ref{3config}\label{3fig52DE}}
\end{figure}
\vspace{-.5pc}


{\bf Case E:} {\it $d_G(v)\ge5$ and $3p_1+2p_2\ge d_G(v)$.}
For $z\in U_2$, let $y$ be the neighbor of $z$ in $G'$.  To satisfy $yz$ when
$y\notin U_2$ (see Figure~\ref{3fig52DE}E), we need $w(vz)\ne\rp yz$; there are
at least two such choices for $w(vz)$.  (If $y\in U_2$, then $yz$ is deleted in
$G'$; let $w(yz)=1$.  Now $w(vy)\ne w(vz)$ is needed to satisfy $yz$, leaving
three choices for $w(vy)+v(vz)$.)

Edges to $U_1$ are automatically satisfied, since $d_G(v)\ge2$.
For $z\in U_2$, the edge $zv$ will be satisfied, since $d_G(v)\ge5$ yields
$\rw vz\ge 4>3\ge w(zy)$.

It remains to satisfy $\Gamma_{G'}(v)$.  
Let $\sigma=\sum_{e\in E(G)-E(G')} w(e)$.  We
need $\sigma\ne\pwp(x)-\pwp(v)$ when $x\in N_{G'}(v)$, so $\sigma$ must avoid
$d_G(v)-p_1-p_2$ values.  It suffices to show that there are $1+2p_1+p_2$
choices for $\sigma$, since we are given $2p_1+p_2\ge d_G(v)-p_1-p_2$.

Weights on edges from $v$ to $U_1$ have three choices.  Those to $U_2$ have at
least two choices, except that two such edges incident to neighboring
$2$-vertices instead have three choices for the sum of their two weights.
Starting with the smallest choices, we can make $2p_1+p_2$ augmentations to the
sum, always using choices that satisfy the constraints discussed earlier.
Hence there are enough choices for $\sigma$ to satisfy the final constraints.
\end{proof}


We now present a discharging argument to obtain an unavoidable set of
configurations.

\begin{lem}\label{52struct}
If a graph $G$ without isolated edges has average degree less than $5/2$,
then $G$ contains one of the following configurations.

{\bf A}. A $2$-vertex or $3$-vertex having a $1$-neighbor.

{\bf B}. A $4^-$-vertex whose neighbors all have degree $2$.

{\bf C}. A $3$-vertex having two $2$-neighbors, one of which has a $2$-neighbor.

{\bf D}. A $4$-vertex having a $1$-neighbor and a $2^-$-neighbor.

{\bf E}. A $5^+$-vertex $v$ with 
$3p_1+p_2\ge 2d_G(v)-4$, where $p_i$ is the number of $i$-neighbors of $v$.  

\end{lem}

\begin{proof}
We prove that a graph $G$ containing none of {\bf A-E} has average degree at
least $5/2$.  Give every vertex $v$ in $G$ initial charge $d_G(v)$.  Move
charge via the following rules:

\medskip
{\narrower

\noindent
(1) Each $4^+$-vertex gives $\FR32$ to each $1$-neighbor and
$\FR12$ to each $2$-neighbor.

\smallskip

\noindent
(2) Each $3$-vertex with a $2$-neighbor gives total $\frac{1}{2}$ to
its $2$-neighbors, split equally if it has two $2$-neighbors.

}
\medskip

\noindent
Let $\mu(v)$ denote the resulting charge at $v$; it suffices to check that
$\mu(v)\ge\FR52$ for all $v$.  For ${\bf Z}\in\{{\bf A,B,C,D,E}\}$, let
$\bZ$ mean ``configuration {\bf Z} does not occur in $G$''.  

Case $d(v)=1$: By $\bA$, the neighbor of a $1$-vertex $v$ has degree at least
$4$, so $\mu(v)=\FR52$.

Case $d(v)=2$: By $\bA$ and $\bB$, $v$ has a $3^+$-neighbor and receives from
it $\FR14$ or $\FR12$.  If only $\FR14$, then $\bC$ implies that $v$ also
receives at least $\FR14$ from its other neighbor, so $\mu(v)\ge\FR52$.

Case $d(v)=3$: By $\bA$ and $\bB$, $v$ has no $1$-neighbor and at most two
$2$-neighbors.  Hence $v$ gives away $0$ or $\FR12$, and $\mu(v)\ge\FR52$.

Case $d(v)=4$: By $\bD$ and $\bB$, $v$ has at most one $1$-neighbor, has a
$2$-neighbor only if it has no $1$-neighbor, and has at most three
$2$-neighbors.  It loses at most $\FR32$, and $\mu(v)\ge\FR52$. 
 
Case $d(v)\ge5$: $v$ gives $\FR32$ to each $1$-neighbor and $\FR12$ to each
$2$-neighbor.  By $\bE$,
$\mu(v)=d_G(v)-\FR12(3p_1+p_2)\ge d_G(v)-\FR12(2d_G(v)-5)=\FR52$.
\end{proof}

\begin{thm}{\label{aved}}
If $G$ has no isolated edge, and $\Mad(G)<\frac{5}{2}$, then $G$ has a proper
$3$-weighting.
\end{thm}
\begin{proof}
A minimal counterexample contains none of the configurations {\bf A-E} in
Lemma~\ref{3config}.  However, it has average degree less than $5/2$, and hence
it contains a configuration listed in Lemma~\ref{52struct}.  The lists are the
same except for {\bf E}.  Since a $5^+$-vertex $v$ satisfying $3p_1+p_2\ge
2d_G(v)-4$ also satisfies $3p_1+2p_2\ge d_G(v)$, every graph with $\Mad(G)<5/2$
contains a reducible configuration.
\end{proof}

For the $1,2$-Conjecture, we again begin with reducible configurations.
Isolated edges are now allowed, which eliminates some technicalities.
We will be able to use the unavoidable set obtained in Lemma~\ref{52struct},
but the list of $2$-reducible configurations is different.  The new technique
here is that in obtaining the proper total $2$-weighting of $G$ from such a
weighting of a subgraph $G'$, we may erase weights from some vertices and
recolor them. 

Configuration {\bf B} in the next lemma is more general than is needed for
the $1,2$-Conjecture when $\Mad(G)<5/2$, but we will need its full generality
in the proof for $\Mad(G)<8/3$.

\begin{lem}\label{2config}
A minimal $2$-bad graph contains none of the following configurations.

{\bf A}. A $3^-$-vertex having a $1$-neighbor.

{\bf B}. A $4^-$-vertex having two $2^-$-neighbors.

{\bf C}. A $5^+$-vertex $v$ whose number of $2^-$-neighbors is at least
$(d(v)-1)/2$. 
\end{lem}

\begin{proof}
In each case, we obtain a proper total $2$-weighting of $G$ from such a
weighting $w'$ of the derived graph $G'$.  Let $v$ be the specified vertex.
Since isolated edges have proper total $2$-weightings,
we may assume that any $1$-vertex in $G$ has a $2^+$-neighbor.  In the
extension arguments, we use Lemma~\ref{choose} frequently to choose labels.

{\bf Case A:} {\it $d(v)\le 3$, and $v$ has a $1$-neighbor $u$.}
For $d(v)=3$, let $N_{G'}(v)=\{x,x'\}$ (see Figure~\ref{2fig52}A).  Uncolor
$v$, and then choose $w(v),w(uv)\in\{1,2\}$ to satisfy $vx$ and $vx'$.
Now choose $w(u)$ to satisfy $uv$.  When $d(v)=2$, we only need $w(v)+w(uv)$
to avoid one value.

\begin{figure}[h]
\gpic{
\expandafter\ifx\csname graph\endcsname\relax \csname newbox\endcsname\graph\fi
\expandafter\ifx\csname graphtemp\endcsname\relax \csname newdimen\endcsname\graphtemp\fi
\setbox\graph=\vtop{\vskip 0pt\hbox{%
    \graphtemp=.5ex\advance\graphtemp by 0.380in
    \rlap{\kern 0.101in\lower\graphtemp\hbox to 0pt{\hss $\bu$\hss}}%
    \graphtemp=.5ex\advance\graphtemp by 0.380in
    \rlap{\kern 0.608in\lower\graphtemp\hbox to 0pt{\hss $\bu$\hss}}%
    \graphtemp=.5ex\advance\graphtemp by 0.127in
    \rlap{\kern 1.114in\lower\graphtemp\hbox to 0pt{\hss $\bu$\hss}}%
    \graphtemp=.5ex\advance\graphtemp by 0.633in
    \rlap{\kern 1.114in\lower\graphtemp\hbox to 0pt{\hss $\bu$\hss}}%
    \special{pn 28}%
    \special{pa 101 380}%
    \special{pa 608 380}%
    \special{fp}%
    \special{pn 11}%
    \special{pa 1114 127}%
    \special{pa 608 380}%
    \special{fp}%
    \special{pa 608 380}%
    \special{pa 1114 633}%
    \special{fp}%
    \graphtemp=.5ex\advance\graphtemp by 0.380in
    \rlap{\kern 0.000in\lower\graphtemp\hbox to 0pt{\hss $u$\hss}}%
    \graphtemp=.5ex\advance\graphtemp by 0.481in
    \rlap{\kern 0.608in\lower\graphtemp\hbox to 0pt{\hss $v$\hss}}%
    \graphtemp=.5ex\advance\graphtemp by 0.278in
    \rlap{\kern 0.354in\lower\graphtemp\hbox to 0pt{\hss $~$\hss}}%
    \graphtemp=.5ex\advance\graphtemp by 0.055in
    \rlap{\kern 1.042in\lower\graphtemp\hbox to 0pt{\hss $x$\hss}}%
    \graphtemp=.5ex\advance\graphtemp by 0.705in
    \rlap{\kern 0.992in\lower\graphtemp\hbox to 0pt{\hss $x'$\hss}}%
    \special{pn 8}%
    \special{ar 1241 127 127 127 0 6.28319}%
    \special{ar 1241 633 127 127 0 6.28319}%
    \graphtemp=.5ex\advance\graphtemp by 0.380in
    \rlap{\kern 3.139in\lower\graphtemp\hbox to 0pt{\hss $\bu$\hss}}%
    \graphtemp=.5ex\advance\graphtemp by 0.380in
    \rlap{\kern 2.633in\lower\graphtemp\hbox to 0pt{\hss $\bu$\hss}}%
    \graphtemp=.5ex\advance\graphtemp by 0.380in
    \rlap{\kern 2.127in\lower\graphtemp\hbox to 0pt{\hss $\bu$\hss}}%
    \graphtemp=.5ex\advance\graphtemp by 0.380in
    \rlap{\kern 2.127in\lower\graphtemp\hbox to 0pt{\hss $\bu$\hss}}%
    \graphtemp=.5ex\advance\graphtemp by 0.380in
    \rlap{\kern 3.646in\lower\graphtemp\hbox to 0pt{\hss $\bu$\hss}}%
    \graphtemp=.5ex\advance\graphtemp by 0.278in
    \rlap{\kern 3.139in\lower\graphtemp\hbox to 0pt{\hss $v$\hss}}%
    \graphtemp=.5ex\advance\graphtemp by 0.278in
    \rlap{\kern 2.633in\lower\graphtemp\hbox to 0pt{\hss $z$\hss}}%
    \graphtemp=.5ex\advance\graphtemp by 0.278in
    \rlap{\kern 2.127in\lower\graphtemp\hbox to 0pt{\hss $y$\hss}}%
    \special{pn 11}%
    \special{ar 2127 380 101 316 0 6.28319}%
    \special{ar 2633 380 101 316 0 6.28319}%
    \special{ar 3646 380 101 316 0 6.28319}%
    \graphtemp=.5ex\advance\graphtemp by 0.823in
    \rlap{\kern 2.633in\lower\graphtemp\hbox to 0pt{\hss $U$\hss}}%
    \graphtemp=.5ex\advance\graphtemp by 0.380in
    \rlap{\kern 4.000in\lower\graphtemp\hbox to 0pt{\hss $X$\hss}}%
    \graphtemp=.5ex\advance\graphtemp by 0.278in
    \rlap{\kern 3.646in\lower\graphtemp\hbox to 0pt{\hss $x$\hss}}%
    \special{pa 2633 63}%
    \special{pa 2127 63}%
    \special{fp}%
    \special{pa 2633 380}%
    \special{pa 2127 380}%
    \special{fp}%
    \special{pa 2633 696}%
    \special{pa 2127 696}%
    \special{fp}%
    \special{pn 28}%
    \special{pa 3139 380}%
    \special{pa 2633 696}%
    \special{fp}%
    \special{pa 3139 380}%
    \special{pa 2633 63}%
    \special{fp}%
    \special{pa 3139 380}%
    \special{pa 2633 380}%
    \special{fp}%
    \special{pn 11}%
    \special{pa 3139 380}%
    \special{pa 3646 63}%
    \special{fp}%
    \special{pa 3139 380}%
    \special{pa 3646 380}%
    \special{fp}%
    \special{pa 3139 380}%
    \special{pa 3646 696}%
    \special{fp}%
    \hbox{\vrule depth0.823in width0pt height 0pt}%
    \kern 4.000in
  }%
}%
}
\caption{Cases {\bf A} and {\bf C} for Lemma~\ref{2config}\label{2fig52}}
\end{figure}

{\bf Case C:} {\it $d(v)\ge5$ and $v$ has at least $\FR{d(v)-1}2$
$2^-$-neighbors.}
Let $U$ be a set of $p$ such neighbors, where $p=\CL{\FR{d(v)-1}2}$, and let
$G'=G-[v,U]$ and $X=N_{G'}(v)$.  By Lemma~\ref{assume2}, for $z\in U$ we may
assume $d_G(z)=2$ and let $\{y\}=N_{G'}(z)$ (see Figure~\ref{2fig52}C).
Uncolor $z$.

By Lemma~\ref{choose}, we can choose the $p+1$ weights on
$\{v\}\cup\{vz\st z\in U\}$ to satisfy $[v,X]$ in $G$, since $p+1\ge d(v)-p$.
Now choose $w(z)$ for $z\in U$ to satisfy $zy$.  Finally, since $d(v)\ge5$, we
have $\rw{v}{z}\ge 5>4\ge w(z)+w(zy)$, so $zv$ is automatically satisfied.

\smallskip
{\bf Case B:} {\it $v$ has two $2^-$-neighbors $z$ and $z'$.}
By Lemma~\ref{assume2}, we may assume $d_G(z)=d_G(z')=2$.
By Lemma~\ref{triangle}, we may assume $zz'\notin E(G)$.
Let $G'=G-\{vz,vz'\}$.  Since the edges of $G'$ incident to the core
must be satisfied in the extension to $G$, uncolor $v$, $z$, and $z'$.

{\bf Subcase 1:} {\it $d_G(v)\in\{3,4\}$}.
Let $X\!=N_G(v)-\{z,z'\}$ (see Figure~\ref{2configB}).  Let $a=\SM xX w'(vx)$.
If $a\ge 2$, then setting $w(v)=2$ ensures satisfying $vz$ and $vz'$.  Using
$|X|\le 2$, choose $w(vz)$ and $w(vz')$ to satisfy $\Gamma_{G'}(v)$.
Now choose $w(z)$ and $w(z')$ to satisfy $yz$ and $y'z'$, respectively.

Hence we may assume $a=1$, which requires $d_G(v)=3$.  If $w'(yz)=1$,
then set $w(vz')=2$ to ensure satisfying $vz$.  Next choose $w(z')$ to satisfy
$z'y'$, and then choose $w(v)$ and $w(vz)$ to satisfy $vz'$ and $vx$.
Finally, choose $w(z)$ to satisfy $yz$.

By symmetry, we may now assume $a=1$ and $w'(yz)=w'(y'z')=2$, as in the
middle in Figure~\ref{2configB}.  If $w'(y)=1$, then we can exchange $w'(y)$
and $w'(yz)$ with no effect on the satisfaction of any edge in $\Gamma_{G'}(y)$
except $yz$, thereby reaching the case in the preceding paragraph.  Hence by
symmetry we may also assume $w'(y)=w'(y')=2$.

Let $b=\rp xv$.  If $b=4$, then set $w(zv)=w(v)=w(vz')=2$ to satisfy $vx$ and
ensure satisfying $vz$ and $vz'$; then choose $w(z)$ and $w(z')$ to satisfy
$zy$ and $z'y'$, respectively.  If $b\ne4$, then set $w(v)=2$ and
$w(z)=w(zv)=w(vz')=w(z')=1$.  By $\bA$, we have $d_G(y)\ge 2$ and hence
$\rp yz>2$, so $yz$ is satisfied (similarly for $y'z'$).  These values also
satisfy $\Gamma_G(v)$.

\begin{figure}[h]
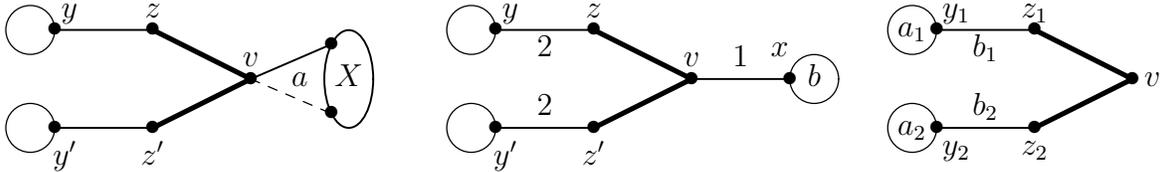

\gpic{
\expandafter\ifx\csname graph\endcsname\relax \csname newbox\endcsname\graph\fi
\expandafter\ifx\csname graphtemp\endcsname\relax \csname newdimen\endcsname\graphtemp\fi
\setbox\graph=\vtop{\vskip 0pt\hbox{%
    \graphtemp=.5ex\advance\graphtemp by 0.641in
    \rlap{\kern 0.256in\lower\graphtemp\hbox to 0pt{\hss $\bu$\hss}}%
    \graphtemp=.5ex\advance\graphtemp by 0.128in
    \rlap{\kern 0.256in\lower\graphtemp\hbox to 0pt{\hss $\bu$\hss}}%
    \graphtemp=.5ex\advance\graphtemp by 0.641in
    \rlap{\kern 0.769in\lower\graphtemp\hbox to 0pt{\hss $\bu$\hss}}%
    \graphtemp=.5ex\advance\graphtemp by 0.128in
    \rlap{\kern 0.769in\lower\graphtemp\hbox to 0pt{\hss $\bu$\hss}}%
    \graphtemp=.5ex\advance\graphtemp by 0.385in
    \rlap{\kern 1.282in\lower\graphtemp\hbox to 0pt{\hss $\bu$\hss}}%
    \special{pn 11}%
    \special{pa 256 641}%
    \special{pa 769 641}%
    \special{fp}%
    \special{pa 256 128}%
    \special{pa 769 128}%
    \special{fp}%
    \special{pn 28}%
    \special{pa 769 641}%
    \special{pa 1282 385}%
    \special{fp}%
    \special{pa 1282 385}%
    \special{pa 769 128}%
    \special{fp}%
    \special{pn 8}%
    \special{ar 128 128 128 128 0 6.28319}%
    \special{ar 128 641 128 128 0 6.28319}%
    \special{pn 11}%
    \special{ar 1795 385 128 256 0 6.28319}%
    \graphtemp=.5ex\advance\graphtemp by 0.566in
    \rlap{\kern 1.704in\lower\graphtemp\hbox to 0pt{\hss $\bu$\hss}}%
    \graphtemp=.5ex\advance\graphtemp by 0.203in
    \rlap{\kern 1.704in\lower\graphtemp\hbox to 0pt{\hss $\bu$\hss}}%
    \graphtemp=.5ex\advance\graphtemp by 0.795in
    \rlap{\kern 0.308in\lower\graphtemp\hbox to 0pt{\hss $y'$\hss}}%
    \graphtemp=.5ex\advance\graphtemp by 0.795in
    \rlap{\kern 0.769in\lower\graphtemp\hbox to 0pt{\hss $z'$\hss}}%
    \graphtemp=.5ex\advance\graphtemp by 0.385in
    \rlap{\kern 1.795in\lower\graphtemp\hbox to 0pt{\hss $X$\hss}}%
    \special{pa 1282 385}%
    \special{pa 1704 203}%
    \special{fp}%
    \special{pn 8}%
    \special{pa 1282 385}%
    \special{pa 1704 566}%
    \special{da 0.051}%
    \graphtemp=.5ex\advance\graphtemp by 0.026in
    \rlap{\kern 0.333in\lower\graphtemp\hbox to 0pt{\hss $y$\hss}}%
    \graphtemp=.5ex\advance\graphtemp by 0.026in
    \rlap{\kern 0.769in\lower\graphtemp\hbox to 0pt{\hss $z$\hss}}%
    \graphtemp=.5ex\advance\graphtemp by 0.282in
    \rlap{\kern 1.282in\lower\graphtemp\hbox to 0pt{\hss $v$\hss}}%
    \graphtemp=.5ex\advance\graphtemp by 0.385in
    \rlap{\kern 1.538in\lower\graphtemp\hbox to 0pt{\hss $a$\hss}}%
    \graphtemp=.5ex\advance\graphtemp by 0.282in
    \rlap{\kern 1.744in\lower\graphtemp\hbox to 0pt{\hss $~$\hss}}%
    \graphtemp=.5ex\advance\graphtemp by 0.641in
    \rlap{\kern 2.564in\lower\graphtemp\hbox to 0pt{\hss $\bu$\hss}}%
    \graphtemp=.5ex\advance\graphtemp by 0.128in
    \rlap{\kern 2.564in\lower\graphtemp\hbox to 0pt{\hss $\bu$\hss}}%
    \graphtemp=.5ex\advance\graphtemp by 0.641in
    \rlap{\kern 3.077in\lower\graphtemp\hbox to 0pt{\hss $\bu$\hss}}%
    \graphtemp=.5ex\advance\graphtemp by 0.128in
    \rlap{\kern 3.077in\lower\graphtemp\hbox to 0pt{\hss $\bu$\hss}}%
    \graphtemp=.5ex\advance\graphtemp by 0.385in
    \rlap{\kern 3.590in\lower\graphtemp\hbox to 0pt{\hss $\bu$\hss}}%
    \graphtemp=.5ex\advance\graphtemp by 0.385in
    \rlap{\kern 4.103in\lower\graphtemp\hbox to 0pt{\hss $\bu$\hss}}%
    \special{pn 11}%
    \special{pa 2564 641}%
    \special{pa 3077 641}%
    \special{fp}%
    \special{pa 2564 128}%
    \special{pa 3077 128}%
    \special{fp}%
    \special{pa 3590 385}%
    \special{pa 4103 385}%
    \special{fp}%
    \special{pn 28}%
    \special{pa 3077 641}%
    \special{pa 3590 385}%
    \special{fp}%
    \special{pa 3590 385}%
    \special{pa 3077 128}%
    \special{fp}%
    \special{pn 8}%
    \special{ar 2436 128 128 128 0 6.28319}%
    \special{ar 2436 641 128 128 0 6.28319}%
    \special{ar 4231 385 128 128 0 6.28319}%
    \graphtemp=.5ex\advance\graphtemp by 0.795in
    \rlap{\kern 2.615in\lower\graphtemp\hbox to 0pt{\hss $y'$\hss}}%
    \graphtemp=.5ex\advance\graphtemp by 0.795in
    \rlap{\kern 3.077in\lower\graphtemp\hbox to 0pt{\hss $z'$\hss}}%
    \graphtemp=.5ex\advance\graphtemp by 0.385in
    \rlap{\kern 4.231in\lower\graphtemp\hbox to 0pt{\hss $b$\hss}}%
    \graphtemp=.5ex\advance\graphtemp by 0.026in
    \rlap{\kern 2.641in\lower\graphtemp\hbox to 0pt{\hss $y$\hss}}%
    \graphtemp=.5ex\advance\graphtemp by 0.026in
    \rlap{\kern 3.077in\lower\graphtemp\hbox to 0pt{\hss $z$\hss}}%
    \graphtemp=.5ex\advance\graphtemp by 0.282in
    \rlap{\kern 3.590in\lower\graphtemp\hbox to 0pt{\hss $v$\hss}}%
    \graphtemp=.5ex\advance\graphtemp by 0.282in
    \rlap{\kern 3.846in\lower\graphtemp\hbox to 0pt{\hss $1$\hss}}%
    \graphtemp=.5ex\advance\graphtemp by 0.231in
    \rlap{\kern 4.051in\lower\graphtemp\hbox to 0pt{\hss $x$\hss}}%
    \graphtemp=.5ex\advance\graphtemp by 0.231in
    \rlap{\kern 2.821in\lower\graphtemp\hbox to 0pt{\hss $2$\hss}}%
    \graphtemp=.5ex\advance\graphtemp by 0.538in
    \rlap{\kern 2.821in\lower\graphtemp\hbox to 0pt{\hss $2$\hss}}%
    \graphtemp=.5ex\advance\graphtemp by 0.641in
    \rlap{\kern 4.872in\lower\graphtemp\hbox to 0pt{\hss $\bu$\hss}}%
    \graphtemp=.5ex\advance\graphtemp by 0.641in
    \rlap{\kern 5.385in\lower\graphtemp\hbox to 0pt{\hss $\bu$\hss}}%
    \graphtemp=.5ex\advance\graphtemp by 0.385in
    \rlap{\kern 5.897in\lower\graphtemp\hbox to 0pt{\hss $\bu$\hss}}%
    \graphtemp=.5ex\advance\graphtemp by 0.128in
    \rlap{\kern 5.385in\lower\graphtemp\hbox to 0pt{\hss $\bu$\hss}}%
    \graphtemp=.5ex\advance\graphtemp by 0.128in
    \rlap{\kern 4.872in\lower\graphtemp\hbox to 0pt{\hss $\bu$\hss}}%
    \special{ar 4744 641 128 128 0 6.28319}%
    \special{ar 4744 128 128 128 0 6.28319}%
    \special{pn 28}%
    \special{pa 5385 641}%
    \special{pa 5897 385}%
    \special{fp}%
    \special{pa 5897 385}%
    \special{pa 5385 128}%
    \special{fp}%
    \special{pn 11}%
    \special{pa 4872 641}%
    \special{pa 5385 641}%
    \special{fp}%
    \special{pa 5385 128}%
    \special{pa 4872 128}%
    \special{fp}%
    \graphtemp=.5ex\advance\graphtemp by 0.128in
    \rlap{\kern 4.744in\lower\graphtemp\hbox to 0pt{\hss $a_1$\hss}}%
    \graphtemp=.5ex\advance\graphtemp by 0.641in
    \rlap{\kern 4.744in\lower\graphtemp\hbox to 0pt{\hss $a_2$\hss}}%
    \graphtemp=.5ex\advance\graphtemp by 0.231in
    \rlap{\kern 5.128in\lower\graphtemp\hbox to 0pt{\hss $b_1$\hss}}%
    \graphtemp=.5ex\advance\graphtemp by 0.538in
    \rlap{\kern 5.128in\lower\graphtemp\hbox to 0pt{\hss $b_2$\hss}}%
    \graphtemp=.5ex\advance\graphtemp by 0.026in
    \rlap{\kern 4.974in\lower\graphtemp\hbox to 0pt{\hss $y_1$\hss}}%
    \graphtemp=.5ex\advance\graphtemp by 0.744in
    \rlap{\kern 4.974in\lower\graphtemp\hbox to 0pt{\hss $y_2$\hss}}%
    \graphtemp=.5ex\advance\graphtemp by 0.026in
    \rlap{\kern 5.385in\lower\graphtemp\hbox to 0pt{\hss $z_1$\hss}}%
    \graphtemp=.5ex\advance\graphtemp by 0.385in
    \rlap{\kern 6.000in\lower\graphtemp\hbox to 0pt{\hss $v$\hss}}%
    \graphtemp=.5ex\advance\graphtemp by 0.744in
    \rlap{\kern 5.385in\lower\graphtemp\hbox to 0pt{\hss $z_2$\hss}}%
    \hbox{\vrule depth0.795in width0pt height 0pt}%
    \kern 6.000in
  }%
}%
}
\caption{Case {\bf B} for Lemma~\ref{83red2}\label{2configB}}
\end{figure}

{\bf Subcase 2:} {\it $d_G(v)=2$.}
For this subcase, let $z_1=z$ and $z_2=z'$.  For $i\in\{1,2\}$, let
$\{y_i\}=N_{G'}(z_i)$, let $b_i=w'(y_iz_i)$, and let $a_i=\rp{y_i}{z_i}$, as on
the right in Figure~\ref{2config}.  To satisfy $y_iz_i$, fix $w(vz_i)=3-w(z_i)$
when $a_i$ is even and $w(vz_i)=w(z_i)$ when $a_i$ is odd.  We then must choose
$w(z_1)$ and $w(z_2)$ (and hence $w(vz_1)$ and $w(vz_2)$) to satisfy $vz_2$ and
$vz_1$.  We need $b_1+w(z_1)\ne w(v)+w(vz_2)$ and $b_2+w(z_2)\ne w(v)+w(vz_1)$.

When $a_1-a_2$ is even, set $w(v)=1$.  When $a_1$ and $a_2$ are both even,
using $w(vz_i)=3-w(z_i)$ converts the requirements to $b_1+w(z_1)\ne 4-w(z_2)$
and $b_2+w(z_2)\ne 4-w(z_1)$.  With two choices for both $w(z_1)$ and $w(z_2)$,
we can pick them so that $w(z_1)+w(z_2)\notin\{4-b_1,4-b_2\}$.  When $a_1$ and
$a_2$ are both odd, using $w(vz_i)=w(z_i)$ it suffices to choose
$w(z_1),w(z_2)\in\{1,2\}$ so that $w(z_1)-w(z_2)\notin\{1-b_1,b_2-1\}$.  Since
the difference can be any of the three values in $\{1,0,-1\}$, this also can be
done.

When $a_1$ and $a_2$ have opposite parity, we may assume that $a_1$ is even.
Now set $w(v)=3-b_1$ and $w(z_1)=w(z_2)=b_1$.  Using $w(vz_1)=3-w(z_1)$ and
$w(vz_2)=w(z_2)$, we have satisfied $vz_1$ because
$b_1+w(z_1)=2b_1\ne 3=w(v)+w(z_2)$, and we have satisfied $vz_2$ because
$b_2+w(z_2)=b_2+b_1\ne 6-2b_1=w(v)+3-w(z_1)$.
\end{proof}

\begin{thm}\label{52thm2}
If $\Mad(G)<5/2$, then $G$ has a proper total $2$-weighting.
\end{thm}
\begin{proof}
By Lemma~\ref{2config}, a minimal counterexample contains no configuration
listed there.  Since it has average degree less than $5/2$, it contains a
configuration listed in Lemma~\ref{52struct}.  Configurations {\bf A--D} are
all $2$-reducible, by {\bf A} and {\bf B} of Lemma~\ref{2config}.  Hence to
show that every graph with $\Mad(G)<5/2$ contains a $2$-reducible
configuration, it
suffices to show that a $5^+$-vertex $v$ satisfying $3p_1+p_2\ge 2d_G(v)-4$
also satisfies $2p_1+2p_2\ge d_G(v)-1$.  If the desired inequality fails, then
subtracting $2p_1+2p_2\le d_G(v)-2$ from the given inequality yields
$p_1\ge d_G(v)-2$.  Since $d_G(v)-2\ge(d_G(v)-1)/2$, the desired inequality
follows.
\end{proof}

\section{Proper Total $2$-weighting when $\Mad(G)<8/3$}\label{sec12}

A graph formed by adding a pendant edge at each vertex of a $3$-regular graph
has average degree $\FR52$.  It has no configuration in Lemma~\ref{2config},
since each $4$-vertex has one $1$-neighbor and three $4$-neighbors.  Further
$2$-reducible configurations will require multiple ``almost-reducible''
vertices.  We introduce two types.

\begin{definition}\label{bvert}
A \textit{$\beta$-vertex} is a  $3$-vertex having exactly one $2$-neighbor
and no $1$-neighbor.
A $\beta'$-vertex is a $2k$-vertex, where $k\geq 2$, having exactly $k-1$
neighbors of degree $1$ and no $2$-neighbor.  For $\gamma\in\{\beta,\beta'\}$,
a \textit{$\gamma$-neighbor} of $v$ is a $\gamma$-vertex in $N(v)$.
\end{definition}

We will show in Lemma~\ref{83red2} that various configurations involving such
vertices are $2$-reducible.  Theorem~\ref{83disch2} shows that these plus
the configurations in Lemma~\ref{2config} form an unavoidable set when
$\Mad(G)<8/3$.  The argument would be shorter if adjacent $\beta'$-vertices of
degree $4$ formed a reducible configuration, but our usual method fails there.

\begin{example}\label{nonred}
Let $v$ and $v'$ be adjacent $\beta'$-vertices of degree $4$ in $G$, having
$1$-neighbors $u$ and $u'$, respectively.  As in Section~\ref{sec52}, the core
$F$ is $\{uv,vv',v'u'\}$, and $G'=G-F$.  A total $2$-weighting $w'$ of $G'$ may
assign labels as indicated in Figure~\ref{betafig}.  To extend $w'$, we need
$w(uv)+w(v)+w(vv')\in \{3,6\}$; hence these three weights must be equal.
Similarly, $w(u'v')+w(v')+w(vv')\in\{3,6\}$.  Since $w(vv')$ can take only one
value, we have forced $\phi_w(v)=\phi_w(v')$.  Hence no extension to a proper
total $2$-weighting is possible.

Another would-be-useful but non-reducible configuration consists of a
$\beta$-vertex $v$ whose $2$-neighbor $z$ has a $2$-neighbor $y$.
A total $2$-weighting $w'$ of $G'$ may assign labels as indicated in
Figure~\ref{betafig}.  The values of $\ppwp$ at the neighbors of $v$ other
than $z$ force $w(v)=w(vz)=1$.  Now satisfying $vz$ requires $w(z)=w(zy)$.
Similarly, satisfying $xy$ requires $w(y)=w(yz)$.  We conclude $w(z)=w(y)$,
but now $yz$ cannot be satisfied, since also $w(yx)=w(zv)$.
\qed
\end{example}

\vspace{-.5pc}
\begin{figure}[h]
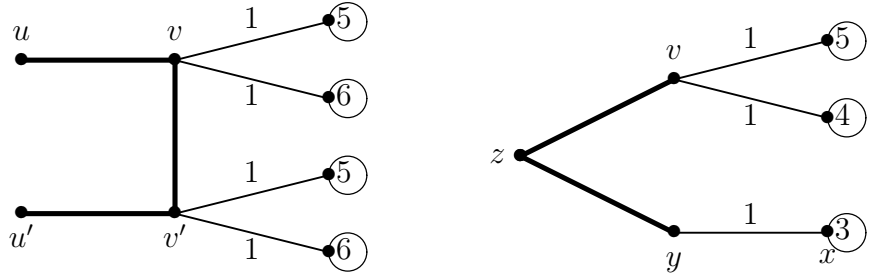

\gpic{
\expandafter\ifx\csname graph\endcsname\relax \csname newbox\endcsname\graph\fi
\expandafter\ifx\csname graphtemp\endcsname\relax \csname newdimen\endcsname\graphtemp\fi
\setbox\graph=\vtop{\vskip 0pt\hbox{%
    \graphtemp=.5ex\advance\graphtemp by 1.105in
    \rlap{\kern 0.080in\lower\graphtemp\hbox to 0pt{\hss $\bu$\hss}}%
    \graphtemp=.5ex\advance\graphtemp by 0.301in
    \rlap{\kern 0.080in\lower\graphtemp\hbox to 0pt{\hss $\bu$\hss}}%
    \graphtemp=.5ex\advance\graphtemp by 1.105in
    \rlap{\kern 0.884in\lower\graphtemp\hbox to 0pt{\hss $\bu$\hss}}%
    \graphtemp=.5ex\advance\graphtemp by 0.301in
    \rlap{\kern 0.884in\lower\graphtemp\hbox to 0pt{\hss $\bu$\hss}}%
    \graphtemp=.5ex\advance\graphtemp by 1.306in
    \rlap{\kern 1.688in\lower\graphtemp\hbox to 0pt{\hss $\bu$\hss}}%
    \graphtemp=.5ex\advance\graphtemp by 0.904in
    \rlap{\kern 1.688in\lower\graphtemp\hbox to 0pt{\hss $\bu$\hss}}%
    \graphtemp=.5ex\advance\graphtemp by 0.502in
    \rlap{\kern 1.688in\lower\graphtemp\hbox to 0pt{\hss $\bu$\hss}}%
    \graphtemp=.5ex\advance\graphtemp by 0.100in
    \rlap{\kern 1.688in\lower\graphtemp\hbox to 0pt{\hss $\bu$\hss}}%
    \special{pn 28}%
    \special{pa 80 1105}%
    \special{pa 884 1105}%
    \special{fp}%
    \special{pa 884 1105}%
    \special{pa 884 301}%
    \special{fp}%
    \special{pa 884 301}%
    \special{pa 80 301}%
    \special{fp}%
    \special{pn 11}%
    \special{pa 1688 1306}%
    \special{pa 884 1105}%
    \special{fp}%
    \special{pa 884 1105}%
    \special{pa 1688 904}%
    \special{fp}%
    \special{pa 1688 502}%
    \special{pa 884 301}%
    \special{fp}%
    \special{pa 884 301}%
    \special{pa 1688 100}%
    \special{fp}%
    \special{pn 8}%
    \special{ar 1788 1306 100 100 0 6.28319}%
    \special{ar 1788 904 100 100 0 6.28319}%
    \special{ar 1788 502 100 100 0 6.28319}%
    \special{ar 1788 100 100 100 0 6.28319}%
    \graphtemp=.5ex\advance\graphtemp by 1.246in
    \rlap{\kern 0.080in\lower\graphtemp\hbox to 0pt{\hss $u'$\hss}}%
    \graphtemp=.5ex\advance\graphtemp by 1.246in
    \rlap{\kern 0.884in\lower\graphtemp\hbox to 0pt{\hss $v'$\hss}}%
    \graphtemp=.5ex\advance\graphtemp by 0.161in
    \rlap{\kern 0.080in\lower\graphtemp\hbox to 0pt{\hss $u$\hss}}%
    \graphtemp=.5ex\advance\graphtemp by 0.161in
    \rlap{\kern 0.884in\lower\graphtemp\hbox to 0pt{\hss $v$\hss}}%
    \graphtemp=.5ex\advance\graphtemp by 1.306in
    \rlap{\kern 1.286in\lower\graphtemp\hbox to 0pt{\hss 1\hss}}%
    \graphtemp=.5ex\advance\graphtemp by 0.904in
    \rlap{\kern 1.286in\lower\graphtemp\hbox to 0pt{\hss 1\hss}}%
    \graphtemp=.5ex\advance\graphtemp by 0.502in
    \rlap{\kern 1.286in\lower\graphtemp\hbox to 0pt{\hss 1\hss}}%
    \graphtemp=.5ex\advance\graphtemp by 0.100in
    \rlap{\kern 1.286in\lower\graphtemp\hbox to 0pt{\hss 1\hss}}%
    \graphtemp=.5ex\advance\graphtemp by 1.306in
    \rlap{\kern 1.768in\lower\graphtemp\hbox to 0pt{\hss 6\hss}}%
    \graphtemp=.5ex\advance\graphtemp by 0.904in
    \rlap{\kern 1.768in\lower\graphtemp\hbox to 0pt{\hss 5\hss}}%
    \graphtemp=.5ex\advance\graphtemp by 0.502in
    \rlap{\kern 1.768in\lower\graphtemp\hbox to 0pt{\hss 6\hss}}%
    \graphtemp=.5ex\advance\graphtemp by 0.100in
    \rlap{\kern 1.768in\lower\graphtemp\hbox to 0pt{\hss 5\hss}}%
    \graphtemp=.5ex\advance\graphtemp by 0.804in
    \rlap{\kern 2.692in\lower\graphtemp\hbox to 0pt{\hss $\bu$\hss}}%
    \graphtemp=.5ex\advance\graphtemp by 0.804in
    \rlap{\kern 2.692in\lower\graphtemp\hbox to 0pt{\hss $\bu$\hss}}%
    \graphtemp=.5ex\advance\graphtemp by 1.205in
    \rlap{\kern 3.496in\lower\graphtemp\hbox to 0pt{\hss $\bu$\hss}}%
    \graphtemp=.5ex\advance\graphtemp by 0.402in
    \rlap{\kern 3.496in\lower\graphtemp\hbox to 0pt{\hss $\bu$\hss}}%
    \graphtemp=.5ex\advance\graphtemp by 1.205in
    \rlap{\kern 4.299in\lower\graphtemp\hbox to 0pt{\hss $\bu$\hss}}%
    \graphtemp=.5ex\advance\graphtemp by 0.603in
    \rlap{\kern 4.299in\lower\graphtemp\hbox to 0pt{\hss $\bu$\hss}}%
    \graphtemp=.5ex\advance\graphtemp by 0.201in
    \rlap{\kern 4.299in\lower\graphtemp\hbox to 0pt{\hss $\bu$\hss}}%
    \special{pn 28}%
    \special{pa 3496 1205}%
    \special{pa 2692 804}%
    \special{fp}%
    \special{pa 2692 804}%
    \special{pa 3496 402}%
    \special{fp}%
    \special{pn 11}%
    \special{pa 4299 603}%
    \special{pa 3496 402}%
    \special{fp}%
    \special{pa 3496 402}%
    \special{pa 4299 201}%
    \special{fp}%
    \special{pa 3496 1205}%
    \special{pa 4299 1205}%
    \special{fp}%
    \special{pn 8}%
    \special{ar 4400 1205 100 100 0 6.28319}%
    \special{ar 4400 603 100 100 0 6.28319}%
    \special{ar 4400 201 100 100 0 6.28319}%
    \graphtemp=.5ex\advance\graphtemp by 0.804in
    \rlap{\kern 2.571in\lower\graphtemp\hbox to 0pt{\hss $z$\hss}}%
    \graphtemp=.5ex\advance\graphtemp by 1.346in
    \rlap{\kern 3.496in\lower\graphtemp\hbox to 0pt{\hss $y$\hss}}%
    \graphtemp=.5ex\advance\graphtemp by 0.261in
    \rlap{\kern 3.496in\lower\graphtemp\hbox to 0pt{\hss $v$\hss}}%
    \graphtemp=.5ex\advance\graphtemp by 1.326in
    \rlap{\kern 4.299in\lower\graphtemp\hbox to 0pt{\hss $x$\hss}}%
    \graphtemp=.5ex\advance\graphtemp by 1.125in
    \rlap{\kern 3.897in\lower\graphtemp\hbox to 0pt{\hss 1\hss}}%
    \graphtemp=.5ex\advance\graphtemp by 0.603in
    \rlap{\kern 3.897in\lower\graphtemp\hbox to 0pt{\hss 1\hss}}%
    \graphtemp=.5ex\advance\graphtemp by 0.201in
    \rlap{\kern 3.897in\lower\graphtemp\hbox to 0pt{\hss 1\hss}}%
    \graphtemp=.5ex\advance\graphtemp by 1.205in
    \rlap{\kern 4.379in\lower\graphtemp\hbox to 0pt{\hss 3\hss}}%
    \graphtemp=.5ex\advance\graphtemp by 0.603in
    \rlap{\kern 4.379in\lower\graphtemp\hbox to 0pt{\hss 4\hss}}%
    \graphtemp=.5ex\advance\graphtemp by 0.201in
    \rlap{\kern 4.379in\lower\graphtemp\hbox to 0pt{\hss 5\hss}}%
    \hbox{\vrule depth1.406in width0pt height 0pt}%
    \kern 4.500in
  }%
}%
}
\caption{Non-reducible: adjacent $\beta'$-vertices, or $\beta$-vertex near
extra $2$-vertex\label{betafig}}
\end{figure}

\vspace{-.5pc}

Graphs formed by adding a pendant edge at each vertex of a $3$-regular graph
contain only configuration {\bf F} among those in Lemma~\ref{83red2}.  Although
its reducibility proof does not require the full flexibility of choosing
weights in it, Example~\ref{nonred} shows that the local argument cannot be
completed when a $\beta'$-vertex has only one $\beta'$-neighbor (of degree $4$).

Example~\ref{nonred} also shows that a $\beta$-vertex is not reducible,
even when its $2$-neighbor has another $2$-neighbor.  Nevertheless, when a
$\beta$-vertex appears in a minimal $2$-bad graph we can guarantee satisfying
all but one specified edge at that vertex.  This is useful when we can ensure
satisfying that edge, such as when its other endpoint has high degree.

\begin{lem}\label{beta2}
Let $v$ be a $\beta$-vertex with $2$-neighbor $z$ in a minimal $2$-bad graph
$G$.  Name vertices so that $N_G(v)=\{z,x,u\}$ and $N_G(z)=\{v,y\}$ (see
Figure~\ref{beta2fig}).  If $G-vz$ has a proper partial $2$-weighting $w'$ that
satisfies $\Gamma_{G}(x)$ and $\Gamma_{G}(y)$, then $G$ has a partial
$2$-weighting $w$ that satisfies the same edges other than $vu$, plus $vz$,
without changing any weights on $G-\{v,z\}$ except possibly on $yz$ and $y$.
\end{lem}

\begin{proof}
Let $G'=G-vz$.  By Lemma~\ref{2config}{\bf A}, $d(y)\geq 2$.  We want to choose
$w(v)$, $w(z)$, and $w(vz)$ to satisfy $\{xv,vz,zy\}$, leaving edges other than
$vu$ satisfied.

Let $a=\rp{y}{z}$.  If $a\ge4$, then setting $w(zv)=1$ ensures satisfying
$yz$, after which we choose $w(v)$ to satisfy $vx$ and $w(z)$ to satisfy $vz$.

If $a=3$ and $d_G(y)=3$, then $w'(y)=1$ ($y=u$ is allowed).  If $w'(yz)=2$,
then we can exchange the weights on $y$ and $yz$ and apply the previous case.
If $w'(yz)=1$, then setting $w(z)=w(zv)=1$ satisfies both $yz$ and $zv$, after
which we choose $w(v)$ to satisfy $vx$.

The remaining case is $d_G(y)=2$; let $N_G(y)=\{z,u'\}$ ($u'=u$ is allowed).
Uncolor $y$ and $yz$.  Setting $w(yz)=1$ and $w(v)=2$ ensures satisfying $zv$.
Now choose $w(vz)$ to satisfy $vx$, $w(y)$ to satisfy $yu'$, and $w(z)$ to
satisfy $yz$.
\end{proof}

\vspace{-1pc}
\begin{figure}[h]
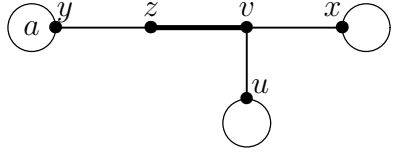

\gpic{
\expandafter\ifx\csname graph\endcsname\relax \csname newbox\endcsname\graph\fi
\expandafter\ifx\csname graphtemp\endcsname\relax \csname newdimen\endcsname\graphtemp\fi
\setbox\graph=\vtop{\vskip 0pt\hbox{%
    \graphtemp=.5ex\advance\graphtemp by 0.125in
    \rlap{\kern 0.250in\lower\graphtemp\hbox to 0pt{\hss $\bu$\hss}}%
    \graphtemp=.5ex\advance\graphtemp by 0.125in
    \rlap{\kern 0.250in\lower\graphtemp\hbox to 0pt{\hss $\bu$\hss}}%
    \graphtemp=.5ex\advance\graphtemp by 0.125in
    \rlap{\kern 0.750in\lower\graphtemp\hbox to 0pt{\hss $\bu$\hss}}%
    \graphtemp=.5ex\advance\graphtemp by 0.125in
    \rlap{\kern 0.750in\lower\graphtemp\hbox to 0pt{\hss $\bu$\hss}}%
    \graphtemp=.5ex\advance\graphtemp by 0.500in
    \rlap{\kern 1.250in\lower\graphtemp\hbox to 0pt{\hss $\bu$\hss}}%
    \graphtemp=.5ex\advance\graphtemp by 0.125in
    \rlap{\kern 1.250in\lower\graphtemp\hbox to 0pt{\hss $\bu$\hss}}%
    \graphtemp=.5ex\advance\graphtemp by 0.125in
    \rlap{\kern 1.750in\lower\graphtemp\hbox to 0pt{\hss $\bu$\hss}}%
    \graphtemp=.5ex\advance\graphtemp by 0.125in
    \rlap{\kern 1.750in\lower\graphtemp\hbox to 0pt{\hss $\bu$\hss}}%
    \special{pn 28}%
    \special{pa 1250 125}%
    \special{pa 750 125}%
    \special{fp}%
    \special{pn 11}%
    \special{pa 250 125}%
    \special{pa 750 125}%
    \special{fp}%
    \special{pa 1250 500}%
    \special{pa 1250 125}%
    \special{fp}%
    \special{pa 1250 125}%
    \special{pa 1750 125}%
    \special{fp}%
    \graphtemp=.5ex\advance\graphtemp by 0.025in
    \rlap{\kern 0.300in\lower\graphtemp\hbox to 0pt{\hss $y$\hss}}%
    \graphtemp=.5ex\advance\graphtemp by 0.025in
    \rlap{\kern 0.750in\lower\graphtemp\hbox to 0pt{\hss $z$\hss}}%
    \graphtemp=.5ex\advance\graphtemp by 0.025in
    \rlap{\kern 1.250in\lower\graphtemp\hbox to 0pt{\hss $v$\hss}}%
    \special{pn 8}%
    \special{ar 125 125 125 125 0 6.28319}%
    \special{ar 1250 625 125 125 0 6.28319}%
    \special{ar 1875 125 125 125 0 6.28319}%
    \graphtemp=.5ex\advance\graphtemp by 0.429in
    \rlap{\kern 1.321in\lower\graphtemp\hbox to 0pt{\hss $u$\hss}}%
    \graphtemp=.5ex\advance\graphtemp by 0.025in
    \rlap{\kern 1.700in\lower\graphtemp\hbox to 0pt{\hss $x$\hss}}%
    \graphtemp=.5ex\advance\graphtemp by 0.125in
    \rlap{\kern 0.125in\lower\graphtemp\hbox to 0pt{\hss $a$\hss}}%
    \hbox{\vrule depth0.750in width0pt height 0pt}%
    \kern 2.000in
  }%
}%
}
\caption{Configuration for Lemma~\ref{beta2}\label{beta2fig}}
\end{figure}

\vspace{-.5pc}

\begin{lem}\label{83red2}
The configurations below are $2$-reducible.

{\bf A}. A $3^-$-vertex having a $1$-neighbor.

{\bf B}. A $4^-$-vertex having two $2^-$-neighbors.

{\bf C}. A $5^+$-vertex $v$ whose number of $2^-$-neighbors is at least
$(d_G(v)-1)/2$.

{\bf D}. Two adjacent $\beta$-vertices.

{\bf E}. A $\beta$-vertex with a $\beta'$-neighbor.

{\bf F.} A $\beta'$-vertex of degree $4$ having two $\beta'$-neighbors of degree $4$.

{\bf G.} A $3$-vertex such that each neighbor is a $\beta$-vertex or is a
$\beta'$-vertex of degree $4$.

\end{lem}

\begin{proof}
Configurations {\bf A}--{\bf C} were shown to be $2$-reducible in
Lemma~\ref{2config}.  For {\bf D}--{\bf G}, as usual we consider a minimal
$2$-bad graph $G$ containing the specified configuration, and the derived graph
$G'$ is obtained by deleting the core, shown in bold in
Figures~\ref{fig2D}--\ref{fig2G}.  In each case we have a proper total
$2$-weighting $w'$ of $G'$ and produce a proper total $2$-weighting $w$ of $G$
by choosing weights on the deleted edges and on their endpoints, leaving all
other weights as in $w'$, with the possible exception of applying Lemma 3.3.
For each successive configuration, we know that the earlier configurations do
not occur in $G$.

\smallskip
\textbf{Case D:} {\it $v$ and $v'$ are adjacent $\beta$-vertices.}
As shown in Figure~\ref{fig2D}D, $v$ and $v'$ have degree $3$, with
$2$-neighbors $z$ and $z'$, respectively.  Let $N_G(v)=\{z,v',x\}$ and
$N_G(v')=\{z',v,x'\}$, also $N_G(z)=\{v,y\}$ and $N_G(z')=\{v,y'\}$ ($y=y'$
and/or $x=x'$ are allowed in the argument).  By Lemma~\ref{triangle}, we may
assume $z\ne z'$, $y\ne x$, and $y'\ne x'$.  Let $G'=G-\{zv,vv',v'z'\}$.  

Consider first the degenerate case $zz'\in E(G)$, so $y=z'$ and $y'=z$.
This also handles the case $y=x'$ or $y'=x$ under appropriate relabeling.
Set $w(zz')=1$ and $w(vv')=2$ to ensure satisfying $zv$ and $z'v'$.
Set $w(z)=w(zv)=2$.  Now choose $w(v)$ to satisfy $zv$, choose $w(v')$ and
$w(v'z')$ to satisfy $vv'$ and $v'x'$, and choose $w(z')$ to satisfy $zz'$.

Hence we may assume that the vertices are distinct as on the left in
Figure~\ref{fig2D}D.  Let $a=w'(zy)$, $b=\rp{y}{z}$, $c=w'(vx)$, and
$d=\rp{x}{v}$.  Define $a',b',c',d'$ analogously using $y',z',v',x'$
In all subcases, set $w(vv')=2$.

{\bf Subcase 1:} {\it $d_G(y)=2$ (or $d_G(y')=2$).}  Let $N_G(y)=\{z,u\}$.
Uncolor $u$ and $yu$.  Treat $y=z'$ (which implies $u=v'$) as a special case.
When $y=z'$, set $w(zz')=w(z')=1$; in general, set $w(z)=w(z')=1$.  In both
cases, this ensures satisfying $vz$ and $v'z'$ (since $w(vv')=2$).  Now set
$w(v'z')=1$ when $y=z'$; otherwise choose $w(v'z')$ to satisfy $y'z'$.  In both
cases, next choose $w(v')$ to satisfy $v'x'$, then $w(v)$ and $w(vz)$ to
satisfy $vv'$ and $vx$.  Finally, choose $w(z)$ to satisfy $zz'$ when $y=z'$;
otherwise, choose $w(u)$ to satisfy $yu$ and then $w(yu)$ to satisfy the other
edge at $u$.

{\bf Subcase 2:} {\it $d_G(y)\ne2$}.  By $\overline{A}$, we may assume
$d(y)\ge3$.
If $c=2$ or $a=1$, then $w(vv')=2$ ensures satisfying $zv$.  Set $w(v')=2$
to guarantee satisfying $z'v'$.  Now choose $w(z'v')$ to satisfy $v'x'$, and
choose $w(z')$ to satisfy $z'y'$.  Next choose $w(zv)$ and $w(v)$ to satisfy
$vx$ and $vv'$.  Finally, choose $w(z)$ to satisfy $zy$.

We may therefore assume $c=1$ and $a=2$, and by symmetry $c'=1$ and $a'=2$.
We may also assume $w'(y)=w'(y')=2$, since otherwise we can switch weights on
$y$ and $yz$ (or on $y'$ and $y'z'$), which leaves the other edges at $y$ or
$y'$ satisfied and yields the subcase above.

With $d_G(y)\ge3$, we have $b\ge4$ (since $w'(y)=2$).  By symmetry,
$d_G(y')\ge3$ and $b'\geq 4$.  Now setting $w(z)=w(z')=1$ ensures satisfying
$\Gamma_G(z)$ and $\Gamma_G(z')$ (since $w(vv')=2$).  Finally, choose
$w(zv)+w(v)$ to avoid $d-2$ and and $w(z'v')+w(v')$ to avoid $d'-2$ (allowing
two choices for each sum) so that the sums are different.  This satisfies $vx$,
$v'x'$, and $vv'$.

\begin{figure}[h]
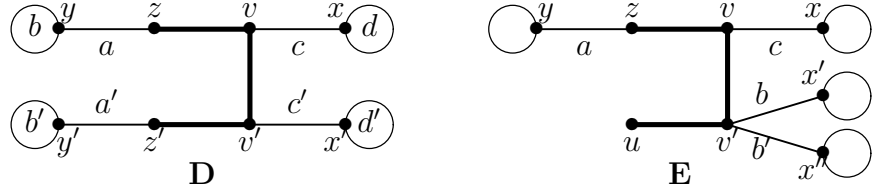

\gpic{
\expandafter\ifx\csname graph\endcsname\relax \csname newbox\endcsname\graph\fi
\expandafter\ifx\csname graphtemp\endcsname\relax \csname newdimen\endcsname\graphtemp\fi
\setbox\graph=\vtop{\vskip 0pt\hbox{%
    \graphtemp=.5ex\advance\graphtemp by 0.625in
    \rlap{\kern 0.250in\lower\graphtemp\hbox to 0pt{\hss $\bu$\hss}}%
    \graphtemp=.5ex\advance\graphtemp by 0.125in
    \rlap{\kern 0.250in\lower\graphtemp\hbox to 0pt{\hss $\bu$\hss}}%
    \graphtemp=.5ex\advance\graphtemp by 0.625in
    \rlap{\kern 0.750in\lower\graphtemp\hbox to 0pt{\hss $\bu$\hss}}%
    \graphtemp=.5ex\advance\graphtemp by 0.125in
    \rlap{\kern 0.750in\lower\graphtemp\hbox to 0pt{\hss $\bu$\hss}}%
    \graphtemp=.5ex\advance\graphtemp by 0.625in
    \rlap{\kern 1.250in\lower\graphtemp\hbox to 0pt{\hss $\bu$\hss}}%
    \graphtemp=.5ex\advance\graphtemp by 0.125in
    \rlap{\kern 1.250in\lower\graphtemp\hbox to 0pt{\hss $\bu$\hss}}%
    \graphtemp=.5ex\advance\graphtemp by 0.625in
    \rlap{\kern 1.750in\lower\graphtemp\hbox to 0pt{\hss $\bu$\hss}}%
    \graphtemp=.5ex\advance\graphtemp by 0.125in
    \rlap{\kern 1.750in\lower\graphtemp\hbox to 0pt{\hss $\bu$\hss}}%
    \special{pn 28}%
    \special{pa 750 625}%
    \special{pa 1250 625}%
    \special{fp}%
    \special{pa 1250 625}%
    \special{pa 1250 125}%
    \special{fp}%
    \special{pa 1250 125}%
    \special{pa 750 125}%
    \special{fp}%
    \special{pn 11}%
    \special{pa 250 625}%
    \special{pa 750 625}%
    \special{fp}%
    \special{pa 250 125}%
    \special{pa 750 125}%
    \special{fp}%
    \special{pa 1250 625}%
    \special{pa 1750 625}%
    \special{fp}%
    \special{pa 1250 125}%
    \special{pa 1750 125}%
    \special{fp}%
    \graphtemp=.5ex\advance\graphtemp by 0.725in
    \rlap{\kern 0.300in\lower\graphtemp\hbox to 0pt{\hss $y'$\hss}}%
    \graphtemp=.5ex\advance\graphtemp by 0.025in
    \rlap{\kern 0.300in\lower\graphtemp\hbox to 0pt{\hss $y$\hss}}%
    \graphtemp=.5ex\advance\graphtemp by 0.725in
    \rlap{\kern 0.750in\lower\graphtemp\hbox to 0pt{\hss $z'$\hss}}%
    \graphtemp=.5ex\advance\graphtemp by 0.025in
    \rlap{\kern 0.750in\lower\graphtemp\hbox to 0pt{\hss $z$\hss}}%
    \graphtemp=.5ex\advance\graphtemp by 0.725in
    \rlap{\kern 1.250in\lower\graphtemp\hbox to 0pt{\hss $v'$\hss}}%
    \graphtemp=.5ex\advance\graphtemp by 0.025in
    \rlap{\kern 1.250in\lower\graphtemp\hbox to 0pt{\hss $v$\hss}}%
    \graphtemp=.5ex\advance\graphtemp by 0.725in
    \rlap{\kern 1.700in\lower\graphtemp\hbox to 0pt{\hss $x'$\hss}}%
    \graphtemp=.5ex\advance\graphtemp by 0.025in
    \rlap{\kern 1.700in\lower\graphtemp\hbox to 0pt{\hss $x$\hss}}%
    \special{pn 8}%
    \special{ar 125 625 125 125 0 6.28319}%
    \special{ar 125 125 125 125 0 6.28319}%
    \special{ar 1875 625 125 125 0 6.28319}%
    \special{ar 1875 125 125 125 0 6.28319}%
    \graphtemp=.5ex\advance\graphtemp by 0.625in
    \rlap{\kern 0.125in\lower\graphtemp\hbox to 0pt{\hss $b'$\hss}}%
    \graphtemp=.5ex\advance\graphtemp by 0.125in
    \rlap{\kern 0.125in\lower\graphtemp\hbox to 0pt{\hss $b$\hss}}%
    \graphtemp=.5ex\advance\graphtemp by 0.225in
    \rlap{\kern 0.500in\lower\graphtemp\hbox to 0pt{\hss $a$\hss}}%
    \graphtemp=.5ex\advance\graphtemp by 0.525in
    \rlap{\kern 0.500in\lower\graphtemp\hbox to 0pt{\hss $a'$\hss}}%
    \graphtemp=.5ex\advance\graphtemp by 0.625in
    \rlap{\kern 1.875in\lower\graphtemp\hbox to 0pt{\hss $d'$\hss}}%
    \graphtemp=.5ex\advance\graphtemp by 0.125in
    \rlap{\kern 1.875in\lower\graphtemp\hbox to 0pt{\hss $d$\hss}}%
    \graphtemp=.5ex\advance\graphtemp by 0.225in
    \rlap{\kern 1.500in\lower\graphtemp\hbox to 0pt{\hss $c$\hss}}%
    \graphtemp=.5ex\advance\graphtemp by 0.525in
    \rlap{\kern 1.500in\lower\graphtemp\hbox to 0pt{\hss $c'$\hss}}%
    \graphtemp=.5ex\advance\graphtemp by 0.775in
    \rlap{\kern 4.250in\lower\graphtemp\hbox to 0pt{\hss $\bu$\hss}}%
    \graphtemp=.5ex\advance\graphtemp by 0.125in
    \rlap{\kern 2.750in\lower\graphtemp\hbox to 0pt{\hss $\bu$\hss}}%
    \graphtemp=.5ex\advance\graphtemp by 0.625in
    \rlap{\kern 3.250in\lower\graphtemp\hbox to 0pt{\hss $\bu$\hss}}%
    \graphtemp=.5ex\advance\graphtemp by 0.125in
    \rlap{\kern 3.250in\lower\graphtemp\hbox to 0pt{\hss $\bu$\hss}}%
    \graphtemp=.5ex\advance\graphtemp by 0.625in
    \rlap{\kern 3.750in\lower\graphtemp\hbox to 0pt{\hss $\bu$\hss}}%
    \graphtemp=.5ex\advance\graphtemp by 0.125in
    \rlap{\kern 3.750in\lower\graphtemp\hbox to 0pt{\hss $\bu$\hss}}%
    \graphtemp=.5ex\advance\graphtemp by 0.475in
    \rlap{\kern 4.250in\lower\graphtemp\hbox to 0pt{\hss $\bu$\hss}}%
    \graphtemp=.5ex\advance\graphtemp by 0.125in
    \rlap{\kern 4.250in\lower\graphtemp\hbox to 0pt{\hss $\bu$\hss}}%
    \special{pn 28}%
    \special{pa 3250 625}%
    \special{pa 3750 625}%
    \special{fp}%
    \special{pa 3750 625}%
    \special{pa 3750 125}%
    \special{fp}%
    \special{pa 3750 125}%
    \special{pa 3250 125}%
    \special{fp}%
    \special{pn 11}%
    \special{pa 4250 775}%
    \special{pa 3750 625}%
    \special{fp}%
    \special{pa 2750 125}%
    \special{pa 3250 125}%
    \special{fp}%
    \special{pa 3750 625}%
    \special{pa 4250 475}%
    \special{fp}%
    \special{pa 3750 125}%
    \special{pa 4250 125}%
    \special{fp}%
    \graphtemp=.5ex\advance\graphtemp by 0.875in
    \rlap{\kern 4.200in\lower\graphtemp\hbox to 0pt{\hss $x''$\hss}}%
    \graphtemp=.5ex\advance\graphtemp by 0.025in
    \rlap{\kern 2.800in\lower\graphtemp\hbox to 0pt{\hss $y$\hss}}%
    \graphtemp=.5ex\advance\graphtemp by 0.725in
    \rlap{\kern 3.250in\lower\graphtemp\hbox to 0pt{\hss $u$\hss}}%
    \graphtemp=.5ex\advance\graphtemp by 0.025in
    \rlap{\kern 3.250in\lower\graphtemp\hbox to 0pt{\hss $z$\hss}}%
    \graphtemp=.5ex\advance\graphtemp by 0.725in
    \rlap{\kern 3.750in\lower\graphtemp\hbox to 0pt{\hss $v'$\hss}}%
    \graphtemp=.5ex\advance\graphtemp by 0.025in
    \rlap{\kern 3.750in\lower\graphtemp\hbox to 0pt{\hss $v$\hss}}%
    \graphtemp=.5ex\advance\graphtemp by 0.375in
    \rlap{\kern 4.200in\lower\graphtemp\hbox to 0pt{\hss $x'$\hss}}%
    \graphtemp=.5ex\advance\graphtemp by 0.025in
    \rlap{\kern 4.200in\lower\graphtemp\hbox to 0pt{\hss $x$\hss}}%
    \special{pn 8}%
    \special{ar 2625 125 125 125 0 6.28319}%
    \special{ar 4375 775 125 125 0 6.28319}%
    \special{ar 4375 475 125 125 0 6.28319}%
    \special{ar 4375 125 125 125 0 6.28319}%
    \graphtemp=.5ex\advance\graphtemp by 0.225in
    \rlap{\kern 3.000in\lower\graphtemp\hbox to 0pt{\hss $a$\hss}}%
    \graphtemp=.5ex\advance\graphtemp by 0.225in
    \rlap{\kern 4.000in\lower\graphtemp\hbox to 0pt{\hss $c$\hss}}%
    \graphtemp=.5ex\advance\graphtemp by 0.479in
    \rlap{\kern 3.929in\lower\graphtemp\hbox to 0pt{\hss $b$\hss}}%
    \graphtemp=.5ex\advance\graphtemp by 0.771in
    \rlap{\kern 3.929in\lower\graphtemp\hbox to 0pt{\hss $b'$\hss}}%
    \graphtemp=.5ex\advance\graphtemp by 0.900in
    \rlap{\kern 1.000in\lower\graphtemp\hbox to 0pt{\hss {\bf D}\hss}}%
    \graphtemp=.5ex\advance\graphtemp by 0.900in
    \rlap{\kern 3.500in\lower\graphtemp\hbox to 0pt{\hss {\bf E}\hss}}%
    \hbox{\vrule depth0.900in width0pt height 0pt}%
    \kern 4.500in
  }%
}%
}

\vspace{-1pc}
\caption{Cases {\bf D} and {\bf E} for Lemma~\ref{83red2}\label{fig2D}}
\end{figure}

\smallskip
\textbf{Case E:} {\it $v$ is a $\beta$-vertex with a $\beta'$-neighbor $v'$.}
Let $z$ be the $2$-neighbor of $v$, with $N_G(v)=\{x,z,v'\}$ and
$N_G(z)=\{y,v\}$.

If $d_G(v')\ge6$, then let $U$ be the set of $1$-neighbors of $v'$.  Set
$w(v')=2$ to guarantee satisfying $vv'$ (since now
$\rw{v'}{v}\ge 7>6\ge\rw{v}{v'}$).  Including $vv'$, the number
of remaining edges incident to $v'$ with unchosen weights equals the number of
edges from $v'$ to $N_G(v')-U$; by Lemma~\ref{choose}, we can choose these
weights to satisfy these edges.  The edges of $[v,U]$ are automatically
satisfied, regardless of the weights on $U$.  Finally, now that $w(vv')$ is
chosen, Lemma~\ref{beta2} allows us to complete $w$ to a proper
$2$-weighting of $G$.

We may therefore assume $d_G(v')=4$, as in Figure~\ref{fig2D}E.  Let  $u$ be
the $1$-neighbor of $v'$.  Let $G'=G-\{zv,vv',v'u\}$, leaving $v'x',v'x''\in
E(G')$.  Let $a=w'(zy)$, $b=w'(v'x')$, $b'=w'(v'x'')$, and $c=w'(vx)$.
The argument allows $y\in\{x',x''\}$.  Fix $w(u)=1$.

If $b+b'-c\geq 2$, then requiring $w(v)+w(zv)=3$ guarantees satisfying $v'v$.
Next choose $w(v'v)$ to satisfy $vx$, and choose $w(v')$ and $w(v'u)$ to satisfy
$v'x'$ and $v'x''$.  With $w(v)=3-w(zv)$, there are three choices for
$w(zv)+w(z)$, so we can choose $w(zv)$ and $w(z)$ with
$w(z)+w(zv)\ne \rp yz$ to satisfy $yz$ and $w(z)+a\ne 3-w(zv)+c+w(vv')$ to
satisfy $zv$.  We may therefore assume $b+b'-c\leq 1$.

If $c=2$ or $a=1$, then requiring $w(v)+w(vv')=3$ guarantees satisfying $zv$.
Now choose $w(zv)$ to satisfy $vx$ and then $w(z)$ to satisfy $zy$.
Finally, tentatively set $w(vv')=2$ and $w(v)=1$, and then choose $w(v')$
and $w(v'u)$ to satisfy $v'x'$ and $v'x''$.  If $vv'$ is not now satisfied,
then $w(v')+w(v'u)<4$.  Now exchange weights on $vv'$ and $v$ while increasing
$w(v')$ or $w(v'u)$ to satisfy $vv'$ and preserve the satisfaction of
$v'x'$ and $v'x''$.

Hence we may assume $c=1$ and $a=2$.  Since also $b+b'-c\leq 1$, we have
$b=b'=1$.  Tentatively set $w(v'u)=2$, and choose $w(v')$ and $w(vv')$ to
satisfy $v'x'$ and $v'x''$, with $w(v')\ge w(vv')$.  If $w(v')=2$, then 
$vv'$ is automatically satisfied (since $c=1$), and Lemma~\ref{beta2}
completes the extension to $w$.  If $w(v')=1$ and the application of
Lemma~\ref{beta2} produces $w(zv)=w(v)=2$, then $vv'$ is not satisfied.
In this case, $\rp{x'}{v'},\rp{x''}{v'}=\{6,7\}$, and changing $w(v'u)$ to
$1$ satisfies $vv'$ while still satisfying $v'x'$ and $v'x''$.

%
%

\smallskip
\textbf{Case F:} {\it $v$ is a $\beta'$-vertex of degree $4$ with
$\beta'$-neighbors $z$ and $z'$ of degree $4$.} Let $N_G(v)=\{x,u,z,z'\}$.  Let
$u,y,y'$ be the $1$-neighbors of $v,z,z'$, respectively.  (see
Figure~\ref{fig2F}).

{\bf Subcase 1:} {$zz'\in E(G)$}.  In this case, we have a triangle of
$\beta'$-vertices having degree $4$.  Let $G'=G-\{vz,vz',zz',vu,zy,z'y'\}$.
Let $t$ and $t'$ be the remaining neighbors of $z$ and $z'$, respectively.
By symmetry, we need only consider two cases: $w'(zt)\ne w'(z't')$, and the
case $w'(zt)=w'(z't')=w'(vx)=c$.

If $w'(zt)\ne w'(z't')$, then by symmetry we may assume $w'(zt)=1$ and
$w'(z't')=2$.  Set $w(zz')=w(z')=w(z'v)=2$ and $w(zv)=w(v)=1$ to ensure
satisfying $zz'$ and $z'v$.  Now choose $w(vu)$ to satisfy $\Gamma_{G'}(v)$ and
choose $w(z'y')$ to satisfy $\Gamma_{G'}(z')$.  Finally, choose $w(z)$ and
$w(zy)$ to satisfy $zv$ and $\Gamma_{G'}(z)$.

In the other case, let $w(vu)=w(v)=w(vz)=w(vz')=a$.  Choose $a$ to satisfy
$vx$.  Let $w(zz')=3-a$; this ensures satisfying $vz$ and $vz'$.  With
$w(z')$ arbitrary, choose $w(z'y')$ to satisfy $z't'$, and finally choose
$w(z)$ and $w(zy)$ to satisfy $zz'$ and $zt$.

\begin{figure}[h]
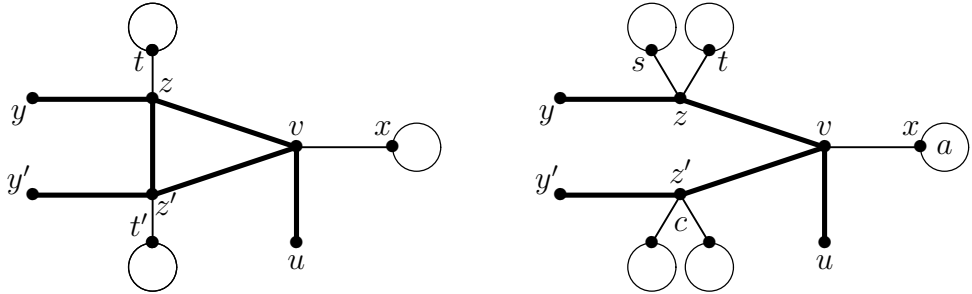

\gpic{
\expandafter\ifx\csname graph\endcsname\relax \csname newbox\endcsname\graph\fi
\expandafter\ifx\csname graphtemp\endcsname\relax \csname newdimen\endcsname\graphtemp\fi
\setbox\graph=\vtop{\vskip 0pt\hbox{%
    \graphtemp=.5ex\advance\graphtemp by 1.005in
    \rlap{\kern 0.101in\lower\graphtemp\hbox to 0pt{\hss $\bu$\hss}}%
    \graphtemp=.5ex\advance\graphtemp by 0.503in
    \rlap{\kern 0.101in\lower\graphtemp\hbox to 0pt{\hss $\bu$\hss}}%
    \graphtemp=.5ex\advance\graphtemp by 1.256in
    \rlap{\kern 0.729in\lower\graphtemp\hbox to 0pt{\hss $\bu$\hss}}%
    \graphtemp=.5ex\advance\graphtemp by 1.256in
    \rlap{\kern 0.729in\lower\graphtemp\hbox to 0pt{\hss $\bu$\hss}}%
    \graphtemp=.5ex\advance\graphtemp by 0.251in
    \rlap{\kern 0.729in\lower\graphtemp\hbox to 0pt{\hss $\bu$\hss}}%
    \graphtemp=.5ex\advance\graphtemp by 0.251in
    \rlap{\kern 0.729in\lower\graphtemp\hbox to 0pt{\hss $\bu$\hss}}%
    \graphtemp=.5ex\advance\graphtemp by 1.005in
    \rlap{\kern 0.729in\lower\graphtemp\hbox to 0pt{\hss $\bu$\hss}}%
    \graphtemp=.5ex\advance\graphtemp by 0.503in
    \rlap{\kern 0.729in\lower\graphtemp\hbox to 0pt{\hss $\bu$\hss}}%
    \graphtemp=.5ex\advance\graphtemp by 0.754in
    \rlap{\kern 1.482in\lower\graphtemp\hbox to 0pt{\hss $\bu$\hss}}%
    \special{pn 28}%
    \special{pa 101 1005}%
    \special{pa 729 1005}%
    \special{fp}%
    \special{pa 729 1005}%
    \special{pa 1482 754}%
    \special{fp}%
    \special{pa 1482 754}%
    \special{pa 729 503}%
    \special{fp}%
    \special{pa 729 503}%
    \special{pa 101 503}%
    \special{fp}%
    \special{pn 11}%
    \special{pa 729 1256}%
    \special{pa 729 1005}%
    \special{fp}%
    \special{pa 729 251}%
    \special{pa 729 503}%
    \special{fp}%
    \special{pn 28}%
    \special{pa 729 1005}%
    \special{pa 729 503}%
    \special{fp}%
    \graphtemp=.5ex\advance\graphtemp by 0.934in
    \rlap{\kern 0.029in\lower\graphtemp\hbox to 0pt{\hss $y'$\hss}}%
    \graphtemp=.5ex\advance\graphtemp by 0.574in
    \rlap{\kern 0.029in\lower\graphtemp\hbox to 0pt{\hss $y$\hss}}%
    \graphtemp=.5ex\advance\graphtemp by 1.076in
    \rlap{\kern 0.800in\lower\graphtemp\hbox to 0pt{\hss $z'$\hss}}%
    \graphtemp=.5ex\advance\graphtemp by 0.431in
    \rlap{\kern 0.800in\lower\graphtemp\hbox to 0pt{\hss $z$\hss}}%
    \graphtemp=.5ex\advance\graphtemp by 0.653in
    \rlap{\kern 1.482in\lower\graphtemp\hbox to 0pt{\hss $v$\hss}}%
    \special{pn 8}%
    \special{ar 729 1382 126 126 0 6.28319}%
    \special{ar 729 1382 126 126 0 6.28319}%
    \special{ar 729 126 126 126 0 6.28319}%
    \special{ar 729 126 126 126 0 6.28319}%
    \graphtemp=.5ex\advance\graphtemp by 1.256in
    \rlap{\kern 1.482in\lower\graphtemp\hbox to 0pt{\hss $\bu$\hss}}%
    \graphtemp=.5ex\advance\graphtemp by 0.754in
    \rlap{\kern 1.985in\lower\graphtemp\hbox to 0pt{\hss $\bu$\hss}}%
    \special{pn 28}%
    \special{pa 1482 1256}%
    \special{pa 1482 754}%
    \special{fp}%
    \special{pn 11}%
    \special{pa 1985 754}%
    \special{pa 1482 754}%
    \special{fp}%
    \special{pn 8}%
    \special{ar 2111 754 126 126 0 6.28319}%
    \graphtemp=.5ex\advance\graphtemp by 1.357in
    \rlap{\kern 1.482in\lower\graphtemp\hbox to 0pt{\hss $u$\hss}}%
    \graphtemp=.5ex\advance\graphtemp by 0.653in
    \rlap{\kern 1.935in\lower\graphtemp\hbox to 0pt{\hss $x$\hss}}%
    \graphtemp=.5ex\advance\graphtemp by 0.754in
    \rlap{\kern 2.111in\lower\graphtemp\hbox to 0pt{\hss $~$\hss}}%
    \graphtemp=.5ex\advance\graphtemp by 0.322in
    \rlap{\kern 0.658in\lower\graphtemp\hbox to 0pt{\hss $t$\hss}}%
    \graphtemp=.5ex\advance\graphtemp by 0.322in
    \rlap{\kern 0.800in\lower\graphtemp\hbox to 0pt{\hss $~$\hss}}%
    \graphtemp=.5ex\advance\graphtemp by 1.185in
    \rlap{\kern 0.658in\lower\graphtemp\hbox to 0pt{\hss $t'$\hss}}%
    \graphtemp=.5ex\advance\graphtemp by 1.005in
    \rlap{\kern 2.864in\lower\graphtemp\hbox to 0pt{\hss $\bu$\hss}}%
    \graphtemp=.5ex\advance\graphtemp by 0.503in
    \rlap{\kern 2.864in\lower\graphtemp\hbox to 0pt{\hss $\bu$\hss}}%
    \graphtemp=.5ex\advance\graphtemp by 1.256in
    \rlap{\kern 3.342in\lower\graphtemp\hbox to 0pt{\hss $\bu$\hss}}%
    \graphtemp=.5ex\advance\graphtemp by 1.256in
    \rlap{\kern 3.643in\lower\graphtemp\hbox to 0pt{\hss $\bu$\hss}}%
    \graphtemp=.5ex\advance\graphtemp by 0.251in
    \rlap{\kern 3.342in\lower\graphtemp\hbox to 0pt{\hss $\bu$\hss}}%
    \graphtemp=.5ex\advance\graphtemp by 0.251in
    \rlap{\kern 3.643in\lower\graphtemp\hbox to 0pt{\hss $\bu$\hss}}%
    \graphtemp=.5ex\advance\graphtemp by 1.005in
    \rlap{\kern 3.492in\lower\graphtemp\hbox to 0pt{\hss $\bu$\hss}}%
    \graphtemp=.5ex\advance\graphtemp by 0.503in
    \rlap{\kern 3.492in\lower\graphtemp\hbox to 0pt{\hss $\bu$\hss}}%
    \graphtemp=.5ex\advance\graphtemp by 0.754in
    \rlap{\kern 4.246in\lower\graphtemp\hbox to 0pt{\hss $\bu$\hss}}%
    \special{pn 28}%
    \special{pa 2864 1005}%
    \special{pa 3492 1005}%
    \special{fp}%
    \special{pa 3492 1005}%
    \special{pa 4246 754}%
    \special{fp}%
    \special{pa 4246 754}%
    \special{pa 3492 503}%
    \special{fp}%
    \special{pa 3492 503}%
    \special{pa 2864 503}%
    \special{fp}%
    \special{pn 11}%
    \special{pa 3342 1256}%
    \special{pa 3492 1005}%
    \special{fp}%
    \special{pa 3492 1005}%
    \special{pa 3643 1256}%
    \special{fp}%
    \special{pa 3342 251}%
    \special{pa 3492 503}%
    \special{fp}%
    \special{pa 3492 503}%
    \special{pa 3643 251}%
    \special{fp}%
    \graphtemp=.5ex\advance\graphtemp by 0.934in
    \rlap{\kern 2.793in\lower\graphtemp\hbox to 0pt{\hss $y'$\hss}}%
    \graphtemp=.5ex\advance\graphtemp by 0.574in
    \rlap{\kern 2.793in\lower\graphtemp\hbox to 0pt{\hss $y$\hss}}%
    \graphtemp=.5ex\advance\graphtemp by 0.905in
    \rlap{\kern 3.492in\lower\graphtemp\hbox to 0pt{\hss $z'$\hss}}%
    \graphtemp=.5ex\advance\graphtemp by 0.603in
    \rlap{\kern 3.492in\lower\graphtemp\hbox to 0pt{\hss $z$\hss}}%
    \graphtemp=.5ex\advance\graphtemp by 0.653in
    \rlap{\kern 4.246in\lower\graphtemp\hbox to 0pt{\hss $v$\hss}}%
    \special{pn 8}%
    \special{ar 3342 1382 126 126 0 6.28319}%
    \special{ar 3643 1382 126 126 0 6.28319}%
    \special{ar 3342 126 126 126 0 6.28319}%
    \special{ar 3643 126 126 126 0 6.28319}%
    \graphtemp=.5ex\advance\graphtemp by 1.256in
    \rlap{\kern 4.246in\lower\graphtemp\hbox to 0pt{\hss $\bu$\hss}}%
    \graphtemp=.5ex\advance\graphtemp by 0.754in
    \rlap{\kern 4.749in\lower\graphtemp\hbox to 0pt{\hss $\bu$\hss}}%
    \special{pn 28}%
    \special{pa 4246 1256}%
    \special{pa 4246 754}%
    \special{fp}%
    \special{pn 11}%
    \special{pa 4749 754}%
    \special{pa 4246 754}%
    \special{fp}%
    \special{pn 8}%
    \special{ar 4874 754 126 126 0 6.28319}%
    \graphtemp=.5ex\advance\graphtemp by 1.357in
    \rlap{\kern 4.246in\lower\graphtemp\hbox to 0pt{\hss $u$\hss}}%
    \graphtemp=.5ex\advance\graphtemp by 0.653in
    \rlap{\kern 4.698in\lower\graphtemp\hbox to 0pt{\hss $x$\hss}}%
    \graphtemp=.5ex\advance\graphtemp by 0.754in
    \rlap{\kern 4.874in\lower\graphtemp\hbox to 0pt{\hss $a$\hss}}%
    \graphtemp=.5ex\advance\graphtemp by 0.322in
    \rlap{\kern 3.271in\lower\graphtemp\hbox to 0pt{\hss $s$\hss}}%
    \graphtemp=.5ex\advance\graphtemp by 0.322in
    \rlap{\kern 3.714in\lower\graphtemp\hbox to 0pt{\hss $t$\hss}}%
    \graphtemp=.5ex\advance\graphtemp by 1.156in
    \rlap{\kern 3.492in\lower\graphtemp\hbox to 0pt{\hss $c$\hss}}%
    \hbox{\vrule depth1.508in width0pt height 0pt}%
    \kern 5.000in
  }%
}%
}
\caption{Cases {\bf F1} and {\bf F2} for Lemma~\ref{83red2}\label{fig2F}}
\end{figure}

{\bf Subcase 2:} {\it $zz'\notin E(G)$}.  Let $G'=G-\{yz,zv,vu,vz',z'y'\}$.
Let $a=\rp{x}{v}$.  If $a\ne 6$, then requiring $w(v)+w(vz')=3$ and
$w(vz)+w(vu)=3$ satisfies $xv$, with $w(vz')$ and $w(vz)$ still choosable
freely.  Choose $w(zy)$, $w(z)$, and $w(vz)$ so that their sum avoids
$\{\rp{s}{z},\rp{t}{z}\}$, where $\{s,t\}=N_{G'}(z)$, and so that
$w(zy)+w(z)+w'(sz)+w'(tz)\ne 3-w(vz)+w(v)+w(vz')+w'(vx)$ (to satisfy $vz$).
Since $w(v)+w(vz')=3$, there are three constants for $w(zy)+w(z)+w(vz)$ to
avoid, so such a choice exists.  Finally, choose $w(vz')$, $w(z')$, and
$w(z'y')$ to satisfy $vz$ and $\Gamma_{G'}(z)$, again making their sum avoid
three known values.

Hence we may assume $a=6$.  Now choosing $w(v)=w(vu)=w(vz)$ guarantees
satisfying $vx$; let $b$ denote the value to be chosen for them.
Let $c$ be the total weight assigned by $w'$ to $\Gamma_{G'}(z')$.
If $c=2$, or if $c=3$ and $w(vx)=2$, then let $b=2$.  Otherwise, let $b=1$.
In either case, $vz'$ is guaranteed to be satisfied.  Finally, choose $w(z)$
and $w(zy)$ to satisfy $zs$ and $zt$, choose $w(vz')$ to satisfy $vz$, and 
choose $w(z')$ and $w(z'y')$ to satisfy $\Gamma_{G'}(z')$.

\smallskip
\textbf{Case G:} {\it $v$ is a $3$-vertex with neighbors $z_1,z_2,z_3$ (where
$d_G(z_1)\ge d_G(z_2)\ge d_G(z_3)$) such that each is a $\beta$-vertex or is a
$\beta'$-vertex of degree $4$.}
For $i\in\{1,2,3\}$, let $y_i$ be the neighbor of $z_i$ with degree
$5-d_G(z_i)$.  When $d_G(z_i)=3$, let $x_i$ be the other neighbor of $z_i$ and
let $y'_i$ be the other neighbor of $y_i$.  When $d_G(z_i)=4$, let
$\{x_i,x'_i\}=N_G(z_i)-\{v,y_i\}$.

We first reduce to the case where $N_G(v)$ is independent.  Adjacent
$\beta$-vertices are forbidden by $\bD$.  Adjacent $\beta$- and
$\beta'$-vertices are forbidden by $\bE$.

The third possibility is that $z$ and $z'$ are adjacent $\beta'$-vertices
having a common $3$-neighbor $v$.  The situation is illustrated by deleting $u$
from the left graph in Figure~\ref{fig2F}.  Label the vertices as described
there, with $G'=G-\{vz,vz',zz',zy,z'y'\}$.  If $w'(zt)=2$ or $w'(vx)=1$, then
set $w(vz)=w(vz')=1$ and $w(z')=w(zz')=2$ to ensure satisfying $vz$ and $vz'$.
Choose $w(v)$ to satisfy $vx$, choose $w(z'y')$ to satisfy $z't'$, and
choose $w(z)$ and $w(zy)$ to satisfy $zt$ and $zz'$.

By symmetry, we may now assume $w'(zt)=w'(z't')=1$ and $w'(vx)=2$.  Also,
$w'(x)=2$, or we can switch the weights on $x$ and $vx$ to reach the case just
discussed.  Now, using $d(x)\ge3$ (since $x$ is a $\beta$- or $\beta'$-neighbor
of $v$), we have $\rp xv\ge4$.  Now set $w(zv)=w(v)=w(vz')=1$ and $w(zz')=2$
to satisfy $vx$ and ensure satisfying $vz$ and $vz'$.  Finally, set
$w(z')=1$, choose $w(z'y')$ to satisfy $z't'$, and choose $w(z)$ and $w(zy)$ to
satisfy $zt$ and $zz'$.

Hence we may assume that $N_G(v)$ is independent.  If two $\beta$-vertices in
$N_G(v)$ have a common $2$-neighbor, say $z_1$ and $z_2$ with common
$2$-neighbor $y$, then let $G'=G-\{vz_1,vz_2,z_1y,z_2y\}$.  Set
$w(vz_1)=w(vz_2)=2$ and $w(y)=1$ to ensure satisfying $z_1y$ and $z_2y$.
Now choose $w(v)$ to satisfy $vz_3$, choose $w(z_1)$ and $w(z_1y)$ to satisfy
$vz_1$ and $\Gamma_{G'}(z_1)$, and choose $w(z_2)$ and $w(z_2y)$ to satisfy
$vz_2$ and $\Gamma_{G'}(z_2)$.

Now $N_G(v)$ independent and the $2$-neighbors of $\beta$-neighbors of $v$ are
distinct.  The remaining cases are shown in Figure~\ref{fig2G}.  The argument
does not require the vertices on circles to be distinct.  Let Subcase $j$ be
the situation where $j$ neighbors of $v$ are $\beta'$-vertices.  In each
subcase, the deleted core consists of $\Gamma_G(v)$ and
$\{z_1y_1,z_2y_2,z_3y_3\}$.  To obtain $w$ from $w'$, we must satisfy these six
edges and six additional edges incident to them.  We have the freedom to choose
weights on the six deleted edges and their seven incident vertices.

We define operation $S_i$ to satisfy the edges in the $i$th ``branch'' when
$w(vz_i)$ has been specified.  If $d_G(z_i)=3$, then $S_i$ uses
Lemma~\ref{beta2} to choose $w(z_i)$, $w(z_iy_i)$, and $w(y_i)$ (plus possible
changes to weights on $y_iy'_i$ and $y'_i$ but not on $z_ix_i$ or $z_iv$) so
that $z_ix_i$, $z_iy_i$, and $y_iy'_i$ become satisfied.  If $d_G(z_i)=4$, then
$S_i$ chooses $w(z_i)$ and $w(z_iy_i)$ to satisfy $z_ix_i$ and $z_ix_i'$.
(When $d_G(z_i)=4$, automatically $z_iy_i$ is satisfied, and $w(y_i)$ is
irrelevant.)

\begin{figure}[h]
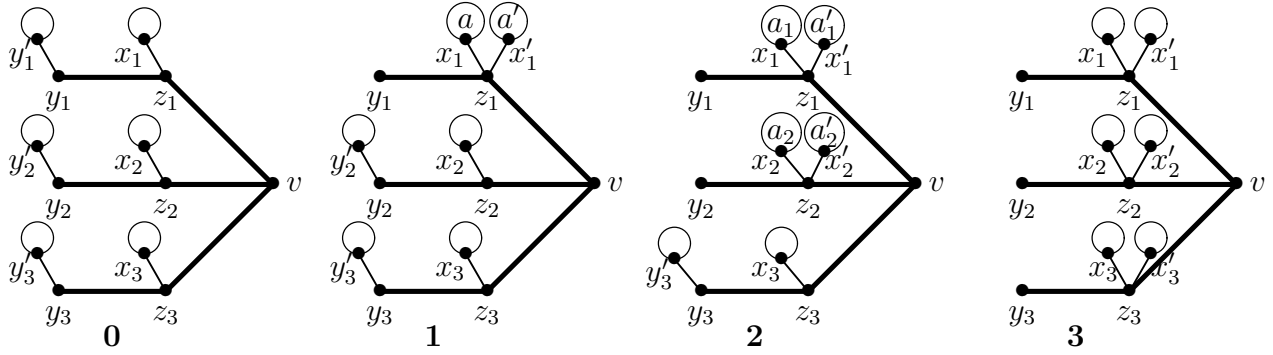

\gpic{
\expandafter\ifx\csname graph\endcsname\relax \csname newbox\endcsname\graph\fi
\expandafter\ifx\csname graphtemp\endcsname\relax \csname newdimen\endcsname\graphtemp\fi
\setbox\graph=\vtop{\vskip 0pt\hbox{%
    \graphtemp=.5ex\advance\graphtemp by 1.527in
    \rlap{\kern 0.224in\lower\graphtemp\hbox to 0pt{\hss $\bu$\hss}}%
    \graphtemp=.5ex\advance\graphtemp by 0.967in
    \rlap{\kern 0.224in\lower\graphtemp\hbox to 0pt{\hss $\bu$\hss}}%
    \graphtemp=.5ex\advance\graphtemp by 0.406in
    \rlap{\kern 0.224in\lower\graphtemp\hbox to 0pt{\hss $\bu$\hss}}%
    \graphtemp=.5ex\advance\graphtemp by 1.527in
    \rlap{\kern 0.784in\lower\graphtemp\hbox to 0pt{\hss $\bu$\hss}}%
    \graphtemp=.5ex\advance\graphtemp by 0.967in
    \rlap{\kern 0.784in\lower\graphtemp\hbox to 0pt{\hss $\bu$\hss}}%
    \graphtemp=.5ex\advance\graphtemp by 0.406in
    \rlap{\kern 0.784in\lower\graphtemp\hbox to 0pt{\hss $\bu$\hss}}%
    \graphtemp=.5ex\advance\graphtemp by 0.967in
    \rlap{\kern 1.345in\lower\graphtemp\hbox to 0pt{\hss $\bu$\hss}}%
    \special{pn 28}%
    \special{pa 224 1527}%
    \special{pa 784 1527}%
    \special{fp}%
    \special{pa 224 967}%
    \special{pa 784 967}%
    \special{fp}%
    \special{pa 224 406}%
    \special{pa 784 406}%
    \special{fp}%
    \special{pa 1345 967}%
    \special{pa 784 406}%
    \special{fp}%
    \special{pa 1345 967}%
    \special{pa 784 1527}%
    \special{fp}%
    \special{pa 1345 967}%
    \special{pa 784 967}%
    \special{fp}%
    \graphtemp=.5ex\advance\graphtemp by 1.331in
    \rlap{\kern 0.112in\lower\graphtemp\hbox to 0pt{\hss $\bu$\hss}}%
    \graphtemp=.5ex\advance\graphtemp by 1.331in
    \rlap{\kern 0.672in\lower\graphtemp\hbox to 0pt{\hss $\bu$\hss}}%
    \graphtemp=.5ex\advance\graphtemp by 0.770in
    \rlap{\kern 0.112in\lower\graphtemp\hbox to 0pt{\hss $\bu$\hss}}%
    \graphtemp=.5ex\advance\graphtemp by 0.770in
    \rlap{\kern 0.672in\lower\graphtemp\hbox to 0pt{\hss $\bu$\hss}}%
    \graphtemp=.5ex\advance\graphtemp by 0.210in
    \rlap{\kern 0.112in\lower\graphtemp\hbox to 0pt{\hss $\bu$\hss}}%
    \graphtemp=.5ex\advance\graphtemp by 0.210in
    \rlap{\kern 0.672in\lower\graphtemp\hbox to 0pt{\hss $\bu$\hss}}%
    \special{pn 11}%
    \special{pa 112 1331}%
    \special{pa 224 1527}%
    \special{fp}%
    \special{pa 784 1527}%
    \special{pa 672 1331}%
    \special{fp}%
    \special{pa 112 770}%
    \special{pa 224 967}%
    \special{fp}%
    \special{pa 784 967}%
    \special{pa 672 770}%
    \special{fp}%
    \special{pa 112 210}%
    \special{pa 224 406}%
    \special{fp}%
    \special{pa 784 406}%
    \special{pa 672 210}%
    \special{fp}%
    \special{pn 8}%
    \special{ar 112 1247 84 84 0 6.28319}%
    \special{ar 672 1247 84 84 0 6.28319}%
    \special{ar 112 686 84 84 0 6.28319}%
    \special{ar 672 686 84 84 0 6.28319}%
    \special{ar 112 126 84 84 0 6.28319}%
    \special{ar 672 126 84 84 0 6.28319}%
    \graphtemp=.5ex\advance\graphtemp by 1.639in
    \rlap{\kern 0.224in\lower\graphtemp\hbox to 0pt{\hss $y_3$\hss}}%
    \graphtemp=.5ex\advance\graphtemp by 1.079in
    \rlap{\kern 0.224in\lower\graphtemp\hbox to 0pt{\hss $y_2$\hss}}%
    \graphtemp=.5ex\advance\graphtemp by 0.518in
    \rlap{\kern 0.224in\lower\graphtemp\hbox to 0pt{\hss $y_1$\hss}}%
    \graphtemp=.5ex\advance\graphtemp by 0.967in
    \rlap{\kern 1.457in\lower\graphtemp\hbox to 0pt{\hss $v$\hss}}%
    \graphtemp=.5ex\advance\graphtemp by 1.639in
    \rlap{\kern 0.784in\lower\graphtemp\hbox to 0pt{\hss $z_3$\hss}}%
    \graphtemp=.5ex\advance\graphtemp by 1.079in
    \rlap{\kern 0.784in\lower\graphtemp\hbox to 0pt{\hss $z_2$\hss}}%
    \graphtemp=.5ex\advance\graphtemp by 0.518in
    \rlap{\kern 0.784in\lower\graphtemp\hbox to 0pt{\hss $z_1$\hss}}%
    \graphtemp=.5ex\advance\graphtemp by 1.779in
    \rlap{\kern 0.504in\lower\graphtemp\hbox to 0pt{\hss {\bf 0}\hss}}%
    \graphtemp=.5ex\advance\graphtemp by 1.410in
    \rlap{\kern 0.033in\lower\graphtemp\hbox to 0pt{\hss $y'_3$\hss}}%
    \graphtemp=.5ex\advance\graphtemp by 0.850in
    \rlap{\kern 0.033in\lower\graphtemp\hbox to 0pt{\hss $y'_2$\hss}}%
    \graphtemp=.5ex\advance\graphtemp by 0.289in
    \rlap{\kern 0.033in\lower\graphtemp\hbox to 0pt{\hss $y'_1$\hss}}%
    \graphtemp=.5ex\advance\graphtemp by 1.410in
    \rlap{\kern 0.593in\lower\graphtemp\hbox to 0pt{\hss $x_3$\hss}}%
    \graphtemp=.5ex\advance\graphtemp by 0.850in
    \rlap{\kern 0.593in\lower\graphtemp\hbox to 0pt{\hss $x_2$\hss}}%
    \graphtemp=.5ex\advance\graphtemp by 0.289in
    \rlap{\kern 0.593in\lower\graphtemp\hbox to 0pt{\hss $x_1$\hss}}%
    \graphtemp=.5ex\advance\graphtemp by 1.527in
    \rlap{\kern 1.905in\lower\graphtemp\hbox to 0pt{\hss $\bu$\hss}}%
    \graphtemp=.5ex\advance\graphtemp by 0.967in
    \rlap{\kern 1.905in\lower\graphtemp\hbox to 0pt{\hss $\bu$\hss}}%
    \graphtemp=.5ex\advance\graphtemp by 0.406in
    \rlap{\kern 1.905in\lower\graphtemp\hbox to 0pt{\hss $\bu$\hss}}%
    \graphtemp=.5ex\advance\graphtemp by 1.527in
    \rlap{\kern 2.466in\lower\graphtemp\hbox to 0pt{\hss $\bu$\hss}}%
    \graphtemp=.5ex\advance\graphtemp by 0.967in
    \rlap{\kern 2.466in\lower\graphtemp\hbox to 0pt{\hss $\bu$\hss}}%
    \graphtemp=.5ex\advance\graphtemp by 0.406in
    \rlap{\kern 2.466in\lower\graphtemp\hbox to 0pt{\hss $\bu$\hss}}%
    \graphtemp=.5ex\advance\graphtemp by 0.967in
    \rlap{\kern 3.026in\lower\graphtemp\hbox to 0pt{\hss $\bu$\hss}}%
    \special{pn 28}%
    \special{pa 1905 1527}%
    \special{pa 2466 1527}%
    \special{fp}%
    \special{pa 1905 967}%
    \special{pa 2466 967}%
    \special{fp}%
    \special{pa 1905 406}%
    \special{pa 2466 406}%
    \special{fp}%
    \special{pa 3026 967}%
    \special{pa 2466 406}%
    \special{fp}%
    \special{pa 3026 967}%
    \special{pa 2466 1527}%
    \special{fp}%
    \special{pa 3026 967}%
    \special{pa 2466 967}%
    \special{fp}%
    \graphtemp=.5ex\advance\graphtemp by 1.331in
    \rlap{\kern 1.793in\lower\graphtemp\hbox to 0pt{\hss $\bu$\hss}}%
    \graphtemp=.5ex\advance\graphtemp by 1.331in
    \rlap{\kern 2.353in\lower\graphtemp\hbox to 0pt{\hss $\bu$\hss}}%
    \graphtemp=.5ex\advance\graphtemp by 0.770in
    \rlap{\kern 1.793in\lower\graphtemp\hbox to 0pt{\hss $\bu$\hss}}%
    \graphtemp=.5ex\advance\graphtemp by 0.770in
    \rlap{\kern 2.353in\lower\graphtemp\hbox to 0pt{\hss $\bu$\hss}}%
    \graphtemp=.5ex\advance\graphtemp by 0.210in
    \rlap{\kern 2.578in\lower\graphtemp\hbox to 0pt{\hss $\bu$\hss}}%
    \graphtemp=.5ex\advance\graphtemp by 0.210in
    \rlap{\kern 2.353in\lower\graphtemp\hbox to 0pt{\hss $\bu$\hss}}%
    \special{pn 11}%
    \special{pa 1793 1331}%
    \special{pa 1905 1527}%
    \special{fp}%
    \special{pa 2466 1527}%
    \special{pa 2353 1331}%
    \special{fp}%
    \special{pa 1793 770}%
    \special{pa 1905 967}%
    \special{fp}%
    \special{pa 2466 967}%
    \special{pa 2353 770}%
    \special{fp}%
    \special{pa 2578 210}%
    \special{pa 2466 406}%
    \special{fp}%
    \special{pa 2466 406}%
    \special{pa 2353 210}%
    \special{fp}%
    \special{pn 8}%
    \special{ar 1793 1247 84 84 0 6.28319}%
    \special{ar 2353 1247 84 84 0 6.28319}%
    \special{ar 1793 686 84 84 0 6.28319}%
    \special{ar 2353 686 84 84 0 6.28319}%
    \special{ar 2578 112 98 98 0 6.28319}%
    \special{ar 2353 112 98 98 0 6.28319}%
    \graphtemp=.5ex\advance\graphtemp by 0.112in
    \rlap{\kern 2.353in\lower\graphtemp\hbox to 0pt{\hss $a$\hss}}%
    \graphtemp=.5ex\advance\graphtemp by 0.112in
    \rlap{\kern 2.578in\lower\graphtemp\hbox to 0pt{\hss $a'$\hss}}%
    \graphtemp=.5ex\advance\graphtemp by 1.639in
    \rlap{\kern 1.905in\lower\graphtemp\hbox to 0pt{\hss $y_3$\hss}}%
    \graphtemp=.5ex\advance\graphtemp by 1.079in
    \rlap{\kern 1.905in\lower\graphtemp\hbox to 0pt{\hss $y_2$\hss}}%
    \graphtemp=.5ex\advance\graphtemp by 0.518in
    \rlap{\kern 1.905in\lower\graphtemp\hbox to 0pt{\hss $y_1$\hss}}%
    \graphtemp=.5ex\advance\graphtemp by 0.967in
    \rlap{\kern 3.138in\lower\graphtemp\hbox to 0pt{\hss $v$\hss}}%
    \graphtemp=.5ex\advance\graphtemp by 1.639in
    \rlap{\kern 2.466in\lower\graphtemp\hbox to 0pt{\hss $z_3$\hss}}%
    \graphtemp=.5ex\advance\graphtemp by 1.079in
    \rlap{\kern 2.466in\lower\graphtemp\hbox to 0pt{\hss $z_2$\hss}}%
    \graphtemp=.5ex\advance\graphtemp by 0.518in
    \rlap{\kern 2.466in\lower\graphtemp\hbox to 0pt{\hss $z_1$\hss}}%
    \graphtemp=.5ex\advance\graphtemp by 1.779in
    \rlap{\kern 2.185in\lower\graphtemp\hbox to 0pt{\hss {\bf 1}\hss}}%
    \graphtemp=.5ex\advance\graphtemp by 1.410in
    \rlap{\kern 1.714in\lower\graphtemp\hbox to 0pt{\hss $y'_3$\hss}}%
    \graphtemp=.5ex\advance\graphtemp by 0.850in
    \rlap{\kern 1.714in\lower\graphtemp\hbox to 0pt{\hss $y'_2$\hss}}%
    \graphtemp=.5ex\advance\graphtemp by 0.289in
    \rlap{\kern 2.657in\lower\graphtemp\hbox to 0pt{\hss $x'_1$\hss}}%
    \graphtemp=.5ex\advance\graphtemp by 1.410in
    \rlap{\kern 2.274in\lower\graphtemp\hbox to 0pt{\hss $x_3$\hss}}%
    \graphtemp=.5ex\advance\graphtemp by 0.850in
    \rlap{\kern 2.274in\lower\graphtemp\hbox to 0pt{\hss $x_2$\hss}}%
    \graphtemp=.5ex\advance\graphtemp by 0.289in
    \rlap{\kern 2.274in\lower\graphtemp\hbox to 0pt{\hss $x_1$\hss}}%
    \graphtemp=.5ex\advance\graphtemp by 1.527in
    \rlap{\kern 3.586in\lower\graphtemp\hbox to 0pt{\hss $\bu$\hss}}%
    \graphtemp=.5ex\advance\graphtemp by 0.967in
    \rlap{\kern 3.586in\lower\graphtemp\hbox to 0pt{\hss $\bu$\hss}}%
    \graphtemp=.5ex\advance\graphtemp by 0.406in
    \rlap{\kern 3.586in\lower\graphtemp\hbox to 0pt{\hss $\bu$\hss}}%
    \graphtemp=.5ex\advance\graphtemp by 1.527in
    \rlap{\kern 4.147in\lower\graphtemp\hbox to 0pt{\hss $\bu$\hss}}%
    \graphtemp=.5ex\advance\graphtemp by 0.967in
    \rlap{\kern 4.147in\lower\graphtemp\hbox to 0pt{\hss $\bu$\hss}}%
    \graphtemp=.5ex\advance\graphtemp by 0.406in
    \rlap{\kern 4.147in\lower\graphtemp\hbox to 0pt{\hss $\bu$\hss}}%
    \graphtemp=.5ex\advance\graphtemp by 0.967in
    \rlap{\kern 4.707in\lower\graphtemp\hbox to 0pt{\hss $\bu$\hss}}%
    \special{pn 28}%
    \special{pa 3586 1527}%
    \special{pa 4147 1527}%
    \special{fp}%
    \special{pa 3586 967}%
    \special{pa 4147 967}%
    \special{fp}%
    \special{pa 3586 406}%
    \special{pa 4147 406}%
    \special{fp}%
    \special{pa 4707 967}%
    \special{pa 4147 406}%
    \special{fp}%
    \special{pa 4707 967}%
    \special{pa 4147 1527}%
    \special{fp}%
    \special{pa 4707 967}%
    \special{pa 4147 967}%
    \special{fp}%
    \graphtemp=.5ex\advance\graphtemp by 1.359in
    \rlap{\kern 3.446in\lower\graphtemp\hbox to 0pt{\hss $\bu$\hss}}%
    \graphtemp=.5ex\advance\graphtemp by 1.359in
    \rlap{\kern 4.006in\lower\graphtemp\hbox to 0pt{\hss $\bu$\hss}}%
    \graphtemp=.5ex\advance\graphtemp by 0.798in
    \rlap{\kern 4.231in\lower\graphtemp\hbox to 0pt{\hss $\bu$\hss}}%
    \graphtemp=.5ex\advance\graphtemp by 0.798in
    \rlap{\kern 4.006in\lower\graphtemp\hbox to 0pt{\hss $\bu$\hss}}%
    \graphtemp=.5ex\advance\graphtemp by 0.238in
    \rlap{\kern 4.231in\lower\graphtemp\hbox to 0pt{\hss $\bu$\hss}}%
    \graphtemp=.5ex\advance\graphtemp by 0.238in
    \rlap{\kern 4.006in\lower\graphtemp\hbox to 0pt{\hss $\bu$\hss}}%
    \special{pn 11}%
    \special{pa 3446 1359}%
    \special{pa 3586 1527}%
    \special{fp}%
    \special{pa 4147 1527}%
    \special{pa 4006 1359}%
    \special{fp}%
    \special{pa 4231 798}%
    \special{pa 4147 967}%
    \special{fp}%
    \special{pa 4147 967}%
    \special{pa 4006 798}%
    \special{fp}%
    \special{pa 4231 238}%
    \special{pa 4147 406}%
    \special{fp}%
    \special{pa 4147 406}%
    \special{pa 4006 238}%
    \special{fp}%
    \special{pn 8}%
    \special{ar 3446 1275 84 84 0 6.28319}%
    \special{ar 4006 1275 84 84 0 6.28319}%
    \special{ar 4231 695 104 104 0 6.28319}%
    \special{ar 4006 695 104 104 0 6.28319}%
    \special{ar 4231 134 104 104 0 6.28319}%
    \special{ar 4006 134 104 104 0 6.28319}%
    \graphtemp=.5ex\advance\graphtemp by 0.134in
    \rlap{\kern 4.006in\lower\graphtemp\hbox to 0pt{\hss $a_1$\hss}}%
    \graphtemp=.5ex\advance\graphtemp by 0.134in
    \rlap{\kern 4.231in\lower\graphtemp\hbox to 0pt{\hss $a_1'$\hss}}%
    \graphtemp=.5ex\advance\graphtemp by 0.695in
    \rlap{\kern 4.006in\lower\graphtemp\hbox to 0pt{\hss $a_2$\hss}}%
    \graphtemp=.5ex\advance\graphtemp by 0.695in
    \rlap{\kern 4.231in\lower\graphtemp\hbox to 0pt{\hss $a_2'$\hss}}%
    \graphtemp=.5ex\advance\graphtemp by 1.639in
    \rlap{\kern 3.586in\lower\graphtemp\hbox to 0pt{\hss $y_3$\hss}}%
    \graphtemp=.5ex\advance\graphtemp by 1.079in
    \rlap{\kern 3.586in\lower\graphtemp\hbox to 0pt{\hss $y_2$\hss}}%
    \graphtemp=.5ex\advance\graphtemp by 0.518in
    \rlap{\kern 3.586in\lower\graphtemp\hbox to 0pt{\hss $y_1$\hss}}%
    \graphtemp=.5ex\advance\graphtemp by 0.967in
    \rlap{\kern 4.819in\lower\graphtemp\hbox to 0pt{\hss $v$\hss}}%
    \graphtemp=.5ex\advance\graphtemp by 1.639in
    \rlap{\kern 4.147in\lower\graphtemp\hbox to 0pt{\hss $z_3$\hss}}%
    \graphtemp=.5ex\advance\graphtemp by 1.079in
    \rlap{\kern 4.147in\lower\graphtemp\hbox to 0pt{\hss $z_2$\hss}}%
    \graphtemp=.5ex\advance\graphtemp by 0.518in
    \rlap{\kern 4.147in\lower\graphtemp\hbox to 0pt{\hss $z_1$\hss}}%
    \graphtemp=.5ex\advance\graphtemp by 1.779in
    \rlap{\kern 3.866in\lower\graphtemp\hbox to 0pt{\hss {\bf 2}\hss}}%
    \graphtemp=.5ex\advance\graphtemp by 1.438in
    \rlap{\kern 3.367in\lower\graphtemp\hbox to 0pt{\hss $y'_3$\hss}}%
    \graphtemp=.5ex\advance\graphtemp by 0.850in
    \rlap{\kern 4.310in\lower\graphtemp\hbox to 0pt{\hss $x'_2$\hss}}%
    \graphtemp=.5ex\advance\graphtemp by 0.317in
    \rlap{\kern 4.310in\lower\graphtemp\hbox to 0pt{\hss $x'_1$\hss}}%
    \graphtemp=.5ex\advance\graphtemp by 1.410in
    \rlap{\kern 3.927in\lower\graphtemp\hbox to 0pt{\hss $x_3$\hss}}%
    \graphtemp=.5ex\advance\graphtemp by 0.850in
    \rlap{\kern 3.927in\lower\graphtemp\hbox to 0pt{\hss $x_2$\hss}}%
    \graphtemp=.5ex\advance\graphtemp by 0.289in
    \rlap{\kern 3.927in\lower\graphtemp\hbox to 0pt{\hss $x_1$\hss}}%
    \graphtemp=.5ex\advance\graphtemp by 1.527in
    \rlap{\kern 5.267in\lower\graphtemp\hbox to 0pt{\hss $\bu$\hss}}%
    \graphtemp=.5ex\advance\graphtemp by 0.967in
    \rlap{\kern 5.267in\lower\graphtemp\hbox to 0pt{\hss $\bu$\hss}}%
    \graphtemp=.5ex\advance\graphtemp by 0.406in
    \rlap{\kern 5.267in\lower\graphtemp\hbox to 0pt{\hss $\bu$\hss}}%
    \graphtemp=.5ex\advance\graphtemp by 1.527in
    \rlap{\kern 5.828in\lower\graphtemp\hbox to 0pt{\hss $\bu$\hss}}%
    \graphtemp=.5ex\advance\graphtemp by 0.967in
    \rlap{\kern 5.828in\lower\graphtemp\hbox to 0pt{\hss $\bu$\hss}}%
    \graphtemp=.5ex\advance\graphtemp by 0.406in
    \rlap{\kern 5.828in\lower\graphtemp\hbox to 0pt{\hss $\bu$\hss}}%
    \graphtemp=.5ex\advance\graphtemp by 0.967in
    \rlap{\kern 6.388in\lower\graphtemp\hbox to 0pt{\hss $\bu$\hss}}%
    \special{pn 28}%
    \special{pa 5267 1527}%
    \special{pa 5828 1527}%
    \special{fp}%
    \special{pa 5267 967}%
    \special{pa 5828 967}%
    \special{fp}%
    \special{pa 5267 406}%
    \special{pa 5828 406}%
    \special{fp}%
    \special{pa 6388 967}%
    \special{pa 5828 406}%
    \special{fp}%
    \special{pa 6388 967}%
    \special{pa 5828 1527}%
    \special{fp}%
    \special{pa 6388 967}%
    \special{pa 5828 967}%
    \special{fp}%
    \graphtemp=.5ex\advance\graphtemp by 1.331in
    \rlap{\kern 5.940in\lower\graphtemp\hbox to 0pt{\hss $\bu$\hss}}%
    \graphtemp=.5ex\advance\graphtemp by 1.331in
    \rlap{\kern 5.716in\lower\graphtemp\hbox to 0pt{\hss $\bu$\hss}}%
    \graphtemp=.5ex\advance\graphtemp by 0.770in
    \rlap{\kern 5.940in\lower\graphtemp\hbox to 0pt{\hss $\bu$\hss}}%
    \graphtemp=.5ex\advance\graphtemp by 0.770in
    \rlap{\kern 5.716in\lower\graphtemp\hbox to 0pt{\hss $\bu$\hss}}%
    \graphtemp=.5ex\advance\graphtemp by 0.210in
    \rlap{\kern 5.940in\lower\graphtemp\hbox to 0pt{\hss $\bu$\hss}}%
    \graphtemp=.5ex\advance\graphtemp by 0.210in
    \rlap{\kern 5.716in\lower\graphtemp\hbox to 0pt{\hss $\bu$\hss}}%
    \special{pn 11}%
    \special{pa 5940 1331}%
    \special{pa 5828 1527}%
    \special{fp}%
    \special{pa 5828 1527}%
    \special{pa 5716 1331}%
    \special{fp}%
    \special{pa 5940 770}%
    \special{pa 5828 967}%
    \special{fp}%
    \special{pa 5828 967}%
    \special{pa 5716 770}%
    \special{fp}%
    \special{pa 5940 210}%
    \special{pa 5828 406}%
    \special{fp}%
    \special{pa 5828 406}%
    \special{pa 5716 210}%
    \special{fp}%
    \special{pn 8}%
    \special{ar 5940 1247 84 84 0 6.28319}%
    \special{ar 5716 1247 84 84 0 6.28319}%
    \special{ar 5940 686 84 84 0 6.28319}%
    \special{ar 5716 686 84 84 0 6.28319}%
    \special{ar 5940 126 84 84 0 6.28319}%
    \special{ar 5716 126 84 84 0 6.28319}%
    \graphtemp=.5ex\advance\graphtemp by 1.639in
    \rlap{\kern 5.267in\lower\graphtemp\hbox to 0pt{\hss $y_3$\hss}}%
    \graphtemp=.5ex\advance\graphtemp by 1.079in
    \rlap{\kern 5.267in\lower\graphtemp\hbox to 0pt{\hss $y_2$\hss}}%
    \graphtemp=.5ex\advance\graphtemp by 0.518in
    \rlap{\kern 5.267in\lower\graphtemp\hbox to 0pt{\hss $y_1$\hss}}%
    \graphtemp=.5ex\advance\graphtemp by 0.967in
    \rlap{\kern 6.500in\lower\graphtemp\hbox to 0pt{\hss $v$\hss}}%
    \graphtemp=.5ex\advance\graphtemp by 1.639in
    \rlap{\kern 5.828in\lower\graphtemp\hbox to 0pt{\hss $z_3$\hss}}%
    \graphtemp=.5ex\advance\graphtemp by 1.079in
    \rlap{\kern 5.828in\lower\graphtemp\hbox to 0pt{\hss $z_2$\hss}}%
    \graphtemp=.5ex\advance\graphtemp by 0.518in
    \rlap{\kern 5.828in\lower\graphtemp\hbox to 0pt{\hss $z_1$\hss}}%
    \graphtemp=.5ex\advance\graphtemp by 1.779in
    \rlap{\kern 5.547in\lower\graphtemp\hbox to 0pt{\hss {\bf 3}\hss}}%
    \graphtemp=.5ex\advance\graphtemp by 1.410in
    \rlap{\kern 6.019in\lower\graphtemp\hbox to 0pt{\hss $x'_3$\hss}}%
    \graphtemp=.5ex\advance\graphtemp by 0.850in
    \rlap{\kern 6.019in\lower\graphtemp\hbox to 0pt{\hss $x'_2$\hss}}%
    \graphtemp=.5ex\advance\graphtemp by 0.289in
    \rlap{\kern 6.019in\lower\graphtemp\hbox to 0pt{\hss $x'_1$\hss}}%
    \graphtemp=.5ex\advance\graphtemp by 1.410in
    \rlap{\kern 5.692in\lower\graphtemp\hbox to 0pt{\hss $x_3$\hss}}%
    \graphtemp=.5ex\advance\graphtemp by 0.850in
    \rlap{\kern 5.636in\lower\graphtemp\hbox to 0pt{\hss $x_2$\hss}}%
    \graphtemp=.5ex\advance\graphtemp by 0.289in
    \rlap{\kern 5.636in\lower\graphtemp\hbox to 0pt{\hss $x_1$\hss}}%
    \hbox{\vrule depth1.779in width0pt height 0pt}%
    \kern 6.500in
  }%
}%
}
\vspace{-1pc}

\caption{Case {\bf G} for Lemma~\ref{83red2}\label{fig2G}}
\end{figure}

\textbf{Subcase 0:}
Set $w(v)=w(vz_1)=w(vz_2)=w(vz_3)=2$, and consider $i\in\{1,2,3\}$.
If $w'(x_iz_i)=1$, then $z_iv$ is automatically satisfied; apply $S_i$.
If $w'(x_iz_i)=2$, then $z_iy_i$ is automatically satisfied.  Set $w(z_i)=1$
to satisfy $z_iv$.  Choose $w(z_iy_i)$ to satisfy $z_ix_i$, and choose $w(y_i)$
to satisfy $y_iy'_i$.

\textbf{Subcase 1:} 
Let $a=\rp{x_1}{z_1}$ and $a'=\rp{x_1'}{z_1}$.  When $w'(z_1x_1)=w'(z_1x'_1)=1$
and $\{a,a'\}=\{6,7\}$, set $w(z_3v)=w(z_2v)=2$.  Otherwise, set
$w(z_3v)=w(z_2v)=1$.  With $w(z_3v)$ and $w(z_2v)$ fixed, apply $S_2$ and
$S_3$.  Now choose $w(v)$ and $w(vz_1)$ to satisfy $vz_2$ and $vz_3$, with
$w(v)\le w(vz_1)$.

If we have set $w(z_3v)=w(z_2v)=2$, then $w'(z_1x_1)=w'(z_1x_1')=1$;
satisfy $vz_1$ by setting $w(z_1)=w(z_1y_1)=1$.  Since $\{a,a'\}=\{6,7\}$,
this also satisfies $\Gamma_{G'}(z_1)$.

If we have set $w(z_3v)=w(z_2v)=1$, then $w'(z_1x_1)+w'(z_1x_1')\ge3$ or
$\{a,a'\}\ne\{6,7\}$.  In the first case, $vz_1$ is automatically satisfied;
apply $S_1$.  In the second, $w(z_3v)=w(z_2v)=w'(z_1x_1)=w'(z_1x'_1)=1$ and
$\{a,a'\}\ne\{6,7\}$; choose $b\in\{6,7\}-\{a,a'\}$.  If $w(v)=1$, then $vz_1$
is automatically satisfied; apply $S_1$.  Otherwise, $w(v)=w(vz_1)=2$, since we
chose $w(v)\le w(vz_1)$.  Now choose $w(z_1)$ and $w(z_1y_1)$ with sum $b-3$.
This satisfies $vz_1$ and $\Gamma_{G'}(z_1)$.

\textbf{Subcase 2:}
Set $w(z_3v)=1$.  With $w(z_3v)$ fixed, apply $S_3$.
If $w'(z_1x_1)=2$, then setting $w(z_1v)=w(v)=1$ ensures satisfying $z_1v$
and $z_2v$; choose $w(z_2v)$ to satisfy $vz_3$ and apply $S_1$ and $S_2$.
By symmetry, we may thus assume $w'(z_ix_i)=w'(z_ix'_i)=1$ for $i\in\{1,2\}$.

If $w(z_3)+w'(z_3x_3)+w(z_3y_3)>3$, then setting $w(v)=w(vz_2)=w(vz_1)=1$
satisfies $vz_3$ and ensures satisfying $vz_2$ and $vz_1$; apply $S_2$
and $S_1$.  Hence we may also assume $w(z_3)+w'(z_3x_3)+w(z_3y_3)=3$.
Let $a_i=\rp{x_i}{z_i}$ and $a'_i=\rp{x_i'}{z_i}$, for $i\in\{1,2\}$.

If  $\{a_1,a'_1\}\neq \{5,6\}$, then set $w(vz_1)=w(v)=1$ and $w(vz_2)=2$.
Now $vz_3$ is satisfied and $vz_2$ is automatically satisfied; apply $S_2$.
Choose $b\in\{5,6\}-\{a_1,a'_1\}$.  Choose $w(z_1)$ and $w(z_1y_1)$ summing to
$b-2$.  This satisfies $vz_1$ and $\Gamma_{G'}(z_1)$.

By symmetry, we may thus assume $\{a_1,a'_1\}=\{a_2,a'_2\}=\{5,6\}$.  Let
$w(v)=2$ and $w(vz_i)=w(z_i)=w(z_iy_i)=2$ for $i\in\{1,2\}$ to satisfy all
remaining edges.

\textbf{Subcase 3:}
Set $w(v)=w(vz_1)=w(vz_2)=w(vz_3)=1$ to guarantee satisfying each $vz_i$.  Now
for each $i$ choose $w(y_iz_i)$ and $w(z_i)$ to satisfy $z_ix_i$ and $z_ix'_i$.
\end{proof}

Case {\bf F} and Case {\bf G} in Lemma~\ref{83red2} can be generalized.
If $v$ in the former case or $z_i$ in the latter is a $\beta'$-vertex of
any even degree, then the configuration remains $2$-reducible.  We omit this
since it is not needed to prove the $1,2$-Conjecture for $\Mad(G)<8/3$; the
more restrictive configurations in the lemma complete an unavoidable set.

\begin{lem}\label{83disch2}
If $G$ has average degree less than $8/3$, then $G$ contains one of the
following:

{\bf A}. A $3^-$-vertex having a $1$-neighbor.

{\bf B}. A $4^-$-vertex having two $2^-$-neighbors.

{\bf C}. A $5^+$-vertex $v$ whose number of $2^-$-neighbors is at least
$(d_G(v)-1)/2$.

{\bf D}. Two adjacent $\beta$-vertices.

{\bf E}. A $\beta$-vertex with a $\beta'$-neighbor.

{\bf F.} A $\beta'$-vertex of degree 4 having two $\beta'$-neighbors of degree
$4$.

{\bf G.} A $3$-vertex such that each neighbor is a $\beta$-vertex or is a
$\beta'$-vertex of degree $4$.
\end{lem}

\begin{proof}
We prove that a graph $G$ containing none of {\bf A-G} has average degree at
least $8/3$.  Give every vertex $v$ in $G$ initial charge $d_G(v)$.  Move
charge via the following rules:

\medskip
{\narrower

\noindent
(1) Each $1$-vertex takes $\FR53$ from its neighbor.
\smallskip

\noindent
(2) Each $2$-vertex takes $\FR23$ from one $3^+$-neighbor.
\smallskip

\noindent
(3) Each $3$-vertex with a $2$-neighbor takes $\FR16$ from each other neighbor.
\smallskip

\noindent
(4) Each $4$-vertex with a $1$-neighbor takes $\FR16$ from each other neighbor
not a $\beta'$-vertex.
\smallskip

}

\noindent
Let $\mu(v)$ denote the resulting charge at $v$; it suffices to check that
$\mu(v)\ge\FR83$ for all $v$.  For ${\bf Z}\in\{{\bf A},\ldots,{\bf G}\}$, let
$\bZ$ mean ``configuration {\bf Z} does not occur in $G$''.  Note that by
$\bB$, the vertex taking charge in Rule 3 or Rule 4 is a $\beta$-vertex or a
$\beta'$-vertex, respectively.

Case $d(v)=1$: By Rule 1, $v$ receives $\FR53$ and has final charge $\FR83$.

Case $d(v)=2$: By $\bA$ and $\bB$, $v$ gives no charge and receives $\FR23$
from a $3^+$-neighbor.

Case $d(v)=3$:
By $\bA$, $v$ has no $1$-neighbor.  If $v$ also has no $2$-neighbor, then by
$\bG$ at most two neighbors take charge $\FR16$ from it, so $\mu(v)\ge\FR83$.
If $v$ has a $2$-neighbor, then by $\bB$ it is a $\beta$-vertex, has only
one $2$-neighbor, and may give $\FR23$ to that $2$-neighbor.  By $\bD$ and
$\bE$, $v$ loses no other charge.  If $v$ does loses $\FR23$, then $v$ needs to
regain $\FR13$ and does so via Rule 3.

Case $d(v)=4$:
By $\bB$, $v$ has at most one $2^-$-neighbor.  If $v$ has no $1$-neighbor, then
$v$ loses at most $\FR23+\FR36$.  Hence in this case $\mu(v)>\FR83$.  If $v$
has a $1$-neighbor, then $v$ loses $\FR53$ to it and is a $\beta'$-vertex.
By $\bE$ and $\bF$, $v$ has no $\beta$-neighbor and at most one
$\beta'$-neighbor.  Hence it gives away no other charge and receives at least
$\FR26$ to reach $\mu(v)\ge \FR83$.

Case $d(v)\ge5$:
By $\bC$, $v$ has at most $\FR{d(v)-2}2$ $2^-$-neighbors.  If the inequality
is strict, then $v$ gives at most $\FR53\FR{d(v)-3}2$ to those vertices and
at most $\FR16$ to each other neighbor, yielding
$$
\mu(v)\ge d(v)-\FR53\FR{d(v)-3}2-\FR16\FR{d(v)+3}2
=\FR{d(v)}{12}+\FR94\ge \FR{32}{12}=\FR83.
$$
If $v$ has exactly $\FR{d(v)-2}2$ $2^-$-neighbors and at least one of them
is a $2$-vertex, then $d(v)\ge6$ and
$$
\mu(v)\ge d(v)-\FR53\FR{d(v)-4}2-\FR23-\FR16\FR{d(v)+2}2
=\FR{d(v)}{12}+\FR83-\FR16> \FR83.
$$
In the remaining case, $v$ is a $\beta'$-vertex with degree at least $6$.
By definition, $v$ has $\FR{d(v)-2}2$ $1$-neighbors and no $2$-neighbor.
By $\bE$, $v$ has no $\beta$-neighbor.  By Rules 3 and 4, $v$ gives charge
only to its $1$-neighbors.  Hence
$$
\mu(v)\ge d(v)-\FR53\FR{d(v)-2}2 =\FR{d(v)}{6}+\FR53\ge \FR83.
$$
\vspace{-3pc}

\end{proof}
\medskip

Since every configuration in Lemma~\ref{83disch2} 
is listed in Lemma~\ref{83red2}, the following is proved.

\begin{thm}
The $1,2$-Conjecture holds for each graph $G$ such that $\Mad(G)<8/3$.
\end{thm}

\section{Proper $3$-weighting when $\Mad(G)<8/3$}\label{sec123}

For the discussion of proper $3$-weightings, again it will be helpful to have
notation for special types of vertices.  The definition of $\beta$-vertex is
the same as before, but instead of $\beta'$-vertices we introduce
$\alpha$-vertices and $\gamma$-vertices.

\begin{definition}
An \textit{$\alpha$-vertex} is a $2$ vertex with a $2$-neighbor.
A \textit{$\beta$-vertex} is a $3$-vertex with a $2$-neighbor.
A \textit{$\gamma$-vertex} is a $4$-vertex with a $1$-neighbor
or is a $3$-vertex with an $\alpha$-neighbor or two $2$-neighbors.
\end{definition}

\begin{lem}\label{83disch3}
If $G$ has average degree less than $8/3$, then $G$ contains one of the
following:

{\bf A}. A $2$-vertex or $3$-vertex having a $1$-neighbor.

{\bf B}. A $4^-$-vertex whose neighbors all have degree $2$.

{\bf C}. A $3$-vertex having an $\alpha$-neighbor and another $2$-neighbor.

{\bf D}. A $4$-vertex having a $1$-neighbor and a $2^-$-neighbor.

{\bf E}. A $5^+$-vertex $v$ with
$3p_1+2p_2\ge d(v)$, where $p_i$ is the number of $i$-neighbors of $v$.

{\bf F}. Two adjacent $\gamma$-vertices.

{\bf G}. A $3$-vertex with two $\gamma$-neighbors.

{\bf H}. A $6$-vertex or $7$-vertex having a $1$-neighbor and four
$\gamma$-neighbors.

{\bf I}. A $5$-vertex having a $1$-neighbor and three $\gamma$-neighbors.

{\bf J}. A $4$-vertex with $p+q+r\ge5$, where $p,q,r$ are its numbers of
$2$-neighbors, $\gamma$-neighbors, and $\alpha$-neighbors, respectively.

{\bf K}. A $\gamma$-vertex whose $3^+$-neighbors are all $\beta$-vertices.
\end{lem}

\begin{proof}
We prove that a graph $G$ containing none of {\bf A-K} has average degree at
least $8/3$.  Give every vertex $v$ in $G$ initial charge $d(v)$.
Move charge via the following rules:

\medskip
{\narrower

\noindent
(1) Each $1$-vertex takes $\FR53$ from its neighbor.
\smallskip

\noindent
(2) Each $\alpha$-vertex takes $\FR23$ from its $3^+$-neighbor.
\smallskip

\noindent
(3) Each $2$-vertex that is not an $\alpha$-vertex takes $\FR13$ from each
neighbor.
\smallskip

\noindent
(4) Each $\gamma$-vertex takes $\FR13$ from each $3^+$-neighbor that is not
a $\beta$-vertex.

}
\medskip

\noindent
Let $\mu(v)$ denote the resulting charge at $v$; it suffices to check that
$\mu(v)\ge\frac{8}{3}$ for all $v$.  For
${\bf Z}\in \{{\bf A},\ldots,{\bf K}\}$, let $\bZ$ mean ``configuration
${\bf Z}$ does not occur in $G$''.  Fix a vertex $v$ and let $p_i$, $q$, and
$r$ count its $i$-neighbors, $\gamma$-neighbors, and $\alpha$-neighbors,
respectively.

Case $d(v)=1$: By Rule 1, $v$ receives $\FR53$ and has final charge $\FR83$.

Case $d(v)=2$: By $\bA$, $v$ gives no charge; via Rule 2 or Rule 3, $v$
receives $\FR23$.

Case $d(v)=3$: By $\bA$ and $\bB$, $v$ has no $1$-neighbor and at most two
$2$-neighbors.

If $v$ has no $2$-neighbor, then $v$ is not a $\gamma$-vertex or a
$\beta$-vertex.  It loses nothing via Rule 2 or 3, and via Rule 4 it gives
$\FR13$ to each $\gamma$-neighbor.  By $\bG$, $v$ loses at most $\FR13$, and
$\mu(v)\ge\FR83$.

If $v$ has a $2$-neighbor, then $v$ is a $\beta$-vertex and loses nothing via
Rule 4.  By $\bB$ and $\bC$, $v$ gives at most $\frac{2}{3}$ to its
$2$-neighbors, with equality only if it is a $\gamma$-vertex.  In this case, by
$\bA$ and $\bB$, $v$ has a $3^+$-neighbor $u$.  By $\bF$, $u$ is not a
$\gamma$-vertex, and by $\bK$ not all choices of $u$ are $\beta$-vertices.
Hence from some $3^+$-neighbor $v$ receives $\FR13$, and $\mu(v)\ge\FR83$.

Case $d(v)=4$:
If $v$ has a $1$-neighbor, then $v$ is a $\gamma$-vertex.  By $\bD$ it has no
other $2^-$-neighbor and loses at most $\FR53$.  Its other neighbors are
$3^+$-vertices.  By $\bF$, none is a $\gamma$-vertex, and by $\bK$ they are not
all $\beta$-vertices.  Hence $v$ receives at least $\FR13$ via Rule 4, and
$\mu(v)\ge\FR83$.

If $v$ has no $1$-neighbor, then $v$ is not a $\gamma$-vertex.  By $\bJ$,
$p+q+r\le 4$.  An $\alpha$-neighbor contributes to both $p$ and $r$.  Hence $v$
loses exactly $(p+q+r)/3$, yielding $\mu(v)\ge\FR83$.

Case $d(v)\ge5$:
The charge lost by $v$ is $\FR13(5p_1+p_2+q+r)$.  Hence $v$ is happy if
$5p_1+p_2+q+r\le 3d(v)-8$.  Also, configurations with $3p_1+2p_2\ge d(v)$
are forbidden, by $\bE$.  Hence if $v$ is not happy and is not in a
configuration forbidden by $\bE$, then
\begin{equation}\label{ineq}
5p_1+p_2+q+r\ge 3d(v)-7 \quad\textrm{and}\quad 3p_1+2p_2\le d(v)-1.
\end{equation}
Since $r\le p_2$ and $q\le d(v)-p_1$, the first inequality above yields
$4p_1+2p_2\ge 2d(v)-7$.  Eliminating $2p_2$ from the two inequalities then
yields $2d(v)-7-4p_1\le d(v)-1-3p_1$, which simplifies to $d(v)-6\le p_1$.
Since also $p_1\le \FL{(d(v)-1)/3}$, we obtain $d(v)\le8$.

If $p_1=0$, then substituting $q+r\le d(v)$ in (1) yields
$(d(v)-1)/2\ge 2d(v)-7$, which simplifies to $d(v)\le13/3$.  Hence we may
assume $p_1\ge1$.  We consider below the remaining unexcluded possibilities for
$(p_1,p_2,r,q)$.  In each case these are the choices allowed by (\ref{ineq}).

For $d(v)=5$, the remaining case is $(1,0,0,q)$ with $q\in\{3,4\}$, forbidden
by $\bI$.

For $d(v)=6$, the remaining case is $(1,1,1,4)$, forbidden by $\bH$.

For $d(v)=7$, the case $(1,p_2,r,q)$ requires $p_2\le1$, which yields
$5p_1+2p_2+q\le 12 <14$, so $v$ remains happy.  Hence the remaining case is
$(2,0,0,q)$ with $q\in\{4,5\}$, forbidden by $\bH$.

For $d(v)=8$, with $p_1\le 2$ and $3p_1+2p_2\le 7$, we have
$5p_1+2p_2\le 10$.  At most $16/3$ is lost, and hence $\mu(v)\ge\FR83$.
\end{proof}

The next lemma explains the role of $\gamma$-vertices.

\begin{lem}\label{gamma}
Let $v$ be a $\gamma$-vertex having a $3^+$-neighbor $x$.
Define $F\esub E(G)$ as follows: 

$F=\{vu\}$ if $v$ is a $4$-vertex with $1$-neighbor $u$,

$F=\{vz,vz'\}$ if $v$ is a $3$-vertex with $2$-neighbors $z$ and $z'$,

$F=\Gamma_G(z)$ if $v$ is a $3$-vertex with $\alpha$-neighbor $z$.

\noindent
Given any weighting of $G-F$, weights in $\{1,2,3\}$ can be chosen on $F$ to
satisfy all edges in $F$ or incident to edges of $F$ except $vx$, without
changing the weights on edges not in $F$.
\end{lem}
\begin{proof}
Figure~\ref{figgamma} shows $F$ in bold; the weight on $vx$ is fixed.
When $v$ is a $4$-vertex, choose $w(vu)$ to satisfy the two edges from $v$
to $N_G(v)-\{x,u\}$.  When $v$ is a $3$-vertex with $2$-neighbors $z$ and 
$z'$ having neighbors $y$ and $y'$ other than $v$, choose $w(vz)$ to 
satisfy $vz'$ and $zy$, and choose $w(vz')$ to satisfy $vz$ and $z'y'$.
When $v$ is a $3$-vertex with $\alpha$-neighbor $z$ having neighbor $y$ other
than $v$, let $x'$ and $y'$ be the remaining neighbors of $v$ and $y$.
Choose $w(vz)$ to satisfy $vx'$ and $zy$, and choose $w(zy)$ to satisfy
$vz$ and $yy'$.
\end{proof}

\begin{figure}[h]
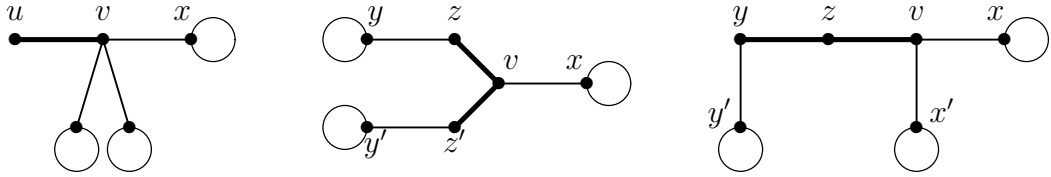

\gpic{
\expandafter\ifx\csname graph\endcsname\relax \csname newbox\endcsname\graph\fi
\expandafter\ifx\csname graphtemp\endcsname\relax \csname newdimen\endcsname\graphtemp\fi
\setbox\graph=\vtop{\vskip 0pt\hbox{%
    \graphtemp=.5ex\advance\graphtemp by 0.138in
    \rlap{\kern 0.092in\lower\graphtemp\hbox to 0pt{\hss $\bu$\hss}}%
    \graphtemp=.5ex\advance\graphtemp by 0.138in
    \rlap{\kern 0.552in\lower\graphtemp\hbox to 0pt{\hss $\bu$\hss}}%
    \graphtemp=.5ex\advance\graphtemp by 0.138in
    \rlap{\kern 1.013in\lower\graphtemp\hbox to 0pt{\hss $\bu$\hss}}%
    \graphtemp=.5ex\advance\graphtemp by 0.598in
    \rlap{\kern 0.414in\lower\graphtemp\hbox to 0pt{\hss $\bu$\hss}}%
    \graphtemp=.5ex\advance\graphtemp by 0.598in
    \rlap{\kern 0.690in\lower\graphtemp\hbox to 0pt{\hss $\bu$\hss}}%
    \special{pn 28}%
    \special{pa 92 138}%
    \special{pa 552 138}%
    \special{fp}%
    \special{pn 11}%
    \special{pa 552 138}%
    \special{pa 1013 138}%
    \special{fp}%
    \special{pa 552 138}%
    \special{pa 414 598}%
    \special{fp}%
    \special{pa 552 138}%
    \special{pa 690 598}%
    \special{fp}%
    \special{pn 8}%
    \special{ar 414 713 115 115 0 6.28319}%
    \special{ar 690 713 115 115 0 6.28319}%
    \special{ar 1128 138 115 115 0 6.28319}%
    \graphtemp=.5ex\advance\graphtemp by 0.000in
    \rlap{\kern 0.092in\lower\graphtemp\hbox to 0pt{\hss $u$\hss}}%
    \graphtemp=.5ex\advance\graphtemp by 0.000in
    \rlap{\kern 0.552in\lower\graphtemp\hbox to 0pt{\hss $v$\hss}}%
    \graphtemp=.5ex\advance\graphtemp by 0.000in
    \rlap{\kern 0.967in\lower\graphtemp\hbox to 0pt{\hss $x$\hss}}%
    \graphtemp=.5ex\advance\graphtemp by 0.598in
    \rlap{\kern 1.933in\lower\graphtemp\hbox to 0pt{\hss $\bu$\hss}}%
    \graphtemp=.5ex\advance\graphtemp by 0.138in
    \rlap{\kern 1.933in\lower\graphtemp\hbox to 0pt{\hss $\bu$\hss}}%
    \graphtemp=.5ex\advance\graphtemp by 0.598in
    \rlap{\kern 2.393in\lower\graphtemp\hbox to 0pt{\hss $\bu$\hss}}%
    \graphtemp=.5ex\advance\graphtemp by 0.138in
    \rlap{\kern 2.393in\lower\graphtemp\hbox to 0pt{\hss $\bu$\hss}}%
    \graphtemp=.5ex\advance\graphtemp by 0.368in
    \rlap{\kern 2.623in\lower\graphtemp\hbox to 0pt{\hss $\bu$\hss}}%
    \graphtemp=.5ex\advance\graphtemp by 0.368in
    \rlap{\kern 3.084in\lower\graphtemp\hbox to 0pt{\hss $\bu$\hss}}%
    \special{pn 28}%
    \special{pa 2393 598}%
    \special{pa 2623 368}%
    \special{fp}%
    \special{pa 2623 368}%
    \special{pa 2393 138}%
    \special{fp}%
    \special{pn 11}%
    \special{pa 1933 598}%
    \special{pa 2393 598}%
    \special{fp}%
    \special{pa 1933 138}%
    \special{pa 2393 138}%
    \special{fp}%
    \special{pa 2623 368}%
    \special{pa 3084 368}%
    \special{fp}%
    \special{pn 8}%
    \special{ar 1818 598 115 115 0 6.28319}%
    \special{ar 1818 138 115 115 0 6.28319}%
    \special{ar 3199 368 115 115 0 6.28319}%
    \graphtemp=.5ex\advance\graphtemp by 0.690in
    \rlap{\kern 1.979in\lower\graphtemp\hbox to 0pt{\hss $y'$\hss}}%
    \graphtemp=.5ex\advance\graphtemp by 0.000in
    \rlap{\kern 1.979in\lower\graphtemp\hbox to 0pt{\hss $y$\hss}}%
    \graphtemp=.5ex\advance\graphtemp by 0.690in
    \rlap{\kern 2.393in\lower\graphtemp\hbox to 0pt{\hss $z'$\hss}}%
    \graphtemp=.5ex\advance\graphtemp by 0.000in
    \rlap{\kern 2.393in\lower\graphtemp\hbox to 0pt{\hss $z$\hss}}%
    \graphtemp=.5ex\advance\graphtemp by 0.257in
    \rlap{\kern 2.689in\lower\graphtemp\hbox to 0pt{\hss $v$\hss}}%
    \graphtemp=.5ex\advance\graphtemp by 0.257in
    \rlap{\kern 3.019in\lower\graphtemp\hbox to 0pt{\hss $x$\hss}}%
    \graphtemp=.5ex\advance\graphtemp by 0.598in
    \rlap{\kern 3.889in\lower\graphtemp\hbox to 0pt{\hss $\bu$\hss}}%
    \graphtemp=.5ex\advance\graphtemp by 0.138in
    \rlap{\kern 3.889in\lower\graphtemp\hbox to 0pt{\hss $\bu$\hss}}%
    \graphtemp=.5ex\advance\graphtemp by 0.138in
    \rlap{\kern 4.349in\lower\graphtemp\hbox to 0pt{\hss $\bu$\hss}}%
    \graphtemp=.5ex\advance\graphtemp by 0.138in
    \rlap{\kern 4.810in\lower\graphtemp\hbox to 0pt{\hss $\bu$\hss}}%
    \graphtemp=.5ex\advance\graphtemp by 0.598in
    \rlap{\kern 4.810in\lower\graphtemp\hbox to 0pt{\hss $\bu$\hss}}%
    \graphtemp=.5ex\advance\graphtemp by 0.138in
    \rlap{\kern 5.270in\lower\graphtemp\hbox to 0pt{\hss $\bu$\hss}}%
    \special{pn 28}%
    \special{pa 3889 138}%
    \special{pa 4349 138}%
    \special{fp}%
    \special{pa 4349 138}%
    \special{pa 4810 138}%
    \special{fp}%
    \special{pn 11}%
    \special{pa 3889 598}%
    \special{pa 3889 138}%
    \special{fp}%
    \special{pa 4810 138}%
    \special{pa 4810 598}%
    \special{fp}%
    \special{pa 4810 138}%
    \special{pa 5270 138}%
    \special{fp}%
    \special{pn 8}%
    \special{ar 3889 713 115 115 0 6.28319}%
    \special{ar 4810 713 115 115 0 6.28319}%
    \special{ar 5385 138 115 115 0 6.28319}%
    \graphtemp=.5ex\advance\graphtemp by 0.533in
    \rlap{\kern 3.778in\lower\graphtemp\hbox to 0pt{\hss $y'$\hss}}%
    \graphtemp=.5ex\advance\graphtemp by 0.000in
    \rlap{\kern 3.889in\lower\graphtemp\hbox to 0pt{\hss $y$\hss}}%
    \graphtemp=.5ex\advance\graphtemp by 0.000in
    \rlap{\kern 4.349in\lower\graphtemp\hbox to 0pt{\hss $z$\hss}}%
    \graphtemp=.5ex\advance\graphtemp by 0.000in
    \rlap{\kern 4.810in\lower\graphtemp\hbox to 0pt{\hss $v$\hss}}%
    \graphtemp=.5ex\advance\graphtemp by 0.533in
    \rlap{\kern 4.944in\lower\graphtemp\hbox to 0pt{\hss $x'$\hss}}%
    \graphtemp=.5ex\advance\graphtemp by 0.000in
    \rlap{\kern 5.224in\lower\graphtemp\hbox to 0pt{\hss $x$\hss}}%
    \hbox{\vrule depth0.828in width0pt height 0pt}%
    \kern 5.500in
  }%
}%
}
\caption{Three cases for Lemma~\ref{gamma}, with $F$ in bold\label{figgamma}}
\end{figure}

In employing Lemma~\ref{gamma}, the difficulty is ensuring that the edge $vx$
will be satisfied.  Generally, we will need to ensure that some edge is
satisfied regardless of the choice of weight on some incident edge.  When
$v$ is a $\gamma$-vertex, let $F_v$ denote the set of one or two edges
designated as $F$ in Lemma~\ref{gamma} (bold in Figure~\ref{figgamma}).

Like Lemma~\ref{triangle}, the next lemma excludes degenerate cases for the
reducibility proofs.

\begin{lem}\label{4cycle}
Let $z$ and $z'$ be $\beta$-vertices having respective $2$-neighbors $y$ and
$y'$ that are equal or adjacent.  The following cases lead to
$3$-reducible configurations:

(1) $zz'\in E(G)$.

(2) $z$ and $z'$ have a common neighbor $v$ with a $1$-neighbor $u$, such that
$d_G(v)\in\{4,5\}$.
 
\end{lem}
\begin{proof}
See Figure~\ref{fig4cyc} for these cases.

(1) Suppose $zz'\in E(G)$.  If $y=y'$, then Lemma~\ref{triangle} applies.
If $yy'\in E(G)$, then let $G'=G-\{zz',zy,yy',z'y'\}$.  Set $w(yy')=1$ to
ensure satisfying $zy$ and $z'y'$.  Set $w(zy)=1$.  Now choose $w(z'y')$ to
satisfy $yy'$ and $zz'$, and choose $w(zz')$ to satisfy $zx$ and $z'x'$, where
$N_G(z)=\{z',y,x\}$ and $N_G(z')=\{z,y',x'\}$ ($x=x'$ is allowed).

(2) By Case (1), we may assume $zz'\notin E(G)$, so $x\ne z'$ and $x'\ne z$.

(2a) If $y=y'$, then let $G'=G-\{vz,vz',zy,z'y',vu\}$.  Set $w(zy)=w(z'y')=1$
to ensure satisfying $zy$ and $z'y'$.  Next choose $w(zv)\in\{2,3\}$ to satisfy
$zx$ and choose $w(z'v)\in\{2,3\}$ to satisfy $z'x'$.  To ensure satisfying
$vz$ and $vz'$, we now choose $w(vu)$ to satisfy $\Gamma_{G'}(v)$ (if
$d_G(v)=5$) or $w(vu)\in\{2,3\}$ to satisfy $\Gamma_{G'}(v)$ (if $d_G(v)=4$).

(2b) If $yy'\in E(G)$, then let $G'=G-\{vz,vz',zy,z'y',yy',vu\}$.  Set
$w(yy')=1$ to ensure satisfying $zy$ and $z'y'$.
Set $w(zy)=1$ and $w(vz')=3$ to ensure satisfying $zv$.
Next choose $w(z'y')$ to satisfy $yy'$ and $z'x'$.
Two choices of $w(vz)$ will satisfy $zx$.  Along with the three choices
available for $w(vu)$
these choices can be made to satisfy $vz'$ and $\Gamma_{G'}(v)$.
\end{proof}

\begin{figure}[h]
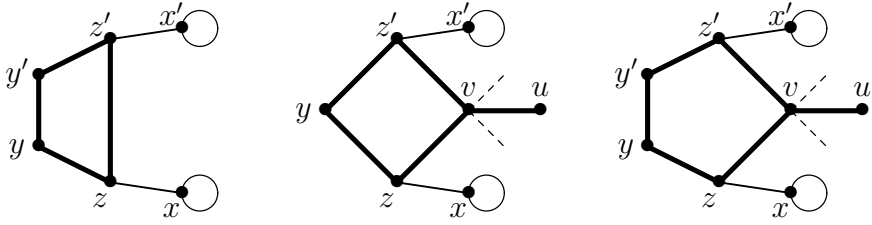

\gpic{
\expandafter\ifx\csname graph\endcsname\relax \csname newbox\endcsname\graph\fi
\expandafter\ifx\csname graphtemp\endcsname\relax \csname newdimen\endcsname\graphtemp\fi
\setbox\graph=\vtop{\vskip 0pt\hbox{%
    \graphtemp=.5ex\advance\graphtemp by 0.713in
    \rlap{\kern 0.113in\lower\graphtemp\hbox to 0pt{\hss $\bu$\hss}}%
    \graphtemp=.5ex\advance\graphtemp by 0.900in
    \rlap{\kern 0.488in\lower\graphtemp\hbox to 0pt{\hss $\bu$\hss}}%
    \graphtemp=.5ex\advance\graphtemp by 0.956in
    \rlap{\kern 0.863in\lower\graphtemp\hbox to 0pt{\hss $\bu$\hss}}%
    \graphtemp=.5ex\advance\graphtemp by 0.150in
    \rlap{\kern 0.488in\lower\graphtemp\hbox to 0pt{\hss $\bu$\hss}}%
    \graphtemp=.5ex\advance\graphtemp by 0.094in
    \rlap{\kern 0.863in\lower\graphtemp\hbox to 0pt{\hss $\bu$\hss}}%
    \graphtemp=.5ex\advance\graphtemp by 0.337in
    \rlap{\kern 0.113in\lower\graphtemp\hbox to 0pt{\hss $\bu$\hss}}%
    \special{pn 28}%
    \special{pa 113 337}%
    \special{pa 113 713}%
    \special{fp}%
    \special{pa 113 713}%
    \special{pa 488 900}%
    \special{fp}%
    \special{pa 488 900}%
    \special{pa 488 150}%
    \special{fp}%
    \special{pa 488 150}%
    \special{pa 113 337}%
    \special{fp}%
    \special{pn 11}%
    \special{pa 488 900}%
    \special{pa 863 956}%
    \special{fp}%
    \special{pa 488 150}%
    \special{pa 863 94}%
    \special{fp}%
    \special{pn 8}%
    \special{ar 956 956 94 94 0 6.28319}%
    \special{ar 956 94 94 94 0 6.28319}%
    \graphtemp=.5ex\advance\graphtemp by 0.713in
    \rlap{\kern 0.000in\lower\graphtemp\hbox to 0pt{\hss $y$\hss}}%
    \graphtemp=.5ex\advance\graphtemp by 0.991in
    \rlap{\kern 0.434in\lower\graphtemp\hbox to 0pt{\hss $z$\hss}}%
    \graphtemp=.5ex\advance\graphtemp by 1.047in
    \rlap{\kern 0.809in\lower\graphtemp\hbox to 0pt{\hss $x$\hss}}%
    \graphtemp=.5ex\advance\graphtemp by 0.097in
    \rlap{\kern 0.434in\lower\graphtemp\hbox to 0pt{\hss $z'$\hss}}%
    \graphtemp=.5ex\advance\graphtemp by 0.041in
    \rlap{\kern 0.809in\lower\graphtemp\hbox to 0pt{\hss $x'$\hss}}%
    \graphtemp=.5ex\advance\graphtemp by 0.337in
    \rlap{\kern 0.000in\lower\graphtemp\hbox to 0pt{\hss $y'$\hss}}%
    \graphtemp=.5ex\advance\graphtemp by 0.525in
    \rlap{\kern 1.613in\lower\graphtemp\hbox to 0pt{\hss $\bu$\hss}}%
    \graphtemp=.5ex\advance\graphtemp by 0.900in
    \rlap{\kern 1.988in\lower\graphtemp\hbox to 0pt{\hss $\bu$\hss}}%
    \graphtemp=.5ex\advance\graphtemp by 0.956in
    \rlap{\kern 2.363in\lower\graphtemp\hbox to 0pt{\hss $\bu$\hss}}%
    \graphtemp=.5ex\advance\graphtemp by 0.150in
    \rlap{\kern 1.988in\lower\graphtemp\hbox to 0pt{\hss $\bu$\hss}}%
    \graphtemp=.5ex\advance\graphtemp by 0.094in
    \rlap{\kern 2.363in\lower\graphtemp\hbox to 0pt{\hss $\bu$\hss}}%
    \graphtemp=.5ex\advance\graphtemp by 0.525in
    \rlap{\kern 2.363in\lower\graphtemp\hbox to 0pt{\hss $\bu$\hss}}%
    \graphtemp=.5ex\advance\graphtemp by 0.525in
    \rlap{\kern 2.738in\lower\graphtemp\hbox to 0pt{\hss $\bu$\hss}}%
    \special{pn 28}%
    \special{pa 2738 525}%
    \special{pa 2363 525}%
    \special{fp}%
    \special{pa 2363 525}%
    \special{pa 1988 150}%
    \special{fp}%
    \special{pa 1988 150}%
    \special{pa 1613 525}%
    \special{fp}%
    \special{pa 1613 525}%
    \special{pa 1988 900}%
    \special{fp}%
    \special{pa 1988 900}%
    \special{pa 2363 525}%
    \special{fp}%
    \special{pn 11}%
    \special{pa 1988 900}%
    \special{pa 2363 956}%
    \special{fp}%
    \special{pa 1988 150}%
    \special{pa 2363 94}%
    \special{fp}%
    \special{pa 2738 525}%
    \special{pa 2550 525}%
    \special{fp}%
    \special{pn 8}%
    \special{pa 2363 525}%
    \special{pa 2550 337}%
    \special{da 0.038}%
    \special{pa 2363 525}%
    \special{pa 2550 713}%
    \special{da 0.038}%
    \special{ar 2456 956 94 94 0 6.28319}%
    \special{ar 2456 94 94 94 0 6.28319}%
    \graphtemp=.5ex\advance\graphtemp by 0.525in
    \rlap{\kern 1.500in\lower\graphtemp\hbox to 0pt{\hss $y$\hss}}%
    \graphtemp=.5ex\advance\graphtemp by 0.991in
    \rlap{\kern 1.934in\lower\graphtemp\hbox to 0pt{\hss $z$\hss}}%
    \graphtemp=.5ex\advance\graphtemp by 1.047in
    \rlap{\kern 2.309in\lower\graphtemp\hbox to 0pt{\hss $x$\hss}}%
    \graphtemp=.5ex\advance\graphtemp by 0.097in
    \rlap{\kern 1.934in\lower\graphtemp\hbox to 0pt{\hss $z'$\hss}}%
    \graphtemp=.5ex\advance\graphtemp by 0.041in
    \rlap{\kern 2.309in\lower\graphtemp\hbox to 0pt{\hss $x'$\hss}}%
    \graphtemp=.5ex\advance\graphtemp by 0.412in
    \rlap{\kern 2.363in\lower\graphtemp\hbox to 0pt{\hss $v$\hss}}%
    \graphtemp=.5ex\advance\graphtemp by 0.412in
    \rlap{\kern 2.738in\lower\graphtemp\hbox to 0pt{\hss $u$\hss}}%
    \graphtemp=.5ex\advance\graphtemp by 0.713in
    \rlap{\kern 3.300in\lower\graphtemp\hbox to 0pt{\hss $\bu$\hss}}%
    \graphtemp=.5ex\advance\graphtemp by 0.900in
    \rlap{\kern 3.675in\lower\graphtemp\hbox to 0pt{\hss $\bu$\hss}}%
    \graphtemp=.5ex\advance\graphtemp by 0.956in
    \rlap{\kern 4.050in\lower\graphtemp\hbox to 0pt{\hss $\bu$\hss}}%
    \graphtemp=.5ex\advance\graphtemp by 0.150in
    \rlap{\kern 3.675in\lower\graphtemp\hbox to 0pt{\hss $\bu$\hss}}%
    \graphtemp=.5ex\advance\graphtemp by 0.094in
    \rlap{\kern 4.050in\lower\graphtemp\hbox to 0pt{\hss $\bu$\hss}}%
    \graphtemp=.5ex\advance\graphtemp by 0.525in
    \rlap{\kern 4.050in\lower\graphtemp\hbox to 0pt{\hss $\bu$\hss}}%
    \graphtemp=.5ex\advance\graphtemp by 0.525in
    \rlap{\kern 4.425in\lower\graphtemp\hbox to 0pt{\hss $\bu$\hss}}%
    \graphtemp=.5ex\advance\graphtemp by 0.337in
    \rlap{\kern 3.300in\lower\graphtemp\hbox to 0pt{\hss $\bu$\hss}}%
    \special{pn 28}%
    \special{pa 4425 525}%
    \special{pa 4050 525}%
    \special{fp}%
    \special{pa 4050 525}%
    \special{pa 3675 150}%
    \special{fp}%
    \special{pa 3675 150}%
    \special{pa 3300 337}%
    \special{fp}%
    \special{pa 3300 337}%
    \special{pa 3300 713}%
    \special{fp}%
    \special{pa 3300 713}%
    \special{pa 3675 900}%
    \special{fp}%
    \special{pa 3675 900}%
    \special{pa 4050 525}%
    \special{fp}%
    \special{pn 11}%
    \special{pa 3675 900}%
    \special{pa 4050 956}%
    \special{fp}%
    \special{pa 3675 150}%
    \special{pa 4050 94}%
    \special{fp}%
    \special{pn 8}%
    \special{pa 4050 525}%
    \special{pa 4237 337}%
    \special{da 0.038}%
    \special{pa 4050 525}%
    \special{pa 4237 713}%
    \special{da 0.038}%
    \special{ar 4144 956 94 94 0 6.28319}%
    \special{ar 4144 94 94 94 0 6.28319}%
    \graphtemp=.5ex\advance\graphtemp by 0.713in
    \rlap{\kern 3.188in\lower\graphtemp\hbox to 0pt{\hss $y$\hss}}%
    \graphtemp=.5ex\advance\graphtemp by 0.991in
    \rlap{\kern 3.622in\lower\graphtemp\hbox to 0pt{\hss $z$\hss}}%
    \graphtemp=.5ex\advance\graphtemp by 1.047in
    \rlap{\kern 3.997in\lower\graphtemp\hbox to 0pt{\hss $x$\hss}}%
    \graphtemp=.5ex\advance\graphtemp by 0.097in
    \rlap{\kern 3.622in\lower\graphtemp\hbox to 0pt{\hss $z'$\hss}}%
    \graphtemp=.5ex\advance\graphtemp by 0.041in
    \rlap{\kern 3.997in\lower\graphtemp\hbox to 0pt{\hss $x'$\hss}}%
    \graphtemp=.5ex\advance\graphtemp by 0.412in
    \rlap{\kern 4.050in\lower\graphtemp\hbox to 0pt{\hss $v$\hss}}%
    \graphtemp=.5ex\advance\graphtemp by 0.412in
    \rlap{\kern 4.425in\lower\graphtemp\hbox to 0pt{\hss $u$\hss}}%
    \graphtemp=.5ex\advance\graphtemp by 0.337in
    \rlap{\kern 3.188in\lower\graphtemp\hbox to 0pt{\hss $y'$\hss}}%
    \hbox{\vrule depth1.069in width0pt height 0pt}%
    \kern 4.500in
  }%
}%
}
\caption{Three cases for Lemma~\ref{4cycle}\label{fig4cyc}}
\end{figure}

\begin{lem}\label{83red3}
The following configurations are $3$-reducible.

{\bf A}. A $2$-vertex or $3$-vertex having a $1$-neighbor.

{\bf B}. A $4^-$-vertex whose neighbors all have degree $2$.

{\bf C}. A $3$-vertex having an $\alpha$-neighbor and another $2$-neighbor.

{\bf D}. A $4$-vertex having a $1$-neighbor and a $2^-$-neighbor.

{\bf E}. A $5^+$-vertex $v$ with
$3p_1+2p_2\ge d(v)$, where $p_i$ is the number of $i$-neighbors of $v$.

{\bf F}. Two adjacent $\gamma$-vertices.

{\bf G}. A $3$-vertex with two $\gamma$-neighbors.

{\bf H}. A vertex $v$ such that $p_1+2q\ge d(v)$ and $p_1+q> 4$, where $p_1$
is the number of $1$-neighbors and $q$ is the number of $\gamma$-neighbors of
$v$.

{\bf I}. A $5$-vertex having a $1$-neighbor and three $\gamma$-neighbors.

{\bf J}. A $4$-vertex with (1) two $\alpha$-neighbors, (2) an $\alpha$-neighbor,
another $2$-neighbor, and a $\gamma$-neighbor, or (3) a $2$-neighbor and three
$\gamma$-neighbors.

{\bf K}. A $\gamma$-vertex whose $3^+$-neighbors are all $\beta$-vertices.
\end{lem}

\begin{proof}
Lemma~\ref{3config} shows that {\bf A-E} are $3$-reducible.  Let $G$ be a
minimal $3$-bad graph containing one of {\bf F-K}.  When $z$ is a
$\gamma$-vertex with a $3^+$-neighbor $v$ and $w(zv)$ has been chosen, the
phrase ``apply Lemma~\ref{gamma} to $z$'' means ``apply Lemma~\ref{gamma} to
choose weights on $F_z$ to satisfy $F_z$ and the edges incident to them other
than $zv$''.  In each case, we extend a proper $3$-weighting $w'$ of a proper
subgraph $G'$ of $G$ to a proper $3$-weighting $w$ of $G$.

The phrase ``Figure $n$ is accurate'' means that the vertices in the
illustration are known to be distinct, except possibly for
non-$\gamma$-vertices on circles, to which no edges of the core are adjacent.
There are three types of $\gamma$-vertices.  Let those of degree $4$ be
{\it $\gamma_4$-vertices}, those of degree $3$ with an $\alpha$-neighbor be
{\it $\gamma_{3a}$-vertices}, and those of degree $3$ with two $2$-neighbors be
{\it $\gamma_{3b}$-vertices}.  

\smallskip
{\bf Case F:} {\it $v$ and $v'$ are adjacent $\gamma$-vertices.}
Let $G'=G-vv'-F_v-F_{v'}$.  By symmetry, the first subcase covers when $v$ or
$v'$ is a $\gamma_{3b}$-vertex.  By Lemma~\ref{triangle}, $v$ and $v'$ do
not have a common $2$-neighbor.

{\bf Subcase 1:} {\it $v$ is a $\gamma_{3b}$-vertex.}
Let $N_G(v)=\{v',z,z'\}$.  By Lemma~\ref{triangle}, we may assume
$zz',zv',z'v'\notin E(G)$, so Figure~\ref{figF12} is accurate.  Set $w(vv')=3$
to ensure satisfying  $vz$ and $vz'$.  Apply Lemma~\ref{gamma} to $v'$.
Choose $w(vz)$ to satisfy $\Gamma_{G'}(z)$.  Choose $w(vz')$ to satisfy $vv'$
and $\Gamma_{G'}(z')$.  

\begin{figure}[h]
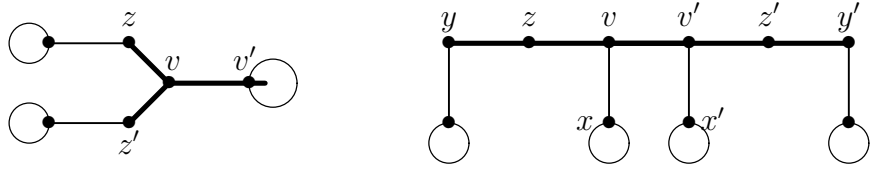

\gpic{
\expandafter\ifx\csname graph\endcsname\relax \csname newbox\endcsname\graph\fi
\expandafter\ifx\csname graphtemp\endcsname\relax \csname newdimen\endcsname\graphtemp\fi
\setbox\graph=\vtop{\vskip 0pt\hbox{%
    \graphtemp=.5ex\advance\graphtemp by 0.544in
    \rlap{\kern 0.209in\lower\graphtemp\hbox to 0pt{\hss $\bu$\hss}}%
    \graphtemp=.5ex\advance\graphtemp by 0.126in
    \rlap{\kern 0.209in\lower\graphtemp\hbox to 0pt{\hss $\bu$\hss}}%
    \graphtemp=.5ex\advance\graphtemp by 0.544in
    \rlap{\kern 0.628in\lower\graphtemp\hbox to 0pt{\hss $\bu$\hss}}%
    \graphtemp=.5ex\advance\graphtemp by 0.126in
    \rlap{\kern 0.628in\lower\graphtemp\hbox to 0pt{\hss $\bu$\hss}}%
    \graphtemp=.5ex\advance\graphtemp by 0.335in
    \rlap{\kern 0.837in\lower\graphtemp\hbox to 0pt{\hss $\bu$\hss}}%
    \graphtemp=.5ex\advance\graphtemp by 0.335in
    \rlap{\kern 1.256in\lower\graphtemp\hbox to 0pt{\hss $\bu$\hss}}%
    \special{pn 28}%
    \special{pa 837 335}%
    \special{pa 1340 335}%
    \special{fp}%
    \special{pa 837 335}%
    \special{pa 628 544}%
    \special{fp}%
    \special{pa 837 335}%
    \special{pa 628 126}%
    \special{fp}%
    \special{pn 11}%
    \special{pa 209 544}%
    \special{pa 628 544}%
    \special{fp}%
    \special{pa 209 126}%
    \special{pa 628 126}%
    \special{fp}%
    \special{pn 8}%
    \special{ar 105 544 105 105 0 6.28319}%
    \special{ar 105 126 105 105 0 6.28319}%
    \special{ar 1381 335 126 126 0 6.28319}%
    \graphtemp=.5ex\advance\graphtemp by 0.670in
    \rlap{\kern 0.628in\lower\graphtemp\hbox to 0pt{\hss $z'$\hss}}%
    \graphtemp=.5ex\advance\graphtemp by 0.000in
    \rlap{\kern 0.628in\lower\graphtemp\hbox to 0pt{\hss $z$\hss}}%
    \graphtemp=.5ex\advance\graphtemp by 0.230in
    \rlap{\kern 0.858in\lower\graphtemp\hbox to 0pt{\hss $v$\hss}}%
    \graphtemp=.5ex\advance\graphtemp by 0.230in
    \rlap{\kern 1.235in\lower\graphtemp\hbox to 0pt{\hss $v'$\hss}}%
    \graphtemp=.5ex\advance\graphtemp by 0.544in
    \rlap{\kern 2.302in\lower\graphtemp\hbox to 0pt{\hss $\bu$\hss}}%
    \graphtemp=.5ex\advance\graphtemp by 0.126in
    \rlap{\kern 2.302in\lower\graphtemp\hbox to 0pt{\hss $\bu$\hss}}%
    \graphtemp=.5ex\advance\graphtemp by 0.126in
    \rlap{\kern 2.721in\lower\graphtemp\hbox to 0pt{\hss $\bu$\hss}}%
    \graphtemp=.5ex\advance\graphtemp by 0.126in
    \rlap{\kern 3.140in\lower\graphtemp\hbox to 0pt{\hss $\bu$\hss}}%
    \graphtemp=.5ex\advance\graphtemp by 0.544in
    \rlap{\kern 3.140in\lower\graphtemp\hbox to 0pt{\hss $\bu$\hss}}%
    \graphtemp=.5ex\advance\graphtemp by 0.000in
    \rlap{\kern 2.302in\lower\graphtemp\hbox to 0pt{\hss $y$\hss}}%
    \graphtemp=.5ex\advance\graphtemp by 0.000in
    \rlap{\kern 2.721in\lower\graphtemp\hbox to 0pt{\hss $z$\hss}}%
    \graphtemp=.5ex\advance\graphtemp by 0.000in
    \rlap{\kern 3.140in\lower\graphtemp\hbox to 0pt{\hss $v$\hss}}%
    \graphtemp=.5ex\advance\graphtemp by 0.544in
    \rlap{\kern 3.014in\lower\graphtemp\hbox to 0pt{\hss $x$\hss}}%
    \special{ar 2302 649 105 105 0 6.28319}%
    \special{ar 3140 649 105 105 0 6.28319}%
    \special{pn 11}%
    \special{pa 2302 544}%
    \special{pa 2302 126}%
    \special{fp}%
    \special{pa 3140 126}%
    \special{pa 3140 544}%
    \special{fp}%
    \special{pn 28}%
    \special{pa 2302 126}%
    \special{pa 3558 126}%
    \special{fp}%
    \graphtemp=.5ex\advance\graphtemp by 0.544in
    \rlap{\kern 4.395in\lower\graphtemp\hbox to 0pt{\hss $\bu$\hss}}%
    \graphtemp=.5ex\advance\graphtemp by 0.126in
    \rlap{\kern 4.395in\lower\graphtemp\hbox to 0pt{\hss $\bu$\hss}}%
    \graphtemp=.5ex\advance\graphtemp by 0.126in
    \rlap{\kern 3.977in\lower\graphtemp\hbox to 0pt{\hss $\bu$\hss}}%
    \graphtemp=.5ex\advance\graphtemp by 0.126in
    \rlap{\kern 3.558in\lower\graphtemp\hbox to 0pt{\hss $\bu$\hss}}%
    \graphtemp=.5ex\advance\graphtemp by 0.544in
    \rlap{\kern 3.558in\lower\graphtemp\hbox to 0pt{\hss $\bu$\hss}}%
    \graphtemp=.5ex\advance\graphtemp by 0.000in
    \rlap{\kern 4.395in\lower\graphtemp\hbox to 0pt{\hss $y'$\hss}}%
    \graphtemp=.5ex\advance\graphtemp by 0.000in
    \rlap{\kern 3.977in\lower\graphtemp\hbox to 0pt{\hss $z'$\hss}}%
    \graphtemp=.5ex\advance\graphtemp by 0.000in
    \rlap{\kern 3.558in\lower\graphtemp\hbox to 0pt{\hss $v'$\hss}}%
    \graphtemp=.5ex\advance\graphtemp by 0.544in
    \rlap{\kern 3.684in\lower\graphtemp\hbox to 0pt{\hss $x'$\hss}}%
    \special{pn 8}%
    \special{ar 4395 649 105 105 0 6.28319}%
    \special{ar 3558 649 105 105 0 6.28319}%
    \special{pn 11}%
    \special{pa 4395 544}%
    \special{pa 4395 126}%
    \special{fp}%
    \special{pa 3558 126}%
    \special{pa 3558 544}%
    \special{fp}%
    \special{pn 28}%
    \special{pa 4395 126}%
    \special{pa 3140 126}%
    \special{fp}%
    \hbox{\vrule depth0.753in width0pt height 0pt}%
    \kern 4.500in
  }%
}%
}
\caption{Cases {\bf F1} and {\bf F2} for Lemma~\ref{83red3}\label{figF12}}
\end{figure}

{\bf Subcase 2:} {\it $v$ and $v'$ are both $\gamma_{3a}$-vertices.}
Let $z$ and $z'$ be the $\alpha$-neighbors of $v$ and $v'$, and let
$y$ and $y'$ be the $2$-neighbors of $z$ and $z'$.  By Lemma~\ref{4cycle},
Lemma~\ref{triangle}, and {\bf B}, Figure~\ref{figF12} is accurate.  Let $x$
and $x'$ be the remaining neighbors of $v$ and $v'$, respectively.

Choose $w(vz)$ to satisfy $zy$.  Now choose $w(v'z')$ to satisfy $vv'$ and
$z'y'$.  Next choose $w(vv')$ to satisfy $vx$ and $v'x'$.  Finally, choose
$w(zy)$ to satisfy $vz$ and $\Gamma_{G'}(y)$, and choose $w(z'y')$ to satisfy
$v'z'$ and $\Gamma_{G'}(y')$.

{\bf Subcase 3:} {\it $v$ and $v'$ are both $\gamma_{4}$-vertices.}
Let $N(v)=\{v',z_1,z_2,u\}$ and $N(v')=\{v,z_1',z_2',u'\}$, with
$d_G(u)=d_G(u')=1$.  For $i\in\{1,2\}$, let $a_i=w'(vz_i)$ and $b_i=\rp{z_i}v$;
similarly define $a_i'$ and $b_i'$ using $\{v',z_1',z_2'\}$.  Since
$d_G(u),d_G(u')=1$, Figure~\ref{figF34} is accurate.

If $a_1+a_2>a_1'+a_2'$, then set $w(uv)=3$ to ensure satisfying $vv'$.  Next
choose $w(vv')$ to satisfy $vz_1$ and $vz_2$, and choose $w(v'u')$ to satisfy
$v'z_1'$ and $v'z_2'$.

By symmetry, we may thus assume $a_1+a_2=a_1'+a_2'$.  If $b_1\ne a_2+4$, then
choose $w(vv')\in\{1,3\}$ to ensure satisfying $vz_1$.  Now choose
$w(v'u')$ to satisfy $v'z_1'$ and $v'z_2'$, and then choose $w(vu)$ to satisfy
$vz_2$ and $vv'$.  Hence by symmetry we may assume that each entry of
$(b_1,b_2,b_1',b_2')$ exceeds the corresponding entry of $(a_2,a_1,a_2',a_1')$
by exactly $4$.  Now setting $w(uv)=w(vv')=3$ and $w(v'u')=2$ completes the
extension.

\begin{figure}[h]
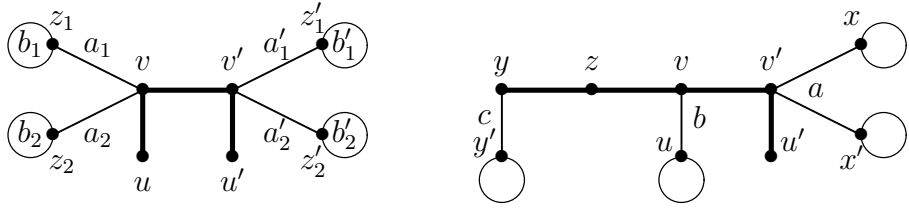

\gpic{
\expandafter\ifx\csname graph\endcsname\relax \csname newbox\endcsname\graph\fi
\expandafter\ifx\csname graphtemp\endcsname\relax \csname newdimen\endcsname\graphtemp\fi
\setbox\graph=\vtop{\vskip 0pt\hbox{%
    \graphtemp=.5ex\advance\graphtemp by 0.611in
    \rlap{\kern 0.235in\lower\graphtemp\hbox to 0pt{\hss $\bu$\hss}}%
    \graphtemp=.5ex\advance\graphtemp by 0.141in
    \rlap{\kern 0.235in\lower\graphtemp\hbox to 0pt{\hss $\bu$\hss}}%
    \graphtemp=.5ex\advance\graphtemp by 0.376in
    \rlap{\kern 0.705in\lower\graphtemp\hbox to 0pt{\hss $\bu$\hss}}%
    \graphtemp=.5ex\advance\graphtemp by 0.729in
    \rlap{\kern 0.705in\lower\graphtemp\hbox to 0pt{\hss $\bu$\hss}}%
    \graphtemp=.5ex\advance\graphtemp by 0.729in
    \rlap{\kern 1.175in\lower\graphtemp\hbox to 0pt{\hss $\bu$\hss}}%
    \graphtemp=.5ex\advance\graphtemp by 0.376in
    \rlap{\kern 1.175in\lower\graphtemp\hbox to 0pt{\hss $\bu$\hss}}%
    \graphtemp=.5ex\advance\graphtemp by 0.611in
    \rlap{\kern 1.645in\lower\graphtemp\hbox to 0pt{\hss $\bu$\hss}}%
    \graphtemp=.5ex\advance\graphtemp by 0.141in
    \rlap{\kern 1.645in\lower\graphtemp\hbox to 0pt{\hss $\bu$\hss}}%
    \special{pn 28}%
    \special{pa 705 729}%
    \special{pa 705 376}%
    \special{fp}%
    \special{pa 705 376}%
    \special{pa 1175 376}%
    \special{fp}%
    \special{pa 1175 376}%
    \special{pa 1175 729}%
    \special{fp}%
    \special{pn 11}%
    \special{pa 235 611}%
    \special{pa 705 376}%
    \special{fp}%
    \special{pa 705 376}%
    \special{pa 235 141}%
    \special{fp}%
    \special{pa 1645 141}%
    \special{pa 1175 376}%
    \special{fp}%
    \special{pa 1175 376}%
    \special{pa 1645 611}%
    \special{fp}%
    \special{pn 8}%
    \special{ar 118 611 118 118 0 6.28319}%
    \special{ar 118 141 118 118 0 6.28319}%
    \special{ar 1762 611 118 118 0 6.28319}%
    \special{ar 1762 141 118 118 0 6.28319}%
    \graphtemp=.5ex\advance\graphtemp by 0.611in
    \rlap{\kern 0.118in\lower\graphtemp\hbox to 0pt{\hss $b_2$\hss}}%
    \graphtemp=.5ex\advance\graphtemp by 0.141in
    \rlap{\kern 0.118in\lower\graphtemp\hbox to 0pt{\hss $b_1$\hss}}%
    \graphtemp=.5ex\advance\graphtemp by 0.752in
    \rlap{\kern 0.282in\lower\graphtemp\hbox to 0pt{\hss $z_2$\hss}}%
    \graphtemp=.5ex\advance\graphtemp by 0.000in
    \rlap{\kern 0.282in\lower\graphtemp\hbox to 0pt{\hss $z_1$\hss}}%
    \graphtemp=.5ex\advance\graphtemp by 0.611in
    \rlap{\kern 0.470in\lower\graphtemp\hbox to 0pt{\hss $a_2$\hss}}%
    \graphtemp=.5ex\advance\graphtemp by 0.141in
    \rlap{\kern 0.470in\lower\graphtemp\hbox to 0pt{\hss $a_1$\hss}}%
    \graphtemp=.5ex\advance\graphtemp by 0.235in
    \rlap{\kern 0.705in\lower\graphtemp\hbox to 0pt{\hss $v$\hss}}%
    \graphtemp=.5ex\advance\graphtemp by 0.870in
    \rlap{\kern 0.705in\lower\graphtemp\hbox to 0pt{\hss $u$\hss}}%
    \graphtemp=.5ex\advance\graphtemp by 0.235in
    \rlap{\kern 1.175in\lower\graphtemp\hbox to 0pt{\hss $v'$\hss}}%
    \graphtemp=.5ex\advance\graphtemp by 0.870in
    \rlap{\kern 1.175in\lower\graphtemp\hbox to 0pt{\hss $u'$\hss}}%
    \graphtemp=.5ex\advance\graphtemp by 0.611in
    \rlap{\kern 1.410in\lower\graphtemp\hbox to 0pt{\hss $a_2'$\hss}}%
    \graphtemp=.5ex\advance\graphtemp by 0.141in
    \rlap{\kern 1.410in\lower\graphtemp\hbox to 0pt{\hss $a_1'$\hss}}%
    \graphtemp=.5ex\advance\graphtemp by 0.752in
    \rlap{\kern 1.598in\lower\graphtemp\hbox to 0pt{\hss $z_2'$\hss}}%
    \graphtemp=.5ex\advance\graphtemp by 0.000in
    \rlap{\kern 1.598in\lower\graphtemp\hbox to 0pt{\hss $z_1'$\hss}}%
    \graphtemp=.5ex\advance\graphtemp by 0.611in
    \rlap{\kern 1.762in\lower\graphtemp\hbox to 0pt{\hss $b_2'$\hss}}%
    \graphtemp=.5ex\advance\graphtemp by 0.141in
    \rlap{\kern 1.762in\lower\graphtemp\hbox to 0pt{\hss $b_1'$\hss}}%
    \graphtemp=.5ex\advance\graphtemp by 0.729in
    \rlap{\kern 2.585in\lower\graphtemp\hbox to 0pt{\hss $\bu$\hss}}%
    \graphtemp=.5ex\advance\graphtemp by 0.376in
    \rlap{\kern 2.585in\lower\graphtemp\hbox to 0pt{\hss $\bu$\hss}}%
    \graphtemp=.5ex\advance\graphtemp by 0.376in
    \rlap{\kern 3.055in\lower\graphtemp\hbox to 0pt{\hss $\bu$\hss}}%
    \graphtemp=.5ex\advance\graphtemp by 0.376in
    \rlap{\kern 3.525in\lower\graphtemp\hbox to 0pt{\hss $\bu$\hss}}%
    \graphtemp=.5ex\advance\graphtemp by 0.729in
    \rlap{\kern 3.525in\lower\graphtemp\hbox to 0pt{\hss $\bu$\hss}}%
    \graphtemp=.5ex\advance\graphtemp by 0.729in
    \rlap{\kern 3.995in\lower\graphtemp\hbox to 0pt{\hss $\bu$\hss}}%
    \graphtemp=.5ex\advance\graphtemp by 0.376in
    \rlap{\kern 3.995in\lower\graphtemp\hbox to 0pt{\hss $\bu$\hss}}%
    \graphtemp=.5ex\advance\graphtemp by 0.611in
    \rlap{\kern 4.465in\lower\graphtemp\hbox to 0pt{\hss $\bu$\hss}}%
    \graphtemp=.5ex\advance\graphtemp by 0.141in
    \rlap{\kern 4.465in\lower\graphtemp\hbox to 0pt{\hss $\bu$\hss}}%
    \special{pn 28}%
    \special{pa 2585 376}%
    \special{pa 3995 376}%
    \special{fp}%
    \special{pa 3995 376}%
    \special{pa 3995 729}%
    \special{fp}%
    \special{pn 11}%
    \special{pa 2585 729}%
    \special{pa 2585 376}%
    \special{fp}%
    \special{pa 3525 376}%
    \special{pa 3525 729}%
    \special{fp}%
    \special{pa 3995 376}%
    \special{pa 4465 611}%
    \special{fp}%
    \special{pa 3995 376}%
    \special{pa 4465 141}%
    \special{fp}%
    \special{pn 8}%
    \special{ar 2585 846 118 118 0 6.28319}%
    \special{ar 3525 846 118 118 0 6.28319}%
    \special{ar 4583 611 118 118 0 6.28319}%
    \special{ar 4583 141 118 118 0 6.28319}%
    \graphtemp=.5ex\advance\graphtemp by 0.662in
    \rlap{\kern 2.495in\lower\graphtemp\hbox to 0pt{\hss $y'$\hss}}%
    \graphtemp=.5ex\advance\graphtemp by 0.235in
    \rlap{\kern 2.585in\lower\graphtemp\hbox to 0pt{\hss $y$\hss}}%
    \graphtemp=.5ex\advance\graphtemp by 0.235in
    \rlap{\kern 3.055in\lower\graphtemp\hbox to 0pt{\hss $z$\hss}}%
    \graphtemp=.5ex\advance\graphtemp by 0.235in
    \rlap{\kern 3.525in\lower\graphtemp\hbox to 0pt{\hss $v$\hss}}%
    \graphtemp=.5ex\advance\graphtemp by 0.662in
    \rlap{\kern 3.435in\lower\graphtemp\hbox to 0pt{\hss $u$\hss}}%
    \graphtemp=.5ex\advance\graphtemp by 0.235in
    \rlap{\kern 3.995in\lower\graphtemp\hbox to 0pt{\hss $v'$\hss}}%
    \graphtemp=.5ex\advance\graphtemp by 0.662in
    \rlap{\kern 4.108in\lower\graphtemp\hbox to 0pt{\hss $u'$\hss}}%
    \graphtemp=.5ex\advance\graphtemp by 0.000in
    \rlap{\kern 4.418in\lower\graphtemp\hbox to 0pt{\hss $x$\hss}}%
    \graphtemp=.5ex\advance\graphtemp by 0.752in
    \rlap{\kern 4.418in\lower\graphtemp\hbox to 0pt{\hss $x'$\hss}}%
    \graphtemp=.5ex\advance\graphtemp by 0.517in
    \rlap{\kern 2.491in\lower\graphtemp\hbox to 0pt{\hss $c$\hss}}%
    \graphtemp=.5ex\advance\graphtemp by 0.540in
    \rlap{\kern 3.619in\lower\graphtemp\hbox to 0pt{\hss $b$\hss}}%
    \graphtemp=.5ex\advance\graphtemp by 0.376in
    \rlap{\kern 4.230in\lower\graphtemp\hbox to 0pt{\hss $a$\hss}}%
    \hbox{\vrule depth0.964in width0pt height 0pt}%
    \kern 4.700in
  }%
}%
}
\caption{Cases {\bf F3} and {\bf F4} for Lemma~\ref{83red3}\label{figF34}}
\end{figure}

{\bf Subcase 4:} {\it $v$ is a $\gamma_{3a}$-vertex and $v'$ is a
$\gamma_{4}$-vertex.} 
Let $N_G(v)=\{v',z,u\}$, with $z$ being the $\alpha$-vertex.  Let $y$ be the 
$2$-neighbor of $z$, with $N_G(y)=\{z,y'\}$.  Let $N_G(v')=\{v,x,x',u'\}$, with
$d_G(u')=1$.  By Lemma~\ref{triangle} and {\bf D}, Figure~\ref{figF34} is
accurate.  Let $a=w'(v'x)+w'(v'x')$, $b=w'(vu)$, and $c=w'(yy')$.

If $a\ne b$, then choose $w(u'v')\in\{1,3\}$ to ensure satisfying $vv'$.
Now choose $w(vv')$ to satisfy $v'x$ and $v'x'$, choose $w(vz)$ to satisfy
$vu$ and $yz$, and choose $w(zy)$ to satisfy $vz$ and the other edge at $y$.

If $a=b$, then set $w(u'v')=c$.  Now choose $w(vv')$ to satisfy $v'x$ and
$v'x'$.  Next choose $w(vz)$ to satisfy $vv'$, $yz$, and $uv$, which succeeds
because $vv'$ and $yz$ both are satisfied if and only if $w(vz)\ne c$.
Finally choose $w(zy)$ to satisfy $vz$ and the other edge at $y$.

{\bf Case G:} {\it $v$ is a $3$-vertex having two $\gamma$-neighbors $z$ and
$z'$.} 
Let $N_G(v)=\{z,z',x\}$.  In each case, we will
let $G'=G-\{vz,vz'\}-F_z-F_{z'}$ (and $w'$ is a proper $3$-weighting of $G'$).
By {\bf F}, $zz'\notin E(G)$.  By Lemma~\ref{triangle}, neither $z$
nor $z'$ shares a $2$-neighbor with $v$ or has a $2$-neighbor adjacent to a
$2$-neighbor of $v$.  Hence Figures~\ref{figG12} and~\ref{figG3} are accurate.

{\bf Subcase 1:} {\it $z$ is a $\gamma_{3a}$-vertex.}
Let $N_G(z)=\{v,y,u\}$, with $y$ the $\alpha$-neighbor of $z$.  Let $y'$ be the
$2$-neighbor of $y$, with $N_G(y')=\{y,y''\}$.  Let $a=w'(vx)$, $b=w'(zu)$, and
$c=w'(y'y'')$, as shown on the left in Figure~\ref{figG12}.

If $b\ne a$, then choose $w(vz')\in\{1,3\}$ to ensure satisfying $vz$.
With $w(vz')$ known, apply Lemma~\ref{gamma} to $z'$.  Next choose $w(vz)$ to
satisfy $vx$ and $vz'$.  With $vz$ automatically satisfied and $w(vz)$ chosen,
it now suffices to apply Lemma~\ref{gamma} to $z$.

Hence we may assume $b=a$.  In this case, set $w(vz')=c$ and apply
Lemma~\ref{gamma} to $z'$.  Choose $w(vz)$ to satisfy $vx$ and $vz'$.
Now choose $w(zy)$ to satisfy $vz$, $zu$, and $yy'$; this succeeds because
$yy'$ and $vz$ both forbid $w(zy)=c$, so at most two choices for $w(zy)$ are
forbidden.  Finally, choose $w(yy')$ to satisfy $zy$ and $y'y''$.

\begin{figure}[h]
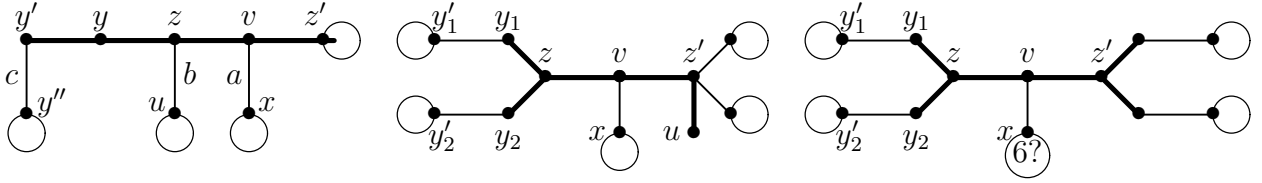

\gpic{
\expandafter\ifx\csname graph\endcsname\relax \csname newbox\endcsname\graph\fi
\expandafter\ifx\csname graphtemp\endcsname\relax \csname newdimen\endcsname\graphtemp\fi
\setbox\graph=\vtop{\vskip 0pt\hbox{%
    \graphtemp=.5ex\advance\graphtemp by 0.504in
    \rlap{\kern 0.097in\lower\graphtemp\hbox to 0pt{\hss $\bu$\hss}}%
    \graphtemp=.5ex\advance\graphtemp by 0.116in
    \rlap{\kern 0.097in\lower\graphtemp\hbox to 0pt{\hss $\bu$\hss}}%
    \graphtemp=.5ex\advance\graphtemp by 0.116in
    \rlap{\kern 0.485in\lower\graphtemp\hbox to 0pt{\hss $\bu$\hss}}%
    \graphtemp=.5ex\advance\graphtemp by 0.504in
    \rlap{\kern 0.873in\lower\graphtemp\hbox to 0pt{\hss $\bu$\hss}}%
    \graphtemp=.5ex\advance\graphtemp by 0.116in
    \rlap{\kern 0.873in\lower\graphtemp\hbox to 0pt{\hss $\bu$\hss}}%
    \graphtemp=.5ex\advance\graphtemp by 0.504in
    \rlap{\kern 1.261in\lower\graphtemp\hbox to 0pt{\hss $\bu$\hss}}%
    \graphtemp=.5ex\advance\graphtemp by 0.116in
    \rlap{\kern 1.261in\lower\graphtemp\hbox to 0pt{\hss $\bu$\hss}}%
    \graphtemp=.5ex\advance\graphtemp by 0.116in
    \rlap{\kern 1.649in\lower\graphtemp\hbox to 0pt{\hss $\bu$\hss}}%
    \special{pn 28}%
    \special{pa 97 116}%
    \special{pa 485 116}%
    \special{fp}%
    \special{pa 485 116}%
    \special{pa 873 116}%
    \special{fp}%
    \special{pa 873 116}%
    \special{pa 1261 116}%
    \special{fp}%
    \special{pa 1261 116}%
    \special{pa 1707 116}%
    \special{fp}%
    \special{pn 11}%
    \special{pa 97 504}%
    \special{pa 97 116}%
    \special{fp}%
    \special{pa 873 504}%
    \special{pa 873 116}%
    \special{fp}%
    \special{pa 1261 116}%
    \special{pa 1261 504}%
    \special{fp}%
    \graphtemp=.5ex\advance\graphtemp by 0.000in
    \rlap{\kern 1.261in\lower\graphtemp\hbox to 0pt{\hss $v$\hss}}%
    \graphtemp=.5ex\advance\graphtemp by 0.450in
    \rlap{\kern 1.355in\lower\graphtemp\hbox to 0pt{\hss $x$\hss}}%
    \special{pn 8}%
    \special{ar 97 601 97 97 0 6.28319}%
    \special{ar 873 601 97 97 0 6.28319}%
    \special{ar 1261 601 97 97 0 6.28319}%
    \special{ar 1746 116 97 97 0 6.28319}%
    \graphtemp=.5ex\advance\graphtemp by 0.000in
    \rlap{\kern 0.097in\lower\graphtemp\hbox to 0pt{\hss $y'$\hss}}%
    \graphtemp=.5ex\advance\graphtemp by 0.000in
    \rlap{\kern 0.485in\lower\graphtemp\hbox to 0pt{\hss $y$\hss}}%
    \graphtemp=.5ex\advance\graphtemp by 0.000in
    \rlap{\kern 1.610in\lower\graphtemp\hbox to 0pt{\hss $z'$\hss}}%
    \graphtemp=.5ex\advance\graphtemp by 0.000in
    \rlap{\kern 0.873in\lower\graphtemp\hbox to 0pt{\hss $z$\hss}}%
    \graphtemp=.5ex\advance\graphtemp by 0.450in
    \rlap{\kern 0.230in\lower\graphtemp\hbox to 0pt{\hss $y''$\hss}}%
    \graphtemp=.5ex\advance\graphtemp by 0.450in
    \rlap{\kern 0.779in\lower\graphtemp\hbox to 0pt{\hss $u$\hss}}%
    \graphtemp=.5ex\advance\graphtemp by 0.310in
    \rlap{\kern 0.951in\lower\graphtemp\hbox to 0pt{\hss $b$\hss}}%
    \graphtemp=.5ex\advance\graphtemp by 0.310in
    \rlap{\kern 1.184in\lower\graphtemp\hbox to 0pt{\hss $a$\hss}}%
    \graphtemp=.5ex\advance\graphtemp by 0.310in
    \rlap{\kern 0.019in\lower\graphtemp\hbox to 0pt{\hss $c$\hss}}%
    \graphtemp=.5ex\advance\graphtemp by 0.504in
    \rlap{\kern 2.231in\lower\graphtemp\hbox to 0pt{\hss $\bu$\hss}}%
    \graphtemp=.5ex\advance\graphtemp by 0.116in
    \rlap{\kern 2.231in\lower\graphtemp\hbox to 0pt{\hss $\bu$\hss}}%
    \graphtemp=.5ex\advance\graphtemp by 0.504in
    \rlap{\kern 2.619in\lower\graphtemp\hbox to 0pt{\hss $\bu$\hss}}%
    \graphtemp=.5ex\advance\graphtemp by 0.116in
    \rlap{\kern 2.619in\lower\graphtemp\hbox to 0pt{\hss $\bu$\hss}}%
    \graphtemp=.5ex\advance\graphtemp by 0.310in
    \rlap{\kern 2.813in\lower\graphtemp\hbox to 0pt{\hss $\bu$\hss}}%
    \graphtemp=.5ex\advance\graphtemp by 0.310in
    \rlap{\kern 3.201in\lower\graphtemp\hbox to 0pt{\hss $\bu$\hss}}%
    \graphtemp=.5ex\advance\graphtemp by 0.601in
    \rlap{\kern 3.201in\lower\graphtemp\hbox to 0pt{\hss $\bu$\hss}}%
    \special{pn 28}%
    \special{pa 2813 310}%
    \special{pa 3201 310}%
    \special{fp}%
    \special{pa 2813 310}%
    \special{pa 2619 504}%
    \special{fp}%
    \special{pa 2813 310}%
    \special{pa 2619 116}%
    \special{fp}%
    \special{pn 11}%
    \special{pa 2231 504}%
    \special{pa 2619 504}%
    \special{fp}%
    \special{pa 2231 116}%
    \special{pa 2619 116}%
    \special{fp}%
    \special{pa 3201 310}%
    \special{pa 3201 601}%
    \special{fp}%
    \graphtemp=.5ex\advance\graphtemp by 0.194in
    \rlap{\kern 2.813in\lower\graphtemp\hbox to 0pt{\hss $z$\hss}}%
    \graphtemp=.5ex\advance\graphtemp by 0.194in
    \rlap{\kern 3.201in\lower\graphtemp\hbox to 0pt{\hss $v$\hss}}%
    \graphtemp=.5ex\advance\graphtemp by 0.601in
    \rlap{\kern 3.085in\lower\graphtemp\hbox to 0pt{\hss $x$\hss}}%
    \graphtemp=.5ex\advance\graphtemp by 0.000in
    \rlap{\kern 2.619in\lower\graphtemp\hbox to 0pt{\hss $y_1$\hss}}%
    \graphtemp=.5ex\advance\graphtemp by 0.621in
    \rlap{\kern 2.619in\lower\graphtemp\hbox to 0pt{\hss $y_2$\hss}}%
    \graphtemp=.5ex\advance\graphtemp by 0.000in
    \rlap{\kern 2.290in\lower\graphtemp\hbox to 0pt{\hss $y_1'$\hss}}%
    \graphtemp=.5ex\advance\graphtemp by 0.621in
    \rlap{\kern 2.270in\lower\graphtemp\hbox to 0pt{\hss $y_2'$\hss}}%
    \special{pn 8}%
    \special{ar 2134 504 97 97 0 6.28319}%
    \special{ar 2134 116 97 97 0 6.28319}%
    \special{ar 3201 699 97 97 0 6.28319}%
    \graphtemp=.5ex\advance\graphtemp by 0.310in
    \rlap{\kern 3.201in\lower\graphtemp\hbox to 0pt{\hss $\bu$\hss}}%
    \graphtemp=.5ex\advance\graphtemp by 0.310in
    \rlap{\kern 3.590in\lower\graphtemp\hbox to 0pt{\hss $\bu$\hss}}%
    \graphtemp=.5ex\advance\graphtemp by 0.601in
    \rlap{\kern 3.590in\lower\graphtemp\hbox to 0pt{\hss $\bu$\hss}}%
    \graphtemp=.5ex\advance\graphtemp by 0.504in
    \rlap{\kern 3.784in\lower\graphtemp\hbox to 0pt{\hss $\bu$\hss}}%
    \graphtemp=.5ex\advance\graphtemp by 0.116in
    \rlap{\kern 3.784in\lower\graphtemp\hbox to 0pt{\hss $\bu$\hss}}%
    \special{pn 28}%
    \special{pa 3201 310}%
    \special{pa 3590 310}%
    \special{fp}%
    \special{pa 3590 310}%
    \special{pa 3590 601}%
    \special{fp}%
    \special{pn 11}%
    \special{pa 3784 504}%
    \special{pa 3590 310}%
    \special{fp}%
    \special{pa 3590 310}%
    \special{pa 3784 116}%
    \special{fp}%
    \special{pn 8}%
    \special{ar 3881 504 97 97 0 6.28319}%
    \special{ar 3881 116 97 97 0 6.28319}%
    \graphtemp=.5ex\advance\graphtemp by 0.194in
    \rlap{\kern 3.590in\lower\graphtemp\hbox to 0pt{\hss $z'$\hss}}%
    \graphtemp=.5ex\advance\graphtemp by 0.601in
    \rlap{\kern 3.473in\lower\graphtemp\hbox to 0pt{\hss $u$\hss}}%
    \graphtemp=.5ex\advance\graphtemp by 0.504in
    \rlap{\kern 4.366in\lower\graphtemp\hbox to 0pt{\hss $\bu$\hss}}%
    \graphtemp=.5ex\advance\graphtemp by 0.116in
    \rlap{\kern 4.366in\lower\graphtemp\hbox to 0pt{\hss $\bu$\hss}}%
    \graphtemp=.5ex\advance\graphtemp by 0.504in
    \rlap{\kern 4.754in\lower\graphtemp\hbox to 0pt{\hss $\bu$\hss}}%
    \graphtemp=.5ex\advance\graphtemp by 0.116in
    \rlap{\kern 4.754in\lower\graphtemp\hbox to 0pt{\hss $\bu$\hss}}%
    \graphtemp=.5ex\advance\graphtemp by 0.310in
    \rlap{\kern 4.948in\lower\graphtemp\hbox to 0pt{\hss $\bu$\hss}}%
    \graphtemp=.5ex\advance\graphtemp by 0.310in
    \rlap{\kern 5.336in\lower\graphtemp\hbox to 0pt{\hss $\bu$\hss}}%
    \graphtemp=.5ex\advance\graphtemp by 0.601in
    \rlap{\kern 5.336in\lower\graphtemp\hbox to 0pt{\hss $\bu$\hss}}%
    \special{pn 28}%
    \special{pa 4948 310}%
    \special{pa 5336 310}%
    \special{fp}%
    \special{pa 4948 310}%
    \special{pa 4754 504}%
    \special{fp}%
    \special{pa 4948 310}%
    \special{pa 4754 116}%
    \special{fp}%
    \special{pn 11}%
    \special{pa 4366 504}%
    \special{pa 4754 504}%
    \special{fp}%
    \special{pa 4366 116}%
    \special{pa 4754 116}%
    \special{fp}%
    \special{pa 5336 310}%
    \special{pa 5336 601}%
    \special{fp}%
    \graphtemp=.5ex\advance\graphtemp by 0.194in
    \rlap{\kern 4.948in\lower\graphtemp\hbox to 0pt{\hss $z$\hss}}%
    \graphtemp=.5ex\advance\graphtemp by 0.194in
    \rlap{\kern 5.336in\lower\graphtemp\hbox to 0pt{\hss $v$\hss}}%
    \graphtemp=.5ex\advance\graphtemp by 0.601in
    \rlap{\kern 5.219in\lower\graphtemp\hbox to 0pt{\hss $x$\hss}}%
    \special{pn 8}%
    \special{ar 4269 504 97 97 0 6.28319}%
    \special{ar 4269 116 97 97 0 6.28319}%
    \special{ar 5336 718 116 116 0 6.28319}%
    \graphtemp=.5ex\advance\graphtemp by 0.000in
    \rlap{\kern 4.754in\lower\graphtemp\hbox to 0pt{\hss $y_1$\hss}}%
    \graphtemp=.5ex\advance\graphtemp by 0.621in
    \rlap{\kern 4.754in\lower\graphtemp\hbox to 0pt{\hss $y_2$\hss}}%
    \graphtemp=.5ex\advance\graphtemp by 0.718in
    \rlap{\kern 5.336in\lower\graphtemp\hbox to 0pt{\hss 6?\hss}}%
    \graphtemp=.5ex\advance\graphtemp by 0.000in
    \rlap{\kern 4.424in\lower\graphtemp\hbox to 0pt{\hss $y_1'$\hss}}%
    \graphtemp=.5ex\advance\graphtemp by 0.621in
    \rlap{\kern 4.404in\lower\graphtemp\hbox to 0pt{\hss $y_2'$\hss}}%
    \graphtemp=.5ex\advance\graphtemp by 0.310in
    \rlap{\kern 5.336in\lower\graphtemp\hbox to 0pt{\hss $\bu$\hss}}%
    \graphtemp=.5ex\advance\graphtemp by 0.310in
    \rlap{\kern 5.724in\lower\graphtemp\hbox to 0pt{\hss $\bu$\hss}}%
    \graphtemp=.5ex\advance\graphtemp by 0.504in
    \rlap{\kern 5.918in\lower\graphtemp\hbox to 0pt{\hss $\bu$\hss}}%
    \graphtemp=.5ex\advance\graphtemp by 0.116in
    \rlap{\kern 5.918in\lower\graphtemp\hbox to 0pt{\hss $\bu$\hss}}%
    \graphtemp=.5ex\advance\graphtemp by 0.504in
    \rlap{\kern 6.306in\lower\graphtemp\hbox to 0pt{\hss $\bu$\hss}}%
    \graphtemp=.5ex\advance\graphtemp by 0.116in
    \rlap{\kern 6.306in\lower\graphtemp\hbox to 0pt{\hss $\bu$\hss}}%
    \special{pn 28}%
    \special{pa 5724 310}%
    \special{pa 5918 116}%
    \special{fp}%
    \special{pa 5724 310}%
    \special{pa 5336 310}%
    \special{fp}%
    \special{pa 5724 310}%
    \special{pa 5918 504}%
    \special{fp}%
    \special{pn 11}%
    \special{pa 5918 504}%
    \special{pa 6306 504}%
    \special{fp}%
    \special{pa 5918 116}%
    \special{pa 6306 116}%
    \special{fp}%
    \special{pn 8}%
    \special{ar 6403 504 97 97 0 6.28319}%
    \special{ar 6403 116 97 97 0 6.28319}%
    \graphtemp=.5ex\advance\graphtemp by 0.194in
    \rlap{\kern 5.724in\lower\graphtemp\hbox to 0pt{\hss $z'$\hss}}%
    \graphtemp=.5ex\advance\graphtemp by 0.504in
    \rlap{\kern 6.034in\lower\graphtemp\hbox to 0pt{\hss $~$\hss}}%
    \hbox{\vrule depth0.834in width0pt height 0pt}%
    \kern 6.500in
  }%
}%
}
\vspace{-.5pc}

\caption{Cases {\bf G1} and {\bf G2}  for Lemma~\ref{83red3}\label{figG12}}
\end{figure}

{\bf Subcase 2:} {\it $z$ is a $\gamma_{3b}$-vertex.} 
Let $N_G(z)=\{v,y_1,y_2\}$, with $N_G(y_1)=\{z,y_1'\}$ and
$N_G(y_2)=\{z,y_2'\}$.  By Subcase 1, we may assume that $z'$ is not a
$\gamma_{3a}$-vertex.

Suppose that $z'$ is a $\gamma_4$-vertex, with $1$-neighbor $u$, as in
Figure~\ref{figG12} in the middle.  Set $w(vz)=3$ to ensure satisfying
$zy_1$ and $zy_2$.  Now $w(uz')$ has two choices that satisfy $vz'$, and
$w(vz')$ has two choices that satisfy $vx$.  With at least three choices for
the sum $w(uz')+w(vz')$, the two edges in $\Gamma_{G'}(z')$ can also be
satisfied.  Now choose $w(zy_1)$ to satisfy $\Gamma_{G'}(y_1)$ and
$w(zy_2)$ to satisfy $zv$ and $\Gamma_{G'}(y_2)$.

Hence we may assume that $z'$ is also a $\gamma_{3b}$-vertex, as on the right
in Figure~\ref{figG12}.  If $\rp xv\ne6$, then set $w(vz)=w(vz')=3$.  This
satisfies $vx$ and also ensures satisfying $F_z$ and $F_{z'}$.  Choose
$w(zy_1)$ to satisfy $y_1y_1'$, and choose $w(zy_2)$ to satisfy $y_2y_2'$ and
$zv$.  Choose weights on $F_{z'}$ by the same method. 

If $\rp xv=6$ and $w'(y_1y_1')\ne 3$, then set $w(vz)=2$ and $w(vz')=3$ to
satisfy $vx$ and ensure satisfying $zy_1$.  Choose $w(zy_1)\in\{2,3\}$ to
satisfy $y_1y_1'$ and ensure satisfying $zy_2$.  Now choose $w(zy_2)$ to
satisfy $y_2y_2'$ and $vz$.  Since setting $w(vz')=3$ ensures satisfying
$F_{z'}$, we can choose weights on $F_{z'}$ to finish as in the preceding
paragraph.

Hence we may assume that $\rp xv=6$ and (by symmetry) that all the edges of
$G'$ incident to the $2$-neighbors of $z$ and $z'$ have weight $3$ under $w'$.
Since we may assume by Subcase 1 that neither $z$ nor $z'$ is a
$\gamma_{3a}$-vertex, we can complete the extension by giving all the missing
edges weight $1$, unless $w'(vx)=1$.  In that case, just change $w(vz)$ to $3$
and choose $w(zy_i)\in\{2,3\}$ to satisfy $y_iy_i'$, for $i\in\{1,2\}$.

\begin{figure}[h]
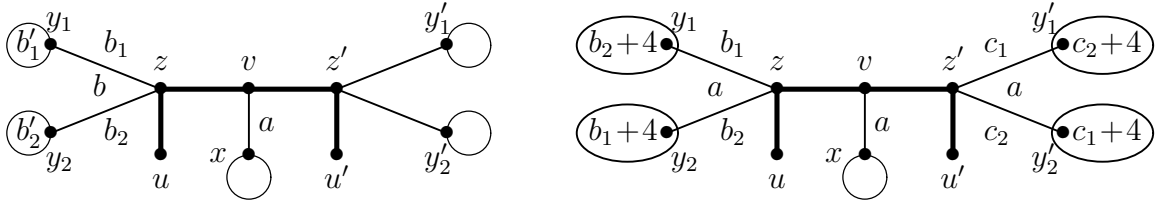

\gpic{
\expandafter\ifx\csname graph\endcsname\relax \csname newbox\endcsname\graph\fi
\expandafter\ifx\csname graphtemp\endcsname\relax \csname newdimen\endcsname\graphtemp\fi
\setbox\graph=\vtop{\vskip 0pt\hbox{%
    \graphtemp=.5ex\advance\graphtemp by 0.610in
    \rlap{\kern 0.230in\lower\graphtemp\hbox to 0pt{\hss $\bu$\hss}}%
    \graphtemp=.5ex\advance\graphtemp by 0.150in
    \rlap{\kern 0.230in\lower\graphtemp\hbox to 0pt{\hss $\bu$\hss}}%
    \graphtemp=.5ex\advance\graphtemp by 0.380in
    \rlap{\kern 0.806in\lower\graphtemp\hbox to 0pt{\hss $\bu$\hss}}%
    \graphtemp=.5ex\advance\graphtemp by 0.726in
    \rlap{\kern 0.806in\lower\graphtemp\hbox to 0pt{\hss $\bu$\hss}}%
    \graphtemp=.5ex\advance\graphtemp by 0.380in
    \rlap{\kern 1.267in\lower\graphtemp\hbox to 0pt{\hss $\bu$\hss}}%
    \graphtemp=.5ex\advance\graphtemp by 0.380in
    \rlap{\kern 1.727in\lower\graphtemp\hbox to 0pt{\hss $\bu$\hss}}%
    \graphtemp=.5ex\advance\graphtemp by 0.726in
    \rlap{\kern 1.727in\lower\graphtemp\hbox to 0pt{\hss $\bu$\hss}}%
    \graphtemp=.5ex\advance\graphtemp by 0.610in
    \rlap{\kern 2.303in\lower\graphtemp\hbox to 0pt{\hss $\bu$\hss}}%
    \graphtemp=.5ex\advance\graphtemp by 0.150in
    \rlap{\kern 2.303in\lower\graphtemp\hbox to 0pt{\hss $\bu$\hss}}%
    \graphtemp=.5ex\advance\graphtemp by 0.726in
    \rlap{\kern 1.267in\lower\graphtemp\hbox to 0pt{\hss $\bu$\hss}}%
    \special{pn 28}%
    \special{pa 806 726}%
    \special{pa 806 380}%
    \special{fp}%
    \special{pa 806 380}%
    \special{pa 1727 380}%
    \special{fp}%
    \special{pa 1727 380}%
    \special{pa 1727 726}%
    \special{fp}%
    \special{pn 11}%
    \special{pa 230 610}%
    \special{pa 806 380}%
    \special{fp}%
    \special{pa 806 380}%
    \special{pa 230 150}%
    \special{fp}%
    \special{pa 2303 150}%
    \special{pa 1727 380}%
    \special{fp}%
    \special{pa 1727 380}%
    \special{pa 2303 610}%
    \special{fp}%
    \special{pa 1267 380}%
    \special{pa 1267 726}%
    \special{fp}%
    \special{pn 8}%
    \special{ar 115 610 115 115 0 6.28319}%
    \special{ar 115 150 115 115 0 6.28319}%
    \special{ar 2418 610 115 115 0 6.28319}%
    \special{ar 2418 150 115 115 0 6.28319}%
    \special{ar 1267 841 115 115 0 6.28319}%
    \graphtemp=.5ex\advance\graphtemp by 0.610in
    \rlap{\kern 0.115in\lower\graphtemp\hbox to 0pt{\hss $b_2'$\hss}}%
    \graphtemp=.5ex\advance\graphtemp by 0.150in
    \rlap{\kern 0.115in\lower\graphtemp\hbox to 0pt{\hss $b_1'$\hss}}%
    \graphtemp=.5ex\advance\graphtemp by 0.749in
    \rlap{\kern 0.276in\lower\graphtemp\hbox to 0pt{\hss $y_2$\hss}}%
    \graphtemp=.5ex\advance\graphtemp by 0.012in
    \rlap{\kern 0.276in\lower\graphtemp\hbox to 0pt{\hss $y_1$\hss}}%
    \graphtemp=.5ex\advance\graphtemp by 0.610in
    \rlap{\kern 0.576in\lower\graphtemp\hbox to 0pt{\hss $b_2$\hss}}%
    \graphtemp=.5ex\advance\graphtemp by 0.150in
    \rlap{\kern 0.576in\lower\graphtemp\hbox to 0pt{\hss $b_1$\hss}}%
    \graphtemp=.5ex\advance\graphtemp by 0.242in
    \rlap{\kern 0.806in\lower\graphtemp\hbox to 0pt{\hss $z$\hss}}%
    \graphtemp=.5ex\advance\graphtemp by 0.864in
    \rlap{\kern 0.806in\lower\graphtemp\hbox to 0pt{\hss $u$\hss}}%
    \graphtemp=.5ex\advance\graphtemp by 0.242in
    \rlap{\kern 1.267in\lower\graphtemp\hbox to 0pt{\hss $v$\hss}}%
    \graphtemp=.5ex\advance\graphtemp by 0.726in
    \rlap{\kern 1.106in\lower\graphtemp\hbox to 0pt{\hss $x$\hss}}%
    \graphtemp=.5ex\advance\graphtemp by 0.242in
    \rlap{\kern 1.727in\lower\graphtemp\hbox to 0pt{\hss $z'$\hss}}%
    \graphtemp=.5ex\advance\graphtemp by 0.864in
    \rlap{\kern 1.727in\lower\graphtemp\hbox to 0pt{\hss $u'$\hss}}%
    \graphtemp=.5ex\advance\graphtemp by 0.749in
    \rlap{\kern 2.257in\lower\graphtemp\hbox to 0pt{\hss $y_2'$\hss}}%
    \graphtemp=.5ex\advance\graphtemp by 0.012in
    \rlap{\kern 2.257in\lower\graphtemp\hbox to 0pt{\hss $y_1'$\hss}}%
    \graphtemp=.5ex\advance\graphtemp by 0.380in
    \rlap{\kern 0.484in\lower\graphtemp\hbox to 0pt{\hss $b$\hss}}%
    \graphtemp=.5ex\advance\graphtemp by 0.564in
    \rlap{\kern 1.359in\lower\graphtemp\hbox to 0pt{\hss $a$\hss}}%
    \graphtemp=.5ex\advance\graphtemp by 0.610in
    \rlap{\kern 3.455in\lower\graphtemp\hbox to 0pt{\hss $\bu$\hss}}%
    \graphtemp=.5ex\advance\graphtemp by 0.150in
    \rlap{\kern 3.455in\lower\graphtemp\hbox to 0pt{\hss $\bu$\hss}}%
    \graphtemp=.5ex\advance\graphtemp by 0.380in
    \rlap{\kern 4.031in\lower\graphtemp\hbox to 0pt{\hss $\bu$\hss}}%
    \graphtemp=.5ex\advance\graphtemp by 0.726in
    \rlap{\kern 4.031in\lower\graphtemp\hbox to 0pt{\hss $\bu$\hss}}%
    \graphtemp=.5ex\advance\graphtemp by 0.380in
    \rlap{\kern 4.491in\lower\graphtemp\hbox to 0pt{\hss $\bu$\hss}}%
    \graphtemp=.5ex\advance\graphtemp by 0.380in
    \rlap{\kern 4.952in\lower\graphtemp\hbox to 0pt{\hss $\bu$\hss}}%
    \graphtemp=.5ex\advance\graphtemp by 0.726in
    \rlap{\kern 4.952in\lower\graphtemp\hbox to 0pt{\hss $\bu$\hss}}%
    \graphtemp=.5ex\advance\graphtemp by 0.610in
    \rlap{\kern 5.528in\lower\graphtemp\hbox to 0pt{\hss $\bu$\hss}}%
    \graphtemp=.5ex\advance\graphtemp by 0.150in
    \rlap{\kern 5.528in\lower\graphtemp\hbox to 0pt{\hss $\bu$\hss}}%
    \graphtemp=.5ex\advance\graphtemp by 0.726in
    \rlap{\kern 4.491in\lower\graphtemp\hbox to 0pt{\hss $\bu$\hss}}%
    \special{pn 28}%
    \special{pa 4031 726}%
    \special{pa 4031 380}%
    \special{fp}%
    \special{pa 4031 380}%
    \special{pa 4952 380}%
    \special{fp}%
    \special{pa 4952 380}%
    \special{pa 4952 726}%
    \special{fp}%
    \special{pn 11}%
    \special{pa 3455 610}%
    \special{pa 4031 380}%
    \special{fp}%
    \special{pa 4031 380}%
    \special{pa 3455 150}%
    \special{fp}%
    \special{pa 5528 150}%
    \special{pa 4952 380}%
    \special{fp}%
    \special{pa 4952 380}%
    \special{pa 5528 610}%
    \special{fp}%
    \special{pa 4491 380}%
    \special{pa 4491 726}%
    \special{fp}%
    \special{ar 3248 610 265 150 0 6.28319}%
    \special{ar 3248 150 265 150 0 6.28319}%
    \special{ar 5735 610 265 150 0 6.28319}%
    \special{ar 5735 150 265 150 0 6.28319}%
    \special{pn 8}%
    \special{ar 4491 841 115 115 0 6.28319}%
    \graphtemp=.5ex\advance\graphtemp by 0.610in
    \rlap{\kern 3.225in\lower\graphtemp\hbox to 0pt{\hss $b_1\!+\!4$\hss}}%
    \graphtemp=.5ex\advance\graphtemp by 0.150in
    \rlap{\kern 3.225in\lower\graphtemp\hbox to 0pt{\hss $b_2\!+\!4$\hss}}%
    \graphtemp=.5ex\advance\graphtemp by 0.749in
    \rlap{\kern 3.547in\lower\graphtemp\hbox to 0pt{\hss $y_2$\hss}}%
    \graphtemp=.5ex\advance\graphtemp by 0.012in
    \rlap{\kern 3.547in\lower\graphtemp\hbox to 0pt{\hss $y_1$\hss}}%
    \graphtemp=.5ex\advance\graphtemp by 0.610in
    \rlap{\kern 3.800in\lower\graphtemp\hbox to 0pt{\hss $b_2$\hss}}%
    \graphtemp=.5ex\advance\graphtemp by 0.150in
    \rlap{\kern 3.800in\lower\graphtemp\hbox to 0pt{\hss $b_1$\hss}}%
    \graphtemp=.5ex\advance\graphtemp by 0.242in
    \rlap{\kern 4.031in\lower\graphtemp\hbox to 0pt{\hss $z$\hss}}%
    \graphtemp=.5ex\advance\graphtemp by 0.864in
    \rlap{\kern 4.031in\lower\graphtemp\hbox to 0pt{\hss $u$\hss}}%
    \graphtemp=.5ex\advance\graphtemp by 0.242in
    \rlap{\kern 4.491in\lower\graphtemp\hbox to 0pt{\hss $v$\hss}}%
    \graphtemp=.5ex\advance\graphtemp by 0.726in
    \rlap{\kern 4.330in\lower\graphtemp\hbox to 0pt{\hss $x$\hss}}%
    \graphtemp=.5ex\advance\graphtemp by 0.242in
    \rlap{\kern 4.952in\lower\graphtemp\hbox to 0pt{\hss $z'$\hss}}%
    \graphtemp=.5ex\advance\graphtemp by 0.864in
    \rlap{\kern 4.952in\lower\graphtemp\hbox to 0pt{\hss $u'$\hss}}%
    \graphtemp=.5ex\advance\graphtemp by 0.749in
    \rlap{\kern 5.436in\lower\graphtemp\hbox to 0pt{\hss $y_2'$\hss}}%
    \graphtemp=.5ex\advance\graphtemp by 0.012in
    \rlap{\kern 5.436in\lower\graphtemp\hbox to 0pt{\hss $y_1'$\hss}}%
    \graphtemp=.5ex\advance\graphtemp by 0.380in
    \rlap{\kern 3.708in\lower\graphtemp\hbox to 0pt{\hss $a$\hss}}%
    \graphtemp=.5ex\advance\graphtemp by 0.564in
    \rlap{\kern 4.583in\lower\graphtemp\hbox to 0pt{\hss $a$\hss}}%
    \graphtemp=.5ex\advance\graphtemp by 0.610in
    \rlap{\kern 5.758in\lower\graphtemp\hbox to 0pt{\hss $c_1\!+\!4$\hss}}%
    \graphtemp=.5ex\advance\graphtemp by 0.150in
    \rlap{\kern 5.758in\lower\graphtemp\hbox to 0pt{\hss $c_2\!+\!4$\hss}}%
    \graphtemp=.5ex\advance\graphtemp by 0.610in
    \rlap{\kern 5.182in\lower\graphtemp\hbox to 0pt{\hss $c_2$\hss}}%
    \graphtemp=.5ex\advance\graphtemp by 0.150in
    \rlap{\kern 5.182in\lower\graphtemp\hbox to 0pt{\hss $c_1$\hss}}%
    \graphtemp=.5ex\advance\graphtemp by 0.380in
    \rlap{\kern 5.274in\lower\graphtemp\hbox to 0pt{\hss $a$\hss}}%
    \hbox{\vrule depth0.956in width0pt height 0pt}%
    \kern 6.000in
  }%
}%
}
\caption{Case {\bf G3}  for Lemma~\ref{83red3}\label{figG3}}
\end{figure}

{\bf Subcase 3:} {\it $z$ and $z'$ are both $\gamma_4$-vertices.} 
Let $N_G(z)=\{v,y_1,y_2,u\}$ and $N_G(z')=\{v,y_1',y_2',u'\}$, with
$d_G(u)=d_G(u')=1$.  Let $a=w'(vx)$.  For $i\in\{1,2\}$, let $b_i=w'(zy_i)$ and
$b_i'=\rp{y_i}z$, as on the left in Figure~\ref{figG3}.

Let $b=b_1+b_2$.  If $b\ne a$, then choose $w(vz')\in\{1,3\}$ to ensure
satisfying  $vz$.  With $w(vz')$ fixed, choose $w(z'u')$ to satisfy
$\Gamma_{G'}(z')$.  Now choose $w(vz)$ to satisfy $vx$ and $vz'$, and then
choose $w(zu)$ to satisfy $\Gamma_{G'}(z)$ and complete the extension.

If $b_1'\ne b_2+4$, then choose $w(vz)\in\{1,3\}$ to ensure satisfying 
$zy_1$.  Now restrict $w(u'z')$ to two choices that satisfy $vz'$, and restrict
$w(vz')$ to two choices that satisfy $vx$.  Having at least three choices
for the sum $w(u'z')+w(vz')$ allows also satisfying the edges of
$\Gamma_{G'}(z')$.  Finally, choose $w(uz)$ to satisfy $zy_2$ and $zv$ and
complete the extension.

By symmetry, the only remaining case is $b=a$, $b_1'=b_2+4$, and $b_2'=b_1+4$,
and similarly for $\Gamma_{G'}(z')$, as on the right in Figure~\ref{figG3}.
Now $w(vz)+w(zu)<4$ satisfies $zy_1$ and $zy_2$, and $w(uz)\ne w(vz')$
satisfies $vz$.  Similarly, $w(vz')+w(z'u')>4$
satisfies $z'y'_1$ and $z'y'_2$, and $w(u'z')\ne w(vz)$ satisfies $vz'$.  Set
$w(vz)=w(zu)=1$ and $w(z'u')=3$, and choose $w(vz')\in\{2,3\}$ to satisfy $vx$.

{\bf Case H:} {\it A vertex $v$ such that $p_1+2q\ge d(v)$ and $p_1+q>4$.}
Let $Z$ be the set of $\gamma$-neighbors of $v$.  Let $R$ be the set of edges
from $v$ to $1$-neighbors and to $Z$, shown bold in Figure~\ref{figHI3}.  By
{\bf F}, the set $Z$ is independent.  Form $G'$ from $G$ by deleting $R$ and
$F_z$ for each $z\in Z$.  Here $v$ and a $\gamma$-neighbor of $v$ play the
roles of $x$ and $v$ in Figure~\ref{figgamma}, respectively.
 
Let $R'$ be a set of $d(v)-p_1-q$ edges from $v$ to $Z$.  Assign weight $3$ to
all of $R-R'$.  Choose weights on $R'$ from $\{2,3\}$ to satisfy the 
$d(v)-p_1-q$ edges in $\Gamma_{G'}(v)$.

Consider $vz$ with $z\in Z$.  Including the weights on $\Gamma_{G'}(v)$, the
sum of the weights on $\Gamma_G(v)-\{vz\}$ is now at least $3(q-1)+3p_1$, which
by hypothesis exceeds $9$.  At most three edges are incident to $vz$ at $z$, so
$vz$ is automatically satisfied, as are the edges from $v$ to $1$-neighbors.
Now the weights on $R$ are fixed; apply Lemma~\ref{gamma} to the vertices of
$Z$.

\begin{figure}[h]
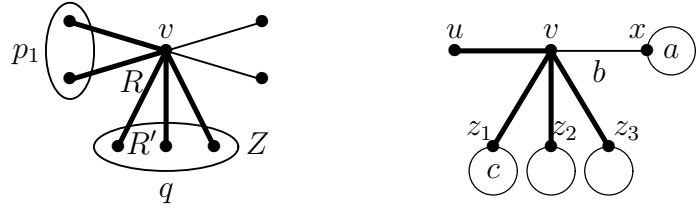

\gpic{
\expandafter\ifx\csname graph\endcsname\relax \csname newbox\endcsname\graph\fi
\expandafter\ifx\csname graphtemp\endcsname\relax \csname newdimen\endcsname\graphtemp\fi
\setbox\graph=\vtop{\vskip 0pt\hbox{%
    \graphtemp=.5ex\advance\graphtemp by 0.252in
    \rlap{\kern 0.730in\lower\graphtemp\hbox to 0pt{\hss $\bu$\hss}}%
    \graphtemp=.5ex\advance\graphtemp by 0.101in
    \rlap{\kern 0.227in\lower\graphtemp\hbox to 0pt{\hss $\bu$\hss}}%
    \graphtemp=.5ex\advance\graphtemp by 0.403in
    \rlap{\kern 0.227in\lower\graphtemp\hbox to 0pt{\hss $\bu$\hss}}%
    \graphtemp=.5ex\advance\graphtemp by 0.101in
    \rlap{\kern 1.234in\lower\graphtemp\hbox to 0pt{\hss $\bu$\hss}}%
    \graphtemp=.5ex\advance\graphtemp by 0.403in
    \rlap{\kern 1.234in\lower\graphtemp\hbox to 0pt{\hss $\bu$\hss}}%
    \graphtemp=.5ex\advance\graphtemp by 0.755in
    \rlap{\kern 0.478in\lower\graphtemp\hbox to 0pt{\hss $\bu$\hss}}%
    \graphtemp=.5ex\advance\graphtemp by 0.755in
    \rlap{\kern 0.982in\lower\graphtemp\hbox to 0pt{\hss $\bu$\hss}}%
    \graphtemp=.5ex\advance\graphtemp by 0.755in
    \rlap{\kern 0.730in\lower\graphtemp\hbox to 0pt{\hss $\bu$\hss}}%
    \special{pn 28}%
    \special{pa 730 252}%
    \special{pa 730 755}%
    \special{fp}%
    \special{pa 730 252}%
    \special{pa 982 755}%
    \special{fp}%
    \special{pa 730 252}%
    \special{pa 478 755}%
    \special{fp}%
    \special{pa 730 252}%
    \special{pa 227 101}%
    \special{fp}%
    \special{pa 730 252}%
    \special{pa 227 403}%
    \special{fp}%
    \special{pn 11}%
    \special{pa 1234 101}%
    \special{pa 730 252}%
    \special{fp}%
    \special{pa 730 252}%
    \special{pa 1234 403}%
    \special{fp}%
    \special{ar 227 252 126 252 0 6.28319}%
    \special{ar 730 755 378 126 0 6.28319}%
    \graphtemp=.5ex\advance\graphtemp by 0.151in
    \rlap{\kern 0.730in\lower\graphtemp\hbox to 0pt{\hss $v$\hss}}%
    \graphtemp=.5ex\advance\graphtemp by 0.252in
    \rlap{\kern 0.000in\lower\graphtemp\hbox to 0pt{\hss $p_1$\hss}}%
    \graphtemp=.5ex\advance\graphtemp by 0.982in
    \rlap{\kern 0.730in\lower\graphtemp\hbox to 0pt{\hss $q$\hss}}%
    \graphtemp=.5ex\advance\graphtemp by 0.755in
    \rlap{\kern 0.604in\lower\graphtemp\hbox to 0pt{\hss $R'$\hss}}%
    \graphtemp=.5ex\advance\graphtemp by 0.432in
    \rlap{\kern 0.550in\lower\graphtemp\hbox to 0pt{\hss $R$\hss}}%
    \graphtemp=.5ex\advance\graphtemp by 0.755in
    \rlap{\kern 1.209in\lower\graphtemp\hbox to 0pt{\hss $Z$\hss}}%
    \graphtemp=.5ex\advance\graphtemp by 0.252in
    \rlap{\kern 2.241in\lower\graphtemp\hbox to 0pt{\hss $\bu$\hss}}%
    \graphtemp=.5ex\advance\graphtemp by 0.252in
    \rlap{\kern 2.745in\lower\graphtemp\hbox to 0pt{\hss $\bu$\hss}}%
    \graphtemp=.5ex\advance\graphtemp by 0.252in
    \rlap{\kern 3.248in\lower\graphtemp\hbox to 0pt{\hss $\bu$\hss}}%
    \graphtemp=.5ex\advance\graphtemp by 0.755in
    \rlap{\kern 2.442in\lower\graphtemp\hbox to 0pt{\hss $\bu$\hss}}%
    \graphtemp=.5ex\advance\graphtemp by 0.755in
    \rlap{\kern 2.745in\lower\graphtemp\hbox to 0pt{\hss $\bu$\hss}}%
    \graphtemp=.5ex\advance\graphtemp by 0.755in
    \rlap{\kern 3.047in\lower\graphtemp\hbox to 0pt{\hss $\bu$\hss}}%
    \special{pn 28}%
    \special{pa 2745 252}%
    \special{pa 3047 755}%
    \special{fp}%
    \special{pa 2745 252}%
    \special{pa 2745 755}%
    \special{fp}%
    \special{pa 2745 252}%
    \special{pa 2241 252}%
    \special{fp}%
    \special{pa 2745 252}%
    \special{pa 2442 755}%
    \special{fp}%
    \special{pn 11}%
    \special{pa 2745 252}%
    \special{pa 3248 252}%
    \special{fp}%
    \special{pn 8}%
    \special{ar 2442 881 126 126 0 6.28319}%
    \special{ar 2745 881 126 126 0 6.28319}%
    \special{ar 3047 881 126 126 0 6.28319}%
    \special{ar 3374 252 126 126 0 6.28319}%
    \graphtemp=.5ex\advance\graphtemp by 0.151in
    \rlap{\kern 2.241in\lower\graphtemp\hbox to 0pt{\hss $u$\hss}}%
    \graphtemp=.5ex\advance\graphtemp by 0.151in
    \rlap{\kern 2.745in\lower\graphtemp\hbox to 0pt{\hss $v$\hss}}%
    \graphtemp=.5ex\advance\graphtemp by 0.151in
    \rlap{\kern 3.198in\lower\graphtemp\hbox to 0pt{\hss $x$\hss}}%
    \graphtemp=.5ex\advance\graphtemp by 0.252in
    \rlap{\kern 3.374in\lower\graphtemp\hbox to 0pt{\hss $a$\hss}}%
    \graphtemp=.5ex\advance\graphtemp by 0.378in
    \rlap{\kern 2.996in\lower\graphtemp\hbox to 0pt{\hss $b$\hss}}%
    \graphtemp=.5ex\advance\graphtemp by 0.659in
    \rlap{\kern 2.371in\lower\graphtemp\hbox to 0pt{\hss $z_1$\hss}}%
    \graphtemp=.5ex\advance\graphtemp by 0.659in
    \rlap{\kern 2.816in\lower\graphtemp\hbox to 0pt{\hss $z_2$\hss}}%
    \graphtemp=.5ex\advance\graphtemp by 0.659in
    \rlap{\kern 3.143in\lower\graphtemp\hbox to 0pt{\hss $z_3$\hss}}%
    \graphtemp=.5ex\advance\graphtemp by 0.881in
    \rlap{\kern 2.442in\lower\graphtemp\hbox to 0pt{\hss $c$\hss}}%
    \hbox{\vrule depth1.007in width0pt height 0pt}%
    \kern 3.500in
  }%
}%
}
\caption{Cases {\bf H} and {\bf I} for Lemma~\ref{83red3}\label{figHI3}}
\end{figure}

{\bf Case I:} {\it A $5$-vertex $v$ having a $1$-neighbor and three
$\gamma$-neighbors.}
The argument of Case {\bf H} does not suffice here, since $p_1+q=4$.  Let $u$
be the $1$-neighbor of $v$, and let $z_1,z_2,z_3$ be its $\gamma$-neighbors.
As in {\bf H}, let $R=\{vu,vz_1,vz_2,vz_3\}$ (bold on the right in
Figure~\ref{figHI3}), and let $G'=G-R-\bigcup_i F_{z_i}$.  Let $a=\rp{x}{v}$,
and let $b=w'(vx)$.  By Lemma~\ref{4cycle}, we may assume that no two of the
$\gamma$-neighbors are $3$-vertices with a common $2$-neighbor, and by {\bf F}
they form an independent set.  Hence Figure~\ref{figHI3} is accurate.

If $a\ne12$, then put weight $3$ on all edges of $R$ to satisfy $vx$ and 
ensure satisfying $\{vz_1,vz_2,vz_3\}$.  Finally, apply
Lemma~\ref{gamma} to each $z_i$.

If $a=12$, then set $w(vz_1)=2$ so that $vx$ is automatically satisfied.
Having specified $w(vz_1)$, apply Lemma~\ref{gamma} to $z_1$.  Now let
$c=\rho_w(z_1,v)$.  If $c\le 7$, or if $c=8$ and $b\ge2$, then set
$w(vz_2)=w(vz_3)=3$ to ensure satisfying $vz_1$.  If $c=9$, or if $c=8$ and
$b=1$, then set $w(vz_2)=w(vz_3)=1$ to ensure satisfying $vz_1$.  Next
apply Lemma~\ref{gamma} to $z_2$ and $z_3$.  Finally, choose $w(vu)$ to
satisfy $vz_2$ and $vz_3$.

\smallskip
{\bf Case J:} {\it A $4$-vertex with specified neighbors.}
As usual, by the prior lemmas and reducible configurations, Figure~\ref{figJ3}
is accurate.

{\bf Subcase 1:} {\it A $4$-vertex $v$ with $\alpha$-neighbors $z$ and $z'$.}
Let $N_G(z)=\{v,y\}$ and $N_G(z')=\{v,y'\}$, as in Figure~\ref{figJ3}.
By Lemma~\ref{triangle}, we may assume $\{y,y'\}\cap\{z,z'\}=\nul$.  By {\bf B},
we have $y\ne y'$ and $yy'\notin E(G)$.  Let $G'=G-\{vz,vz',zy,z'y'\}$.  At
least two choices for $w(vz)$ satisfy $zy$, and similarly two choices for
$w(vz')$ satisfy $z'y'$.  This yields at least three choices for
$w(vz)+w(vz')$, which is enough to satisfy $\Gamma_{G'}(v)$.  Finally, choose
$w(zy)$ to satisfy its two incident edges and $w(z'y')$ to satisfy its two
incident edges.

\begin{figure}[h]
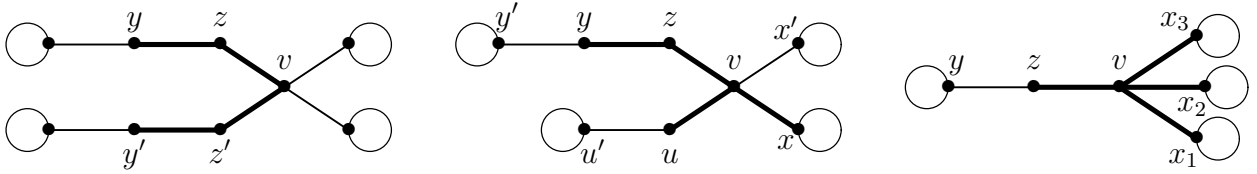

\gpic{
\expandafter\ifx\csname graph\endcsname\relax \csname newbox\endcsname\graph\fi
\expandafter\ifx\csname graphtemp\endcsname\relax \csname newdimen\endcsname\graphtemp\fi
\setbox\graph=\vtop{\vskip 0pt\hbox{%
    \graphtemp=.5ex\advance\graphtemp by 0.605in
    \rlap{\kern 0.224in\lower\graphtemp\hbox to 0pt{\hss $\bu$\hss}}%
    \graphtemp=.5ex\advance\graphtemp by 0.157in
    \rlap{\kern 0.224in\lower\graphtemp\hbox to 0pt{\hss $\bu$\hss}}%
    \graphtemp=.5ex\advance\graphtemp by 0.605in
    \rlap{\kern 0.672in\lower\graphtemp\hbox to 0pt{\hss $\bu$\hss}}%
    \graphtemp=.5ex\advance\graphtemp by 0.157in
    \rlap{\kern 0.672in\lower\graphtemp\hbox to 0pt{\hss $\bu$\hss}}%
    \graphtemp=.5ex\advance\graphtemp by 0.605in
    \rlap{\kern 1.121in\lower\graphtemp\hbox to 0pt{\hss $\bu$\hss}}%
    \graphtemp=.5ex\advance\graphtemp by 0.157in
    \rlap{\kern 1.121in\lower\graphtemp\hbox to 0pt{\hss $\bu$\hss}}%
    \graphtemp=.5ex\advance\graphtemp by 0.381in
    \rlap{\kern 1.457in\lower\graphtemp\hbox to 0pt{\hss $\bu$\hss}}%
    \graphtemp=.5ex\advance\graphtemp by 0.605in
    \rlap{\kern 1.793in\lower\graphtemp\hbox to 0pt{\hss $\bu$\hss}}%
    \graphtemp=.5ex\advance\graphtemp by 0.157in
    \rlap{\kern 1.793in\lower\graphtemp\hbox to 0pt{\hss $\bu$\hss}}%
    \special{pn 28}%
    \special{pa 672 605}%
    \special{pa 1121 605}%
    \special{fp}%
    \special{pa 1121 605}%
    \special{pa 1457 381}%
    \special{fp}%
    \special{pa 1457 381}%
    \special{pa 1121 157}%
    \special{fp}%
    \special{pa 1121 157}%
    \special{pa 672 157}%
    \special{fp}%
    \special{pn 11}%
    \special{pa 1793 605}%
    \special{pa 1457 381}%
    \special{fp}%
    \special{pa 1457 381}%
    \special{pa 1793 157}%
    \special{fp}%
    \special{pa 224 605}%
    \special{pa 672 605}%
    \special{fp}%
    \special{pa 224 157}%
    \special{pa 672 157}%
    \special{fp}%
    \graphtemp=.5ex\advance\graphtemp by 0.247in
    \rlap{\kern 1.457in\lower\graphtemp\hbox to 0pt{\hss $v$\hss}}%
    \special{pn 8}%
    \special{ar 112 605 112 112 0 6.28319}%
    \special{ar 112 157 112 112 0 6.28319}%
    \special{ar 1905 605 112 112 0 6.28319}%
    \special{ar 1905 157 112 112 0 6.28319}%
    \graphtemp=.5ex\advance\graphtemp by 0.740in
    \rlap{\kern 0.672in\lower\graphtemp\hbox to 0pt{\hss $y'$\hss}}%
    \graphtemp=.5ex\advance\graphtemp by 0.022in
    \rlap{\kern 0.672in\lower\graphtemp\hbox to 0pt{\hss $y$\hss}}%
    \graphtemp=.5ex\advance\graphtemp by 0.740in
    \rlap{\kern 1.121in\lower\graphtemp\hbox to 0pt{\hss $z'$\hss}}%
    \graphtemp=.5ex\advance\graphtemp by 0.022in
    \rlap{\kern 1.121in\lower\graphtemp\hbox to 0pt{\hss $z$\hss}}%
    \graphtemp=.5ex\advance\graphtemp by 0.157in
    \rlap{\kern 2.578in\lower\graphtemp\hbox to 0pt{\hss $\bu$\hss}}%
    \graphtemp=.5ex\advance\graphtemp by 0.605in
    \rlap{\kern 3.026in\lower\graphtemp\hbox to 0pt{\hss $\bu$\hss}}%
    \graphtemp=.5ex\advance\graphtemp by 0.157in
    \rlap{\kern 3.026in\lower\graphtemp\hbox to 0pt{\hss $\bu$\hss}}%
    \graphtemp=.5ex\advance\graphtemp by 0.605in
    \rlap{\kern 3.474in\lower\graphtemp\hbox to 0pt{\hss $\bu$\hss}}%
    \graphtemp=.5ex\advance\graphtemp by 0.157in
    \rlap{\kern 3.474in\lower\graphtemp\hbox to 0pt{\hss $\bu$\hss}}%
    \graphtemp=.5ex\advance\graphtemp by 0.381in
    \rlap{\kern 3.810in\lower\graphtemp\hbox to 0pt{\hss $\bu$\hss}}%
    \graphtemp=.5ex\advance\graphtemp by 0.605in
    \rlap{\kern 4.147in\lower\graphtemp\hbox to 0pt{\hss $\bu$\hss}}%
    \graphtemp=.5ex\advance\graphtemp by 0.157in
    \rlap{\kern 4.147in\lower\graphtemp\hbox to 0pt{\hss $\bu$\hss}}%
    \special{pn 28}%
    \special{pa 3474 605}%
    \special{pa 3810 381}%
    \special{fp}%
    \special{pa 3810 381}%
    \special{pa 3474 157}%
    \special{fp}%
    \special{pa 3474 157}%
    \special{pa 3026 157}%
    \special{fp}%
    \special{pa 3810 381}%
    \special{pa 4147 605}%
    \special{fp}%
    \special{pn 11}%
    \special{pa 2578 157}%
    \special{pa 3026 157}%
    \special{fp}%
    \special{pa 3026 605}%
    \special{pa 3474 605}%
    \special{fp}%
    \special{pa 3810 381}%
    \special{pa 4147 157}%
    \special{fp}%
    \special{pn 8}%
    \special{ar 2466 157 112 112 0 6.28319}%
    \special{ar 2914 605 112 112 0 6.28319}%
    \special{ar 4259 605 112 112 0 6.28319}%
    \special{ar 4259 157 112 112 0 6.28319}%
    \graphtemp=.5ex\advance\graphtemp by 0.022in
    \rlap{\kern 2.622in\lower\graphtemp\hbox to 0pt{\hss $y'$\hss}}%
    \graphtemp=.5ex\advance\graphtemp by 0.022in
    \rlap{\kern 3.026in\lower\graphtemp\hbox to 0pt{\hss $y$\hss}}%
    \graphtemp=.5ex\advance\graphtemp by 0.740in
    \rlap{\kern 3.071in\lower\graphtemp\hbox to 0pt{\hss $u'$\hss}}%
    \graphtemp=.5ex\advance\graphtemp by 0.740in
    \rlap{\kern 3.474in\lower\graphtemp\hbox to 0pt{\hss $u$\hss}}%
    \graphtemp=.5ex\advance\graphtemp by 0.022in
    \rlap{\kern 3.474in\lower\graphtemp\hbox to 0pt{\hss $z$\hss}}%
    \graphtemp=.5ex\advance\graphtemp by 0.247in
    \rlap{\kern 3.810in\lower\graphtemp\hbox to 0pt{\hss $v$\hss}}%
    \graphtemp=.5ex\advance\graphtemp by 0.669in
    \rlap{\kern 4.083in\lower\graphtemp\hbox to 0pt{\hss $x$\hss}}%
    \graphtemp=.5ex\advance\graphtemp by 0.094in
    \rlap{\kern 4.083in\lower\graphtemp\hbox to 0pt{\hss $x'$\hss}}%
    \graphtemp=.5ex\advance\graphtemp by 0.381in
    \rlap{\kern 4.931in\lower\graphtemp\hbox to 0pt{\hss $\bu$\hss}}%
    \graphtemp=.5ex\advance\graphtemp by 0.381in
    \rlap{\kern 5.379in\lower\graphtemp\hbox to 0pt{\hss $\bu$\hss}}%
    \graphtemp=.5ex\advance\graphtemp by 0.381in
    \rlap{\kern 5.828in\lower\graphtemp\hbox to 0pt{\hss $\bu$\hss}}%
    \graphtemp=.5ex\advance\graphtemp by 0.650in
    \rlap{\kern 6.231in\lower\graphtemp\hbox to 0pt{\hss $\bu$\hss}}%
    \graphtemp=.5ex\advance\graphtemp by 0.381in
    \rlap{\kern 6.276in\lower\graphtemp\hbox to 0pt{\hss $\bu$\hss}}%
    \graphtemp=.5ex\advance\graphtemp by 0.112in
    \rlap{\kern 6.231in\lower\graphtemp\hbox to 0pt{\hss $\bu$\hss}}%
    \special{pn 28}%
    \special{pa 5828 381}%
    \special{pa 6231 112}%
    \special{fp}%
    \special{pa 5828 381}%
    \special{pa 6276 381}%
    \special{fp}%
    \special{pa 5828 381}%
    \special{pa 5379 381}%
    \special{fp}%
    \special{pa 5828 381}%
    \special{pa 6231 650}%
    \special{fp}%
    \special{pn 11}%
    \special{pa 4931 381}%
    \special{pa 5379 381}%
    \special{fp}%
    \graphtemp=.5ex\advance\graphtemp by 0.247in
    \rlap{\kern 4.976in\lower\graphtemp\hbox to 0pt{\hss $y$\hss}}%
    \graphtemp=.5ex\advance\graphtemp by 0.247in
    \rlap{\kern 5.379in\lower\graphtemp\hbox to 0pt{\hss $z$\hss}}%
    \graphtemp=.5ex\advance\graphtemp by 0.247in
    \rlap{\kern 5.828in\lower\graphtemp\hbox to 0pt{\hss $v$\hss}}%
    \special{pn 8}%
    \special{ar 4819 381 112 112 0 6.28319}%
    \special{ar 6343 650 112 112 0 6.28319}%
    \special{ar 6388 381 112 112 0 6.28319}%
    \special{ar 6343 112 112 112 0 6.28319}%
    \graphtemp=.5ex\advance\graphtemp by 0.740in
    \rlap{\kern 6.164in\lower\graphtemp\hbox to 0pt{\hss $x_1$\hss}}%
    \graphtemp=.5ex\advance\graphtemp by 0.471in
    \rlap{\kern 6.209in\lower\graphtemp\hbox to 0pt{\hss $x_2$\hss}}%
    \graphtemp=.5ex\advance\graphtemp by 0.022in
    \rlap{\kern 6.119in\lower\graphtemp\hbox to 0pt{\hss $x_3$\hss}}%
    \hbox{\vrule depth0.762in width0pt height 0pt}%
    \kern 6.500in
  }%
}%
}
\vspace{-.5pc}

\caption{Cases {\bf J1,J2,J3}  for Lemma~\ref{83red3}\label{figJ3}}
\end{figure}

{\bf Subcase 2:} {\it A $4$-vertex $v$ with an $\alpha$-neighbor $z$, another
$2$-neighbor $u$ (that is not an $\alpha$-vertex), and a $\gamma$-neighbor
$x$.}  Name the vertices (uniquely) so that $y',y,z,v,u,u'$ form a path in
order, and let $x'$ be the remaining neighbor of $v$, as in Figure~\ref{figJ3}.
Let $G'=G-\{yz,zv,vu,vx\}-F_x$.
Set $w(vx)=2$ to ensure satisfying $vu$.  Apply Lemma~\ref{gamma} to $x$.
At least two choices for $w(vu)$ satisfy $uu'$, and at least two choices for
$w(vz)$ satisfy $yz$.  With at least three choices for $w(vu)+w(vz)$, at least
one satisfies $vx$ and $vx'$.  Finally, choose $w(yz)$ to satisfy $vz$ and
$yy'$.

{\bf Subcase 3:} {\it A $4$-vertex $v$ with a $2$-neighbor $z$ and three
$\gamma$-neighbors $x_1,x_2,x_3$.}
Let $N(z)=\{v,y\}$, as in Figure~\ref{figJ3}.  Let
$G'=G-\Gamma_G(v)-\bigcup_iF_{x_i}$.

If $d(x_1)=3$, or if $d(x_1)=4$ and $\phi_{w'}(x_1)\le 4$, then the value
of $\rho_w(x_1,v)$ will be at most $7$, since when $d(x_1)=4$ there is only one
edge in $F_{x_1}$ (the edge incident to the $1$-neighbor of $x_1$).  Hence
setting $w(vx_2)=w(vx_3)=3$ and restricting $w(zv)$ to $\{2,3\}$ ensures
satisfying $vx_1$ and $vz$.  Apply Lemma~\ref{gamma} to $x_2$ and $x_3$.  Now
choose $w(zv)$ (in $\{2,3\}$) to satisfy $yz$, and then choose $w(vx_1)$ to
satisfy $vx_2$ and $vx_3$.  Finally, apply Lemma~\ref{gamma} to $x_1$.

By symmetry, we may now assume $d(x_i)=4$ and $\phi_{w'}(x_i)\ge5$ for all $i$.
Choose $w(vz)\in\{1,2\}$ to satisfy $zy$.  Set $w(vx_1)=2$ and
$w(vx_2)=w(vx_3)=1$ to ensure satisfying all edges incident to $v$.
Finally, apply Lemma~\ref{gamma} to each $x_i$.

\smallskip
{\bf Case K:} {\it $v$ is a $\gamma$-vertex whose $3^+$-neighbors are all
$\beta$-vertices.}  See Figure~\ref{figK3}.  Let $S$ be the set of neighbors of
$v$ whose degrees are not specified by the definition of $v$ being a
$\gamma$-vertex.  By $\{${\bf B,C,D}$\}$, all vertices of $S$ are
$3^+$-vertices.  By {\bf F}, they cannot be $\gamma$-vertices, so by the
hypothesis of this case, each vertex of $S$ has degree $3$, with one
$2$-neighbor and one $3^+$-neighbor other than $v$ (by {\bf A}).
By Lemmas~\ref{triangle} and~\ref{4cycle} and Cases {\bf B} and {\bf F},
Figure~\ref{figK3} is accurate in each subcase.

{\bf Subcase 1:} {\it $v$ is a $\gamma_{3b}$-vertex.}
Let $N_G(v)=\{z,z',u\}$, where $d_G(u)=3$.  Let $y$ be the $2$-neighbor of $u$.
Let $G'=G-\Gamma_G(v)-uy$.  Set $w(vu)=3$ to ensure satisfying all of
$\{vz,vz',uy\}$.  Choose $w(uy)$ to satisfy its two incident edges other than
$vu$.  Choose $w(vz)$ to satisfy the other edge at $z$, and choose $w(vz')$ to
satisfy $vu$ and the other edge at $z'$.

{\bf Subcase 2:} {\it $v$ is a $\gamma_4$-vertex.}
Let $N_G(v)=\{z_1,z_2,z_3,u\}$, with $d_G(u)=1$.  Let $y_i$ be the $2$-neighbor
of $z_i$.  Let $G'=G-\Gamma_G(v)-\{z_1y_1,z_2y_2,z_3y_3\}$.  For each $i$, set
$w(vz_i)=3$ to ensure satisfying $z_iy_i$, and choose $w(z_iy_i)$ to
satisfy its two incident edges other than $z_iy_i$.  Also $vz_1,vz_2,vz_3$
are satisfied.

\begin{figure}[h]
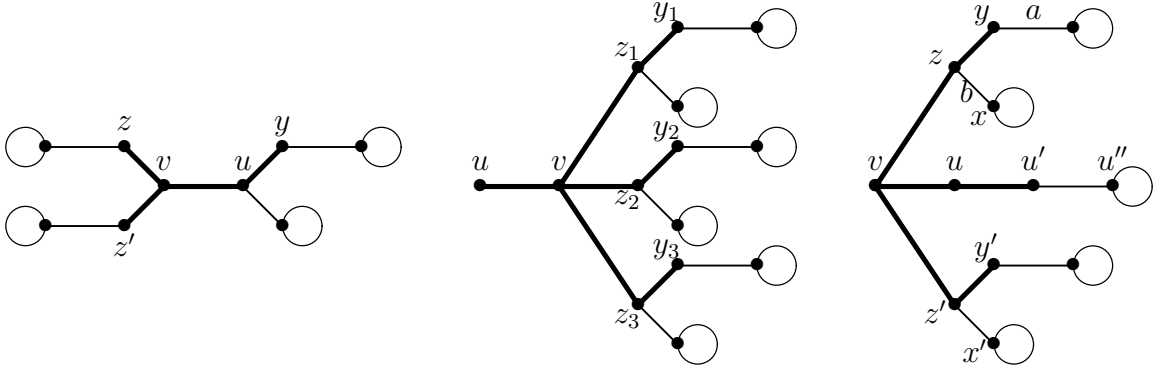

\gpic{
\expandafter\ifx\csname graph\endcsname\relax \csname newbox\endcsname\graph\fi
\expandafter\ifx\csname graphtemp\endcsname\relax \csname newdimen\endcsname\graphtemp\fi
\setbox\graph=\vtop{\vskip 0pt\hbox{%
    \graphtemp=.5ex\advance\graphtemp by 1.159in
    \rlap{\kern 0.207in\lower\graphtemp\hbox to 0pt{\hss $\bu$\hss}}%
    \graphtemp=.5ex\advance\graphtemp by 0.745in
    \rlap{\kern 0.207in\lower\graphtemp\hbox to 0pt{\hss $\bu$\hss}}%
    \graphtemp=.5ex\advance\graphtemp by 1.159in
    \rlap{\kern 0.621in\lower\graphtemp\hbox to 0pt{\hss $\bu$\hss}}%
    \graphtemp=.5ex\advance\graphtemp by 0.745in
    \rlap{\kern 0.621in\lower\graphtemp\hbox to 0pt{\hss $\bu$\hss}}%
    \graphtemp=.5ex\advance\graphtemp by 0.952in
    \rlap{\kern 0.828in\lower\graphtemp\hbox to 0pt{\hss $\bu$\hss}}%
    \graphtemp=.5ex\advance\graphtemp by 0.952in
    \rlap{\kern 1.241in\lower\graphtemp\hbox to 0pt{\hss $\bu$\hss}}%
    \graphtemp=.5ex\advance\graphtemp by 1.159in
    \rlap{\kern 1.448in\lower\graphtemp\hbox to 0pt{\hss $\bu$\hss}}%
    \graphtemp=.5ex\advance\graphtemp by 0.745in
    \rlap{\kern 1.448in\lower\graphtemp\hbox to 0pt{\hss $\bu$\hss}}%
    \graphtemp=.5ex\advance\graphtemp by 0.745in
    \rlap{\kern 1.862in\lower\graphtemp\hbox to 0pt{\hss $\bu$\hss}}%
    \special{pn 28}%
    \special{pa 621 745}%
    \special{pa 828 952}%
    \special{fp}%
    \special{pa 828 952}%
    \special{pa 1241 952}%
    \special{fp}%
    \special{pa 1241 952}%
    \special{pa 1448 745}%
    \special{fp}%
    \special{pa 621 1159}%
    \special{pa 828 952}%
    \special{fp}%
    \special{pn 11}%
    \special{pa 207 1159}%
    \special{pa 621 1159}%
    \special{fp}%
    \special{pa 207 745}%
    \special{pa 621 745}%
    \special{fp}%
    \special{pa 1241 952}%
    \special{pa 1448 1159}%
    \special{fp}%
    \special{pa 1448 745}%
    \special{pa 1862 745}%
    \special{fp}%
    \special{pn 8}%
    \special{ar 103 1159 103 103 0 6.28319}%
    \special{ar 103 745 103 103 0 6.28319}%
    \special{ar 1552 1159 103 103 0 6.28319}%
    \special{ar 1966 745 103 103 0 6.28319}%
    \graphtemp=.5ex\advance\graphtemp by 1.283in
    \rlap{\kern 0.621in\lower\graphtemp\hbox to 0pt{\hss $z'$\hss}}%
    \graphtemp=.5ex\advance\graphtemp by 0.621in
    \rlap{\kern 0.621in\lower\graphtemp\hbox to 0pt{\hss $z$\hss}}%
    \graphtemp=.5ex\advance\graphtemp by 0.828in
    \rlap{\kern 0.828in\lower\graphtemp\hbox to 0pt{\hss $v$\hss}}%
    \graphtemp=.5ex\advance\graphtemp by 0.828in
    \rlap{\kern 1.241in\lower\graphtemp\hbox to 0pt{\hss $u$\hss}}%
    \graphtemp=.5ex\advance\graphtemp by 0.621in
    \rlap{\kern 1.448in\lower\graphtemp\hbox to 0pt{\hss $y$\hss}}%
    \graphtemp=.5ex\advance\graphtemp by 0.952in
    \rlap{\kern 2.483in\lower\graphtemp\hbox to 0pt{\hss $\bu$\hss}}%
    \graphtemp=.5ex\advance\graphtemp by 0.952in
    \rlap{\kern 2.897in\lower\graphtemp\hbox to 0pt{\hss $\bu$\hss}}%
    \graphtemp=.5ex\advance\graphtemp by 1.572in
    \rlap{\kern 3.310in\lower\graphtemp\hbox to 0pt{\hss $\bu$\hss}}%
    \graphtemp=.5ex\advance\graphtemp by 0.952in
    \rlap{\kern 3.310in\lower\graphtemp\hbox to 0pt{\hss $\bu$\hss}}%
    \graphtemp=.5ex\advance\graphtemp by 0.331in
    \rlap{\kern 3.310in\lower\graphtemp\hbox to 0pt{\hss $\bu$\hss}}%
    \graphtemp=.5ex\advance\graphtemp by 1.779in
    \rlap{\kern 3.517in\lower\graphtemp\hbox to 0pt{\hss $\bu$\hss}}%
    \graphtemp=.5ex\advance\graphtemp by 1.366in
    \rlap{\kern 3.517in\lower\graphtemp\hbox to 0pt{\hss $\bu$\hss}}%
    \graphtemp=.5ex\advance\graphtemp by 1.159in
    \rlap{\kern 3.517in\lower\graphtemp\hbox to 0pt{\hss $\bu$\hss}}%
    \graphtemp=.5ex\advance\graphtemp by 0.745in
    \rlap{\kern 3.517in\lower\graphtemp\hbox to 0pt{\hss $\bu$\hss}}%
    \graphtemp=.5ex\advance\graphtemp by 0.538in
    \rlap{\kern 3.517in\lower\graphtemp\hbox to 0pt{\hss $\bu$\hss}}%
    \graphtemp=.5ex\advance\graphtemp by 0.124in
    \rlap{\kern 3.517in\lower\graphtemp\hbox to 0pt{\hss $\bu$\hss}}%
    \graphtemp=.5ex\advance\graphtemp by 1.366in
    \rlap{\kern 3.931in\lower\graphtemp\hbox to 0pt{\hss $\bu$\hss}}%
    \graphtemp=.5ex\advance\graphtemp by 0.745in
    \rlap{\kern 3.931in\lower\graphtemp\hbox to 0pt{\hss $\bu$\hss}}%
    \graphtemp=.5ex\advance\graphtemp by 0.124in
    \rlap{\kern 3.931in\lower\graphtemp\hbox to 0pt{\hss $\bu$\hss}}%
    \special{pn 28}%
    \special{pa 2483 952}%
    \special{pa 3310 952}%
    \special{fp}%
    \special{pa 3310 952}%
    \special{pa 3517 745}%
    \special{fp}%
    \special{pa 3517 1366}%
    \special{pa 3310 1572}%
    \special{fp}%
    \special{pa 3310 1572}%
    \special{pa 2897 952}%
    \special{fp}%
    \special{pa 2897 952}%
    \special{pa 3310 331}%
    \special{fp}%
    \special{pa 3310 331}%
    \special{pa 3517 124}%
    \special{fp}%
    \special{pn 11}%
    \special{pa 3310 1572}%
    \special{pa 3517 1779}%
    \special{fp}%
    \special{pa 3517 1366}%
    \special{pa 3931 1366}%
    \special{fp}%
    \special{pa 3310 952}%
    \special{pa 3517 1159}%
    \special{fp}%
    \special{pa 3517 745}%
    \special{pa 3931 745}%
    \special{fp}%
    \special{pa 3310 331}%
    \special{pa 3517 538}%
    \special{fp}%
    \special{pa 3517 124}%
    \special{pa 3931 124}%
    \special{fp}%
    \graphtemp=.5ex\advance\graphtemp by 0.828in
    \rlap{\kern 2.483in\lower\graphtemp\hbox to 0pt{\hss $u$\hss}}%
    \graphtemp=.5ex\advance\graphtemp by 0.828in
    \rlap{\kern 2.897in\lower\graphtemp\hbox to 0pt{\hss $v$\hss}}%
    \graphtemp=.5ex\advance\graphtemp by 1.631in
    \rlap{\kern 3.252in\lower\graphtemp\hbox to 0pt{\hss $z_3$\hss}}%
    \graphtemp=.5ex\advance\graphtemp by 1.010in
    \rlap{\kern 3.252in\lower\graphtemp\hbox to 0pt{\hss $z_2$\hss}}%
    \graphtemp=.5ex\advance\graphtemp by 0.231in
    \rlap{\kern 3.252in\lower\graphtemp\hbox to 0pt{\hss $z_1$\hss}}%
    \graphtemp=.5ex\advance\graphtemp by 1.266in
    \rlap{\kern 3.459in\lower\graphtemp\hbox to 0pt{\hss $y_3$\hss}}%
    \graphtemp=.5ex\advance\graphtemp by 0.645in
    \rlap{\kern 3.459in\lower\graphtemp\hbox to 0pt{\hss $y_2$\hss}}%
    \graphtemp=.5ex\advance\graphtemp by 0.024in
    \rlap{\kern 3.459in\lower\graphtemp\hbox to 0pt{\hss $y_1$\hss}}%
    \special{pn 8}%
    \special{ar 3621 1779 103 103 0 6.28319}%
    \special{ar 3621 1159 103 103 0 6.28319}%
    \special{ar 3621 538 103 103 0 6.28319}%
    \special{ar 4034 1366 103 103 0 6.28319}%
    \special{ar 4034 745 103 103 0 6.28319}%
    \special{ar 4034 124 103 103 0 6.28319}%
    \graphtemp=.5ex\advance\graphtemp by 0.952in
    \rlap{\kern 4.552in\lower\graphtemp\hbox to 0pt{\hss $\bu$\hss}}%
    \graphtemp=.5ex\advance\graphtemp by 0.952in
    \rlap{\kern 4.552in\lower\graphtemp\hbox to 0pt{\hss $\bu$\hss}}%
    \graphtemp=.5ex\advance\graphtemp by 1.572in
    \rlap{\kern 4.966in\lower\graphtemp\hbox to 0pt{\hss $\bu$\hss}}%
    \graphtemp=.5ex\advance\graphtemp by 0.952in
    \rlap{\kern 4.966in\lower\graphtemp\hbox to 0pt{\hss $\bu$\hss}}%
    \graphtemp=.5ex\advance\graphtemp by 0.331in
    \rlap{\kern 4.966in\lower\graphtemp\hbox to 0pt{\hss $\bu$\hss}}%
    \graphtemp=.5ex\advance\graphtemp by 1.779in
    \rlap{\kern 5.172in\lower\graphtemp\hbox to 0pt{\hss $\bu$\hss}}%
    \graphtemp=.5ex\advance\graphtemp by 1.366in
    \rlap{\kern 5.172in\lower\graphtemp\hbox to 0pt{\hss $\bu$\hss}}%
    \graphtemp=.5ex\advance\graphtemp by 0.952in
    \rlap{\kern 5.379in\lower\graphtemp\hbox to 0pt{\hss $\bu$\hss}}%
    \graphtemp=.5ex\advance\graphtemp by 0.952in
    \rlap{\kern 5.793in\lower\graphtemp\hbox to 0pt{\hss $\bu$\hss}}%
    \graphtemp=.5ex\advance\graphtemp by 0.538in
    \rlap{\kern 5.172in\lower\graphtemp\hbox to 0pt{\hss $\bu$\hss}}%
    \graphtemp=.5ex\advance\graphtemp by 0.124in
    \rlap{\kern 5.172in\lower\graphtemp\hbox to 0pt{\hss $\bu$\hss}}%
    \graphtemp=.5ex\advance\graphtemp by 1.366in
    \rlap{\kern 5.586in\lower\graphtemp\hbox to 0pt{\hss $\bu$\hss}}%
    \graphtemp=.5ex\advance\graphtemp by 0.124in
    \rlap{\kern 5.586in\lower\graphtemp\hbox to 0pt{\hss $\bu$\hss}}%
    \special{pn 28}%
    \special{pa 4552 952}%
    \special{pa 5379 952}%
    \special{fp}%
    \special{pa 5172 1366}%
    \special{pa 4966 1572}%
    \special{fp}%
    \special{pa 4966 1572}%
    \special{pa 4552 952}%
    \special{fp}%
    \special{pa 4552 952}%
    \special{pa 4966 331}%
    \special{fp}%
    \special{pa 4966 331}%
    \special{pa 5172 124}%
    \special{fp}%
    \special{pn 11}%
    \special{pa 4966 1572}%
    \special{pa 5172 1779}%
    \special{fp}%
    \special{pa 5172 1366}%
    \special{pa 5586 1366}%
    \special{fp}%
    \special{pa 4966 331}%
    \special{pa 5172 538}%
    \special{fp}%
    \special{pa 5172 124}%
    \special{pa 5586 124}%
    \special{fp}%
    \special{pa 5379 952}%
    \special{pa 5793 952}%
    \special{fp}%
    \graphtemp=.5ex\advance\graphtemp by 0.828in
    \rlap{\kern 4.552in\lower\graphtemp\hbox to 0pt{\hss $v$\hss}}%
    \graphtemp=.5ex\advance\graphtemp by 0.828in
    \rlap{\kern 4.966in\lower\graphtemp\hbox to 0pt{\hss $u$\hss}}%
    \graphtemp=.5ex\advance\graphtemp by 0.828in
    \rlap{\kern 5.379in\lower\graphtemp\hbox to 0pt{\hss $u'$\hss}}%
    \graphtemp=.5ex\advance\graphtemp by 1.838in
    \rlap{\kern 5.073in\lower\graphtemp\hbox to 0pt{\hss $x'$\hss}}%
    \graphtemp=.5ex\advance\graphtemp by 0.596in
    \rlap{\kern 5.093in\lower\graphtemp\hbox to 0pt{\hss $x$\hss}}%
    \graphtemp=.5ex\advance\graphtemp by 1.631in
    \rlap{\kern 4.866in\lower\graphtemp\hbox to 0pt{\hss $z'$\hss}}%
    \graphtemp=.5ex\advance\graphtemp by 0.273in
    \rlap{\kern 4.866in\lower\graphtemp\hbox to 0pt{\hss $z$\hss}}%
    \graphtemp=.5ex\advance\graphtemp by 1.283in
    \rlap{\kern 5.131in\lower\graphtemp\hbox to 0pt{\hss $y'$\hss}}%
    \graphtemp=.5ex\advance\graphtemp by 0.041in
    \rlap{\kern 5.110in\lower\graphtemp\hbox to 0pt{\hss $y$\hss}}%
    \graphtemp=.5ex\advance\graphtemp by 0.041in
    \rlap{\kern 5.379in\lower\graphtemp\hbox to 0pt{\hss $a$\hss}}%
    \graphtemp=.5ex\advance\graphtemp by 0.476in
    \rlap{\kern 5.028in\lower\graphtemp\hbox to 0pt{\hss $b$\hss}}%
    \graphtemp=.5ex\advance\graphtemp by 0.828in
    \rlap{\kern 5.793in\lower\graphtemp\hbox to 0pt{\hss $u''$\hss}}%
    \special{pn 8}%
    \special{ar 5276 1779 103 103 0 6.28319}%
    \special{ar 5276 538 103 103 0 6.28319}%
    \special{ar 5690 1366 103 103 0 6.28319}%
    \special{ar 5690 124 103 103 0 6.28319}%
    \special{ar 5897 952 103 103 0 6.28319}%
    \hbox{\vrule depth1.883in width0pt height 0pt}%
    \kern 6.000in
  }%
}%
}
\caption{Cases {\bf K1,K2,K3}  for Lemma~\ref{83red3}\label{figK3}}
\end{figure}

{\bf Subcase 3:} {\it $v$ is a $\gamma_{3a}$-vertex.}
Let $u$ be the $\alpha$-neighbor of $v$ (with $N_G(u)=\{v,u'\}$).
By {\bf C}, $v$ does not have another $2$-neighbor.
Let $N_G(z)=\{v,x,y\}$ and $N_G(z')=\{v,x',y'\}$, with $d_G(y)=d_G(y')=2$.
By {\bf J}, each $\beta$-neighbor of $v$ is not a $\gamma$-vertex, which means
that the other neighbors of $y$ and $y'$ are $3^+$-vertices.

Let $G'=G-\Gamma_G(v)-\{uu',zy,z'y'\}$.  Let $a$ and $b$ be the weights under
$w'$ of the edges incident to $y$ and $z$ in $G'$, respectively.
Set $w(vz')=3$ to ensure satisfying $z'y'$, and choose $w(z'y')$
to satisfy the edges incident to $z'y'$ other than $vz'$.

If $a\le b$, then $zy$ is automatically satisfied.  Choose $w(zy)$
to satisfy $\Gamma_{G'}(y)$.  Now choose $w(vu)$ to satisfy $vz$ and $uu'$, and
then choose $w(vz)$ to satisfy $vz'$ and $zx$.  Finally, choose $uu'$ to
satisfy $vu$ and $u'u''$.

If $a>b$, then setting $w(zy)=1$ ensures satisfying the other edge at $y$
(since its other endpoint has degree at least $3$).  With $b\le2$ and
$w(vz')=3$, the edge $vz$ is automatically satisfied.  Now choose $w(vz)$
to satisfy the other edges at $z$, choose $w(vu)$ to satisfy $vz'$ and $uu'$,
and choose $w(uu')$ to satisfy its incident edges.
\end{proof}

\begin{thm}
Every graph $G$ with $\Mad(G)<\FR83$ has a proper $3$-weighting.
\end{thm}
\begin{proof}
It suffices to show that every configuration in the unavoidable set in
Lemma~\ref{83disch3} is shown to be $3$-reducible in Lemma~\ref{83red3}.  The
configurations are the same in the two lemmas except for {\bf H} and {\bf J}.

For {\bf H}, if $d(v)\in\{6,7\}$ and $v$ has a $1$-neighbor and four
$\gamma$-neighbors, then $p_1+2q\ge d(v)$ and $p_1+q>4$.  For {\bf J}, a
$4$-vertex $v$ with $p+q+r\ge 5$ must have an $\alpha$-neighbor.
If $v$ has another $\alpha$-neighbor, then {\bf J1} applies.
If $v$ has a $\gamma$-neighbor and another $2$-neighbor, then {\bf J2} applies.
Otherwise, all other neighbors are $\gamma$-neighbors (reducible by {\bf J3})
or all are $2$-neighbors (reducible by {\bf B}).
\end{proof}

Some of the $3$-reducible configurations in Lemma~\ref{83red3} are more general
than the configurations forced in Lemma~\ref{83disch3}.  Also, there are other
$3$-reducible configuration we have not used, such as (1) a $4$-vertex having
two $2$-neighbors and one $\gamma$-neighbor and (2) a more general version of
configuration {\bf H}.  This suggests that with more work this approach could
be pushed to prove the conclusion under a weaker restriction on $\Mad(G)$.

\end{document}